\documentclass[english, 11pt]{article}
\usepackage{xcolor}
\usepackage{verbatim}
\usepackage{float}
\usepackage{mathtools}
\usepackage{url}
\usepackage{amsmath}
\usepackage{amsthm}
\usepackage{amssymb}
\usepackage{booktabs}

\makeatletter

\usepackage{graphicx,psfrag,amsfonts,verbatim,fullpage,mathtools}
\usepackage[hidelinks]{hyperref}

\usepackage{xurl}

\usepackage{bbm}

\usepackage{dsfont}

\usepackage[framemethod=default]{mdframed}

\global\mdfdefinestyle{vert_line_shuvo}
{
linecolor=black,
linewidth=2pt,
topline=false,
nobreak=true,
bottomline=false,
rightline=false
}

\global\mdfdefinestyle{algorithm_box_shuvo}
{
linecolor=black,
linewidth=1pt,
innerleftmargin=1pt,
innertopmargin=5pt,
innerbottommargin=5pt,
innerrightmargin=1pt,
nobreak=true,
}

\usepackage{upgreek}

\usepackage[bbgreekl]{mathbbol}
\DeclareSymbolFontAlphabet{\mathbbm}{bbold}
\DeclareSymbolFontAlphabet{\mathbb}{AMSb}
\usepackage[bb=px]{mathalfa}
\usepackage{bm}

\PassOptionsToPackage{obeyFinal}{todonotes}
\usepackage{todonotes}

\usepackage{abstract}

\theoremstyle{plain}
\newtheorem{thm}{\protect\theoremname}
\theoremstyle{plain}
\newtheorem{cor}{\protect\corollaryname}
\theoremstyle{plain}
\newtheorem{lem}{\protect\lemmaname}
\theoremstyle{plain}
\newtheorem{assumption}{Assumption}
\theoremstyle{plain}
\newtheorem{remark}{Remark}
\theoremstyle{definition}

\theoremstyle{plain}

\theoremstyle{remark}

\theoremstyle{plain}

\usepackage{babel}
\providecommand{\lemmaname}{Lemma}
\providecommand{\propositionname}{Proposition}
\providecommand{\theoremname}{Theorem}
\providecommand{\corollaryname}{Corollary}
\providecommand{\definitionname}{Definition}

\usepackage{parskip}
\begingroup
\makeatletter
   \@for\theoremstyle:=definition,remark,plain\do{     \expandafter\g@addto@macro\csname th@\theoremstyle\endcsname{        \addtolength\thm@preskip\parskip
     }   }
\endgroup
\usepackage{parskip}

\usepackage{diagbox}

\definecolor{commentcolor}{RGB}{74,112,35}

\usepackage{titlesec}
\def\sectionfont{\sffamily\Large\bfseries\boldmath}
\def\subsectionfont{\sffamily\large\bfseries\boldmath}
\def\paragraphfont{\sffamily\normalsize\bfseries\boldmath}
\titleformat*{\section}{\sectionfont}
\titleformat*{\subsection}{\subsectionfont}
\titleformat*{\subsubsection}{\paragraphfont}
\titleformat*{\paragraph}{\paragraphfont}
\titleformat*{\subparagraph}{\paragraphfont}
\usepackage[small,labelfont={bf,sf}]{caption}
\usepackage{multirow}
\usepackage{enumitem}
\setlist{nolistsep}

\usepackage{booktabs}
\usepackage{adjustbox}

\usepackage{titling}
\pretitle{\begin{center}\LARGE \bfseries\sffamily}
\posttitle{\par\end{center}\vskip 2em}
\preauthor{\begin{center}
\large \lineskip 2em\begin{tabular}[t]{c}}
\postauthor{\end{tabular}\par\end{center}}
\predate{\begin{center}\large}
\postdate{\par\end{center}}

\usepackage{tikz}
\usetikzlibrary{fit, arrows, arrows.meta, shapes, positioning, shadows, matrix, calc}
\usepackage[]{subfig}

\RequirePackage{luatex85}
\usepackage{pgfplots}
\pgfplotsset{compat=newest}
\usepgfplotslibrary{groupplots}
\usepgfplotslibrary{polar}
\usepgfplotslibrary{smithchart}
\usepgfplotslibrary{statistics}
\usepgfplotslibrary{dateplot}
\allowdisplaybreaks

\usepackage{nicefrac}
\usepackage{cancel}

\definecolor{annotgray}{RGB}{128,128,128}

\newcommand{\norm}[1]{\left\lVert#1\right\rVert} 

\title{\textbf{
A Domain-Specific Harness for\\
End-to-End Automation of Optimization Research}}
\author{
  \small{\textbf{Heechang Kim}}\\
  \small{Seoul National University} \\
  \small{\texttt{heechang.kim@snu.ac.kr}} \\
  \and
  \small{\textbf{Ernest Ryu}} \\
  \small{UCLA} \\
  \small{\texttt{eryu@math.ucla.edu}} \\
  \and
  \small{\textbf{Shuvomoy Das Gupta}}\\
  \small{Rice University} \\
  \small{\texttt{sd158@rice.edu}} \\
}
\date{}

\makeatother

\usepackage{babel}

\begin{document}

\maketitle

\begin{abstract}
We present \textsf{AutoOPT}, a domain-specific harness for end-to-end automation of optimization research. \textsf{AutoOPT} organizes the discovery of optimal first-order methods into four stages: numerical design through the BnB-PEP methodology; symbolic discovery of the analytic description and a convergence proof through frontier large language models (LLMs); formal verification in the Lean 4 proof assistant; and human interpretation and write-up. We demonstrate the framework on two case studies, each of independent interest. The first, lemniscate acceleration, is a new accelerated gradient method for minimizing the gradient norm of a smooth convex function: after $N$ gradient steps it reduces the squared gradient norm at the optimal $O(1/N^{4})$ rate, with a constant governed by the lemniscate constant $\varpi$, a classical elliptic-integral constant. The second is the analytic description of ITEM-f, a method previously known only numerically: for $L$-smooth, $\mu$-strongly convex minimization it contracts the function-value gap at an accelerated linear rate with a per-step factor $(1-\sqrt{\mu/L})^{2}$. The convergence theorems of both case studies are formalized and machine-checked in Lean 4.

\end{abstract}

\section*{Preface}

In an era where frontier large language models (LLMs) provide remarkable mathematical intelligence, what role remains for human researchers? In this work, we (human experts) augment the general-purpose intelligence of LLMs with a technical domain-specific numerical methodology, and automate optimization research end-to-end, at a level not yet possible with LLMs alone. We hope this shows the community one way domain experts can productively contribute in an AI-driven research workflow: by augmenting the LLM's general-purpose intelligence with a domain-specific harness derived from their expertise.

\section{Introduction}\label{sec:introduction}

Since the pioneering work of Nesterov on accelerated gradient methods
\cite{nesterov1983method} and of Nemirovsky and Yudin on information-based
complexity \cite{nemirovsky1983problem}, the design of efficient and optimal
first-order methods has been a central pursuit in the study of large-scale
optimization. The performance estimation programming (PEP)  framework 
\cite{drori2014performance, taylor2017smooth, TaylorHendrickxGlineur2017Composite} turned this pursuit
into a computer-assisted discipline by formulating the worst-case performance computation of a first-order method as a convex semidefinite program (SDP). Since its inception, PEP has produced many notable algorithms including the optimized gradient method (OGM)
\cite{kim2016optimized, drori2017exact}, its gradient-norm counterpart OGM-G
\cite{kim2021optimizing}, the information-theoretic exact method (ITEM)
\cite{taylor2021optimal, drori2022oracle}. Designing an optimal method through the PEP framework reduces to optimizing over the space of methods and branch-and-bound performance estimation programming (BnB-PEP) \cite{dasgupta2022BnBPEP} 
formulates this design process as a nonconvex quadratically constrained quadratic
program (QCQP) and solves it to  global optimality using a spatial branch-and-bound algorithm. 

Yet, this computer-assisted methodology can be heavily labor-intensive: every stage from numerically computing the method to coming up with the analytical form of the algorithm and its convergence analysis demands deep domain expertise, and the rate of discovery is set by human effort. While recent open-source packages such as PESTO
\cite{TaylorHendrickxGlineur2017PESTO} and PEPit
\cite{GoujaudEtAl2022PEPit} can help by solving the worst-case SDP numerically, they come at the cost of a rigid, language-specific syntax and a PEP
derivation encoded in low-level library code that can be difficult for
newcomers to parse or modify. Meanwhile, large language models (LLMs) and LLM-driven agents have begun to
produce genuine algorithmic and mathematical discoveries, which we review
in Section~\ref{subsec:prior-work}. The natural next step is to combine LLMs with these classical
computer-assisted tools to further accelerate research in optimization.

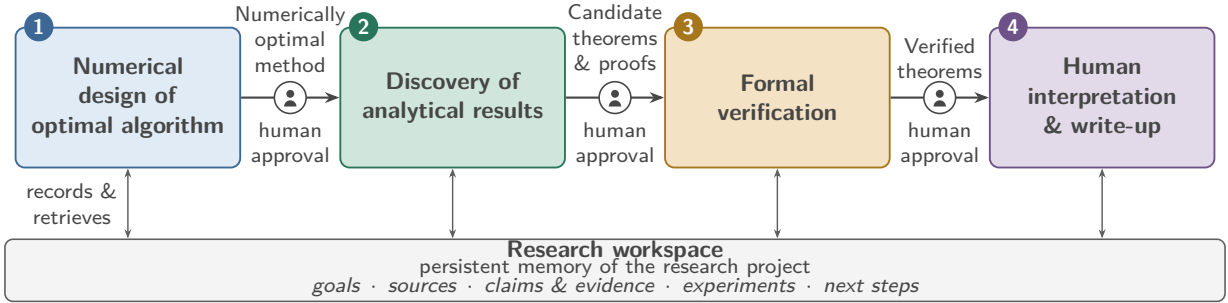
\begin{figure}[h]
\centering
\definecolor{aoptink}{RGB}{70,70,70}
\definecolor{aoptbluefill}{RGB}{224,235,245}
\definecolor{aoptblueline}{RGB}{61,103,148}
\definecolor{aopttealfill}{RGB}{208,230,220}
\definecolor{aopttealline}{RGB}{40,117,91}
\definecolor{aoptamberfill}{RGB}{243,226,196}
\definecolor{aoptamberline}{RGB}{166,118,29}
\definecolor{aoptvioletfill}{RGB}{226,218,233}
\definecolor{aoptvioletline}{RGB}{108,82,135}
\definecolor{aoptledgerfill}{RGB}{244,244,244}
\definecolor{aoptledgerline}{RGB}{95,95,95}
\begin{tikzpicture}[
  font=\sffamily,
  stage/.style={draw, rounded corners=3.5pt, line width=0.7pt, align=center,
    text width=2.62cm, minimum height=1.85cm, inner sep=5pt, text=aoptink,
    font=\sffamily\footnotesize},
  chip/.style={circle, text=white, font=\sffamily\bfseries\scriptsize,
    inner sep=1.2pt, minimum size=11pt},
  flow/.style={-{Stealth[length=5.5pt]}, line width=0.8pt, draw=aoptink},
  rw/.style={{Stealth[length=3.5pt]}-{Stealth[length=3.5pt]},
    line width=0.55pt, draw=aoptledgerline},
  lab/.style={font=\sffamily\scriptsize, text=aoptink, align=center,
    inner sep=1.5pt},
]
\node[stage, draw=aoptblueline, fill=aoptbluefill] (s1) at (0,0)
  {\textbf{Numerical design of}\\\textbf{optimal algorithm}};
\node[stage, draw=aopttealline, fill=aopttealfill, right=1.3cm of s1] (s2)
  {\textbf{Discovery of \\ analytical results}};
\node[stage, draw=aoptamberline, fill=aoptamberfill, right=1.3cm of s2] (s3)
  {\textbf{Formal}\\\textbf{verification}};
\node[stage, draw=aoptvioletline, fill=aoptvioletfill, right=1.3cm of s3] (s4)
  {\textbf{Human}\\\textbf{interpretation}\\\textbf{\& write-up}};
\node[chip, fill=aoptblueline]   at ($(s1.north west)+(0.32,0)$) {1};
\node[chip, fill=aopttealline]   at ($(s2.north west)+(0.32,0)$) {2};
\node[chip, fill=aoptamberline]  at ($(s3.north west)+(0.32,0)$) {3};
\node[chip, fill=aoptvioletline] at ($(s4.north west)+(0.32,0)$) {4};
\draw[flow] (s1) -- (s2)
  node[lab, midway, yshift=22pt] {Numerically\\optimal\\method};
\draw[flow] (s2) -- (s3)
  node[lab, midway, yshift=22pt] {Candidate\\theorems\\\& proofs};
\draw[flow] (s3) -- (s4)
  node[lab, midway, yshift=15pt] {Verified\\theorems};
\foreach \a/\b in {s1/s2, s2/s3, s3/s4}{
  \begin{scope}[shift={($(\a.east)!0.5!(\b.west)$)}]
    \draw[fill=white, draw=aoptink, line width=0.6pt] circle[radius=0.21];
    \fill[aoptink] (0,0.05) circle[radius=0.052];
    \fill[aoptink] (-0.082,-0.118)
      arc[start angle=180, end angle=0, radius=0.082] -- cycle;
    \node[lab, anchor=north] at (0,-0.26) {human\\approval};
  \end{scope}
}
\node[draw=aoptledgerline, fill=aoptledgerfill, rounded corners=3pt,
      line width=0.6pt, inner ysep=5pt, align=center,
      fit={([yshift=-2.05cm]s1.west) ([yshift=-2.52cm]s4.east)}]
  (ledger) {};
\node[align=center, text=aoptink, font=\sffamily\footnotesize\bfseries]
  at ($(ledger.north)+(0,-0.16)$)
  {Research workspace};
\node[align=center, text=aoptink, font=\sffamily\scriptsize]
  at ($(ledger.north)+(0,-0.39)$)
  {persistent memory of the research project};
\node[align=center, text=aoptink, font=\sffamily\scriptsize\itshape]
  at ($(ledger.north)+(0,-0.63)$)
  {goals \,$\cdot$\, sources \,$\cdot$\, claims \& evidence
   \,$\cdot$\, experiments \,$\cdot$\, next steps};
\foreach \s in {s1, s2, s3, s4}{
  \draw[rw] (\s.south) -- (\s.south |- ledger.north);
}
\node[lab, anchor=east]
  at ($(s1.south)!0.5!(s1.south |- ledger.north)+(-0.08,0)$)
  {records \&\\retrieves};
\end{tikzpicture}\caption{ The researcher operates
\textsf{AutoOPT} by conversing with an LLM agent 
in a research workspace in four stages. 
Stage 1 numerically designs an
optimal first-order method for the problem class that the researcher
specifies. Stage 2 consults a frontier LLM to propose an analytic form of
the numerically designed method along with candidate convergence
proofs. Stage 3 formally verifies the candidate theorems and proofs in the
Lean proof assistant, ensuring every proof step is correct
and explicit. In Stage 4, human authors interpret the
verified results and write them up. 
} 
\label{fig:autoopt-pipeline}
\end{figure}

\subsection{Contribution.}
This work presents \textsf{AutoOPT}, a domain-specific harness that end-to-end automates
optimization research by combining the computer-assisted PEP
methodology with LLMs and Lean formalization;
Figure~\ref{fig:autoopt-pipeline} illustrates the pipeline. 

We demonstrate \textsf{AutoOPT} on two case studies, each producing a
result of independent interest.
The first case study, \emph{lemniscate acceleration}
(Section~\ref{sec:lemniscate}), is a new accelerated gradient method for
minimizing the gradient norm of an $L$-smooth convex function $f$ with a
minimizer $x_{\star}$. From any initial point $x_{0}$, its final
iterate $x_{N}$ satisfies
\begin{equation*}
\norm{\nabla f(x_{N})}^{2}
\le \frac{\varpi^{4}L^{2}\norm{x_{0}-x_{\star}}^{2}}{(N+1)^{4}},
\end{equation*}
where $\varpi\approx2.622$ is the lemniscate constant. This rate is
optimal in its $N$-dependence
\cite{nemirovsky1991optimality,nemirovsky1992information}, and its
constant improves upon the previous best guarantee, obtained by
concatenating OGM and OGM-G, by a factor of about $1.35$
(Section~\ref{sec:lemniscate}).
To the best of our knowledge, this is the first time the lemniscate
constant and its associated elliptic functions have been connected to
optimization theory.

The second case study, \emph{analytic ITEM-f}, (Section~\ref{sec:itemf}) is the optimal method for minimizing an
$L$-smooth, $\mu$-strongly convex function $f$ with minimizer $x_{\star}$ together with the accelerated
linear rate
\[
f(x_{N})-f(x_{\star})\leq 4 \left(1-\sqrt{\frac{\mu}{L}}\right)^{2N}\left(f(x_{0})-f(x_{\star})\right),
\]
whose per-step contraction factor $\left(1-\sqrt{\mu/L}\right)^{2}$ matches the
asymptotics of the complexity lower bound \cite{drori2022oracle}. ITEM-f was previously available only through numerical values
\cite[Appendix E]{taylor2021optimal}.
We formalize and machine-check the main convergence theorems of both lemniscate acceleration and ITEM-f in Lean 4 with mathlib \cite{demoura2021lean,mathlib2020lean}.

More broadly, this work serves as a case study on how domain experts can contribute productively to an AI-driven research workflow by designing and utilizing a domain-specific harness.

\subsection{Prior work}\label{subsec:prior-work}

\paragraph{AI-driven discovery and domain-specific harnesses.}
A rapidly growing body of work uses LLMs and related search methods as
discovery engines. Evolutionary symbolic search produced the Lion optimizer
\cite{chen2023symbolic}; AlphaTensor discovered faster matrix multiplication
algorithms through reinforcement learning \cite{fawzi2022discovering}; and
FunSearch and AlphaEvolve couple LLM code generation with automatic
evaluation to find new mathematical constructions and algorithmic
improvements
\cite{romera-paredesMathematicalDiscoveriesProgram2024,novikovAlphaEvolveCodingAgent2025}.
At the level of whole research programs, the AI Scientist pursues end-to-end
automation of machine-learning research
\cite{luAIScientistFully2024,lu2026towards}; Aletheia reports research-level
problem solving in mathematics \cite{fengAutonomousMathematicsResearch2026};
an internal version of OpenAI's unreleased Astra model
solved ten long-open problems in mathematics and theoretical computer
science, with each argument formalized in Lean
\cite{openaiTenAdvances2026};
and, with the Lean 4 proof assistant and its
mathematical library mathlib \cite{demoura2021lean,mathlib2020lean} now
mature enough to formalize research-level mathematics, AI-driven formal
proof search has resolved open problems
\cite{tsoukalas2026advancing}. In the natural sciences,
AlphaFold and RoseTTAFold transformed protein-structure prediction
\cite{jumper2021highly,baek2021accurate}, and the Virtual Lab coordinates a
team of scientific agents that designed and experimentally validated new
nanobodies \cite{swanson2025virtual}.

In the modern parlance of agentic AI, a \emph{harness} is the scaffolding
built around an LLM, including the tools, prompts, workflows, and
verification loops. We define a \emph{domain-specific harness} to be a
harness designed with expert domain knowledge for the purpose of augmenting
the general-purpose intelligence of LLMs with the specialized methodology of
a technical domain. Despite the large body of work using LLMs to conduct
mathematical research, most of it operates the models through simple chat
interfaces, and instances of domain-specific harnesses are not yet common.
Although their authors do not use the term, AlphaEvolve
\cite{novikovAlphaEvolveCodingAgent2025} and Aletheia
\cite{fengAutonomousMathematicsResearch2026} can be regarded as
domain-specific harnesses: AlphaEvolve targets scientific and algorithmic
discovery across a range of computational problems, whereas Aletheia targets
mathematical research. Both systems, however, remain broad in scope and were
built in industry research labs with very large inference compute budgets. In
contrast, \textsf{AutoOPT} specializes further, to the domain of first-order
optimization research, and demonstrates this research principle at an
academic scale of compute and token usage.

\paragraph{Computer-assisted design of optimization methods.}
The performance estimation methodology originates with Drori and Teboulle
\cite{drori2014performance}, who formulated the worst-case performance of a
fixed-step first-order method as an optimization problem and bounded it
through an SDP relaxation; Taylor, Hendrickx, and Glineur proved the
relaxation exact through convex interpolation \cite{taylor2017smooth} and then extended the framework to the composite setup in \cite{TaylorHendrickxGlineur2017Composite}. The
methodology has since driven a sequence of discoveries. Drori and Teboulle
numerically constructed OGM \cite{drori2014performance}, Kim and Fessler
found its analytic description \cite{kim2016optimized}, and Drori proved its
exact optimality through a matching lower bound \cite{drori2017exact}. Kim
and Fessler constructed OGM-G, with the best known rate for reducing the
gradient norm of smooth convex functions \cite{kim2021optimizing}. Taylor and
Drori constructed ITEM for reducing the distance from the set of optimal solutions, whose exact optimality for smooth strongly convex
minimization follows from a matching lower bound
\cite{taylor2021optimal,drori2022oracle}. BnB-PEP \cite{dasgupta2022BnBPEP}
unified the design side of this program by solving the underlying nonconvex
QCQPs to certifiable global optimality, constructing optimal methods in
setups where the prior convexity-based reformulations do not apply. Its
methodology produced OptISTA, an optimal method for composite minimization
\cite{jang2025optista}. Its numerical certificates also helped pave the way
for the provably faster long-step gradient descent of Grimmer
\cite{grimmer2024provably} and the silver stepsize schedule of Altschuler
and Parrilo \cite{altschuler2025acceleration1,altschuler2025acceleration2},
a line of work surveyed in \cite{altschuler2026stepsize}. In every
one of these successes, the path from a numerically optimal solution to a
theorem that includes analytic coefficients, a dual certificate, and a written proof was
traversed by hand. \textsf{AutoOPT} automates precisely this path while
keeping the human researcher in control of every decision boundary.

\paragraph{AI-assisted optimization theory.}
Very recently, frontier models have also begun proving research-level results in optimization theory:
an AI agent writing directly in Lean discovered an improved last-iterate
rate for anchored gradient descent ascent for minimax optimization \cite{surina2026improved},
GPT~5.6 Sol Pro closed a query-complexity gap open since 1996 in
derivative-free convex optimization, with the central
lower bound subsequently formalized in Lean \cite{kerger2026closing}, and
ChatGPT~5.5 produced the first convergence proof for Bregman
Douglas--Rachford splitting as a matrix scaling algorithm
\cite{ma2026convergence}. In all three works the core workflow was that a human prompted a frontier model to solve a known open question, the model supplied the mathematics, and correctness was
secured afterward, by the authors' own checking or by a Lean formalization
 provided for that one result.

Two recent works \cite{upadhyaya2026optimal, suh2026peppy} combine PEP with LLM consultation for optimization research. In \cite{upadhyaya2026optimal}, using ChatGPT and Codex for exploratory PEP computations, the authors discovered an optimal
proximal gradient method for minimizing the sum of a smooth and a
nonsmooth convex function, with the final Lyapunov proofs
derived by the authors. Peppy \cite{suh2026peppy} is an AI-assisted
workflow for tight convergence analysis of known algorithms: given a
problem class, a \emph{known} algorithm, a performance measure, and an initial
condition, it implements and solves the associated PEPs through the
deterministic PEPFlow package, extracts fixed-horizon proof certificates, assembles
Lyapunov partial sums, and simplifies them into closed-form analytic
proofs, with symbolic checks along the way.

\subsection{Organization}\label{subsec:organization}

This paper is organized as follows.
In Section~\ref{sec:background-agentic}, we review conversational chatbots,
agentic research systems, coding and persistent agents, reusable agent skills,
and the research practices on which \textsf{AutoOPT} builds.
In Section~\ref{sec:methodology}, we describe
\textsf{AutoOPT} methodology consisting of numerical design, symbolic discovery, formal
verification, and human interpretation.
In Section~\ref{subsec:notation}, we present lemniscate
acceleration and analytic ITEM-f, respectively, together with their
convergence analyses, continuous-time interpretations, \textsf{AutoOPT}
design histories, and Lean formalizations.
In Section~\ref{sec:conclusion}, we discuss the applicability and limitations
of \textsf{AutoOPT} and the broader role of domain-specific harnesses in
AI-assisted research.
Appendices~\ref{sec:appendix-lemniscate}
and~\ref{sec:appendix-itemf} collect the deferred proofs, PEP formulations,
and technical details for the two case studies.

\subsection{Computational setup}

We open-source the portable agent skills and Lean verification projects associated with
\textsf{AutoOPT} at the link:
\begin{center}
    \url{https://github.com/Shuvomoy/AutoOPT}
\end{center}

Unless otherwise specified, we conducted the experiments from a laptop
running macOS Tahoe with an Apple M5 Max processor and 128 GB of memory. The
agentic experiments used GPT-5.5
\cite{openaiGPT55} and GPT-5.6 Sol \cite{openaiGPT56} through the USD
200-per-month ChatGPT Pro tier \cite{openaiChatGPTPro}. No internal or
unreleased model was used for the agentic experiments in this paper. The Lean verification
projects use Lean 4.32.0
\cite{demoura2021lean} and mathlib 4.32.0 \cite{mathlib2020lean}.

Within Stage 1 of \textsf{AutoOPT}, we formulated the numerical design
problems in \texttt{JuMP} 1.31.1 \cite{Lubin2023}, a domain-specific modeling
language for mathematical optimization embedded in \texttt{Julia} 1.12.6
\cite{bezanson2017julia}. The reported BnB-PEP computations used
\texttt{MOSEK} 11.2.2 \cite{mosek2026} as the primary solver for the 
fixed-method SDPs and \texttt{KNITRO} 16.0.0 \cite{byrd2006knitro} and \texttt{Ipopt} 3.14.19 \cite{wachter2006ipopt}  as the
primary solvers for the local QCQPs.
\texttt{Gurobi} 13.0.2 \cite{gurobi2026} was used to perform
spatial branch-and-bound.
\texttt{MOSEK} and \texttt{Gurobi} provide no-cost
academic licenses \cite{mosekAcademicLicenses,gurobiAcademicLicenses}.
\texttt{Ipopt} is open source under the Eclipse Public License 2.0
\cite{ipoptLicense}, while \texttt{KNITRO} is commercial and offers academic
and teaching programs \cite{knitroPrograms}.

\section{Background on agentic research}\label{sec:background-agentic}

In this section, we briefly review recent agentic research practice,
with an emphasis on what each class of tool can and cannot responsibly
do in a mathematical research workflow.

\paragraph{From chatbots to agents.}
\emph{Conversational} chatbots such as ChatGPT, Claude, and Gemini were
the first widely used form of LLM assistance in
research. A researcher can ask for explanations, brainstorming, proof
debugging, literature orientation, or alternative formulations of a
claim. This is extremely valuable, but it has a structural
limitation: a chat agent does not own a repository, run the tests that
would falsify its suggestions, or maintain a reproducible artifact
trail. It can offer advice, but it cannot execute \emph{research} over a
long time horizon.

\paragraph{Agentic research systems.}
Recent work on using LLMs for research has moved in a more
\emph{agentic} direction, closing part of the propose-test-revise
loop. Program-search systems such as FunSearch and AlphaEvolve use LLMs
to generate candidate programs and evaluate them automatically, yielding
new constructions and algorithmic improvements
\cite{romera-paredesMathematicalDiscoveriesProgram2024,novikovAlphaEvolveCodingAgent2025}.
Broader scientific-agent frameworks, such as the AI Scientist, attempt
larger portions of the open-ended research cycle
\cite{luAIScientistFully2024}, and in mathematics, systems such as
Aletheia point toward autonomous or semi-autonomous research-level
problem solving \cite{fengAutonomousMathematicsResearch2026}.

\paragraph{Coding agents.}
For individual academic research and mathematical software development,
the most practical setup beyond the conversational chatbot is the coding
agent, such as Codex, Claude Code, and OpenCode
\cite{openaiCodexCLI,anthropicClaudeCode,openCode}. A coding agent can
operate on an active repository: it can read files, edit source code, run
shell commands, inspect compiler or solver output, and iterate after
errors. The important shift from the conversational chatbot is that the
agent can act on artifacts the researcher will later inspect.

Persistent personal research agents extend this idea from a single coding session
toward a more persistent research workflow. OpenClaw can be
viewed as a personal AI assistant that runs on the user's devices and
replies on messaging surfaces such as WhatsApp, Telegram, Slack,
Mattermost, and Discord \cite{openclawDocs}. Hermes Agent is a
self-improving agent with persistent memory, generated workflow packages,
scheduled automations, and support for similar messaging surfaces
\cite{nousHermesAgent}. These agents can support research progress by preserving state across repositories, files, commands, logs, memories, scheduled tasks, and other inspectable artifacts; their ability to maintain context across interactions and serve as standing research assistants is therefore highly valuable.

\paragraph{Agent skills.}
In the context of agentic LLMs, a \emph{skill} is a reusable package of instructions and supporting material that an agent can
load for a specialized task. In the open Agent Skills format, the
minimal object is a folder containing a \texttt{SKILL.md} file with
metadata and instructions; the folder may also contain scripts, references,
templates, and other resources \cite{agentSkillsStandard}. OpenAI
Codex and Claude Code document skill mechanisms based on this principle
\cite{openaiCodexSkills,anthropicClaudeCodeSkills}. For an optimization
workflow, the point of a skill is to move recurring guidance out of
a transient prompt and into an inspectable artifact.

\paragraph{Agentic research practice.}
Across these tool classes, effective agentic research consistently
depends on persistent instructions, inspectable
 reports and artifacts, staged evaluation, explicit verification, and
 intermittent human steering \cite{zimmer2026AgenticResearcher}.
Mathematical research tightens each requirement. A claimed algorithm
must have a precise statement; a computational run must record solver
status, tolerances, and enough environment information to be reproduced;
a proof attempt must be labeled theorem, conjecture, checked derivation,
or failed approach; and a formal-verification claim requires actual
checker output. \textsf{AutoOPT} instantiates these requirements for
optimization research: each stage of the pipeline is packaged as a skill
that any skill-loading agent can execute, every stage is approved or revised
by a human, and the research state lives in a versioned repository. We
describe the framework in Section~\ref{sec:methodology}.

\section{Methodology}\label{sec:methodology} \label{subsec:autoopt-workflow}

\textsf{AutoOPT} is an agentic framework for the end-to-end design,
convergence analysis, and formal verification of optimal first-order
methods. The framework has four stages, illustrated in
Figure~\ref{fig:autoopt-pipeline} and detailed in the following
subsections. Throughout this section, we describe \textsf{AutoOPT} at a
high level, focusing on how to use the methodology; readers interested in
implementation details can refer to the appendix and the accompanying skills.

\paragraph{Operating \textsf{AutoOPT}.}
\textsf{AutoOPT} is model- and agent-agnostic:  a user can run
it from any skill-loading agent, e.g., Codex, Claude Code, or OpenCode, and
can use any commercial or open-source LLM supported by the
agent. The user operates
\textsf{AutoOPT} by conversing with the agent.

At the start, \textsf{AutoOPT} invokes the agent skill
\textsf{research-repo-manager}, which creates a research workspace in a
directory of the user's choice. The workspace serves as persistent memory
for the research project, keeping every source, claim, experiment,
artifact, reproducibility note, deferred goal, and potential next step
explicit and inspectable.
The workspace overcomes a key limitation of single-session LLM
workflows: as a conversation grows, the LLM's context window becomes
crowded, compressed, or incomplete, making it harder to reliably recover
earlier decisions, assumptions, and evidence. By storing the research
state in versioned repository files rather than only in the chat
transcript, \textsf{research-repo-manager} lets a later session reload the
current goals, findings, artifacts, and next action directly from the
workspace.

\subsection{Stage 1: Numerical design of optimal first-order methods via
\textsf{bnb-pep-skill}}\label{subsec:autoopt-stage-numerical}

\paragraph{The design problem.}\label{subsec:autoopt-math-background}
The research question \textsf{AutoOPT} addresses is finding the fastest
gradient-based method for a given problem class: minimizing a function
$f\colon\mathbb{R}^{d}\to\mathbb{R}$ from a known \emph{function
class} $\mathcal{F}$, such as smooth convex functions, smooth strongly
convex functions, weakly convex functions, or a nonsmooth class. Here,
``fastest'' means the best worst-case guarantee with respect to a chosen
performance measure $\mathcal{E}$. For example, the goal may be to design an
$N$-step algorithm producing a final point $x_{N}$ such that
$\mathcal{E}=f(x_{N})-\inf f$ or $\mathcal{E}=\|\nabla f(x_{N})\|^{2}$ is
guaranteed to be small, subject to an initial condition $\mathcal{C}\leq0$
normalizing the starting point, such as
$\mathcal{C}=\|x_{0}-x_{\star}\|^{2}-R^{2}$ for a given $R>0$.

\textsf{AutoOPT} searches over the class of \emph{fixed-step first-order
methods} (FSFOMs): methods that form each iterate from the starting point
$x_{0}\in\mathbb{R}^{d}$ and the first-order information observed so far via
\[
x_{i}=x_{i-1}-\sum_{j=0}^{i-1}h_{i,j}\,f^\prime(x_{j}),\qquad i=1,\ldots,N,
\]
where $f^\prime(x_j) \in \partial f(x_j)$ is a subgradient of $f$ at $x_j$ and $\{h_{i,j}\}_{0\le j<i\le N}$ is a lower triangular array of
stepsizes; lower triangularity expresses that each update depends only on
the first-order information observed so far. The stepsizes $\{h_{i,j}\}_{0\le j<i\le N}$ parameterize the FSFOM and
are the main decision variables we optimize over; we write $\mathcal{M}_{N}$ for the class of all $N$-step FSFOMs. We write $M=\{h_{i,j}\}_{0\le j<i\le N}\in \mathcal{M}_N$ to denote an individual FSFOM.

Formally, the design problem becomes a minimax problem of the form
\begin{equation}\label{eq:autoopt-outer-design}
\mathcal{R}^{\star}(\mathcal{M}_{N})\;\triangleq\;\min_{M\in\mathcal{M}_{N}}\;
\underbrace{\max\left\{ \mathcal{E}\;:\;f\in\mathcal{F},\;\mathcal{C}\leq0,\;
x_{1},\ldots,x_{N}\text{ generated by }M=\{h_{i,j}\}_{0\le j<i\le N}\right\}}_{\mathcal{R}(M),\;\text{the worst-case performance of }M}.
\end{equation}

\paragraph{From function classes to finitely many inequalities.}
At first sight, the minimax problem above may seem intractable as the inner
maximization runs over all functions in $\mathcal{F}$, an
infinite-dimensional class.
The performance estimation methodology
\cite{drori2014performance,taylor2017smooth} removes this obstacle in two
steps. First, for a fixed method $M=\{h_{i,j}\}_{0\le j<i\le N}$, the inner problem depends on $f$ only
through its values and first-order information at the iterates
$x_{0},x_{1},\ldots,x_{N}$ and at a stationary point (minimizer for convex setups) $x_{\star}$, adjoined to
encode stationarity (optimality for convex setups) through $0\in\partial f(x_{\star})$; collecting these
indices in $I^{\star}_{N}\triangleq\{0,1,\ldots,N,\star\}$ and writing
$g_{i}\in\partial f(x_{i})$ and $f_{i}\triangleq f(x_{i})$, the trajectory
data $\{(x_{i},g_{i},f_{i})\}_{i\in I^{\star}_{N}}$ carry everything the
inner problem needs to know about $f$. Second, the function classes the
methodology covers are \emph{quadratically representable}
\cite{dasgupta2025BnBPEPTutorial,dasgupta2022BnBPEP}: membership in
$\mathcal{F}$ is described by inequalities of the form
\begin{equation}\label{eq:autoopt-quad-class}
\begin{aligned}
c_{0}f(y)\geq{} & c_{1}f(x)+c_{2}\norm{x}^{2}+c_{3}\norm{y}^{2}+c_{4}\norm{u}^{2}+c_{5}\norm{v}^{2}+c_{6}\langle x,y\rangle\\
 & +c_{7}\langle x,u\rangle+c_{8}\langle x,v\rangle+c_{9}\langle y,u\rangle+c_{10}\langle y,v\rangle+c_{11}\langle u,v\rangle
\end{aligned}
\end{equation}
for all $x,y\in\mathbb{R}^{d}$, $u\in\partial f(x)$, and $v\in\partial
f(y)$, where $\partial f$ denotes the subdifferential, reducing to
$\partial f(x) = \{\nabla f (x)\}$ for the differentiable classes, and $c_{0},\ldots,c_{11}$
are fixed coefficients determined by the class parameters; smooth convex,
smooth strongly convex, and weakly convex functions are all of this form.

Technically, replacing the constraint \(f\in\mathcal{F}\) with the inequalities \eqref{eq:autoopt-quad-class} imposed only on the finitely many pairs in \(I^{\star}_{N}\) is, a priori, a relaxation. Establishing that this finite system is lossless requires a separate theorem for each function class, and such characterizations are known as \emph{interpolation conditions}. We refer the reader to prior work for a detailed discussion of this issue; see, for example, \cite{taylor2017smooth, RubbensEtAl2024Interpolation, RubbensEtAl2025Constructive}.

\paragraph{Gram lifting: a convex SDP for fixed stepsizes.}
It is convenient to reparameterize the FSFOM as
$x_{i}=x_{0}-\sum_{j=0}^{i-1}\alpha_{i,j}g_{j}$, where
$\alpha_{i,j}\triangleq\sum_{k=j+1}^{i}h_{k,j}$; this triangular map is
invertible, so $\alpha$ is the design variable. Collect the relevant
trajectory vectors as columns of a matrix $P$ and the sampled function
values in a row vector $F$, and form the Gram matrix
$G\triangleq P^{\top}P\succeq0$. Since the iterates are affine in
$\alpha$, the sampled inequalities \eqref{eq:autoopt-quad-class}, the
performance measure, and the initial condition are affine in $(F,G)$,
with Gram-side coefficient matrices whose entries have degree at most
two in $\alpha$. Writing $\mathbf{tr}$ for the trace, the inner
worst-case problem becomes
\begin{equation}
\mathcal{R}(M)=\left(\begin{array}{l}
\textrm{maximize}\quad F\mathbf{u}_{\mathcal{E}}+\mathbf{tr}\,GQ_{\mathcal{E}}(\alpha)\\
\textrm{subject to}\\
F\mathbf{u}_{i,j}+\mathbf{tr}\,GQ_{i,j}(\alpha)\leq0,\quad i,j\in I^{\star}_{N}:i\neq j,\quad\mathbin{{\color{annotgray}\rhd}}{\color{annotgray}\,\textsf{dual var.}\,\lambda_{i,j}\geq0}\\
-G\preceq0,\quad\mathbin{{\color{annotgray}\rhd}}{\color{annotgray}\,\textsf{dual var.}\,Z\succeq0}\\
F\mathbf{u}_{\mathcal{C}}+\mathbf{tr}\,GQ_{\mathcal{C}}(\alpha)\leq r,\quad\mathbin{{\color{annotgray}\rhd}}{\color{annotgray}\,\textsf{dual var.}\,\nu\geq0}\\
\mathop{\textbf{rank}}G\leq d,
\end{array}\right),\label{eq:autoopt-inner-sdp}
\end{equation}

where the vectors
$\mathbf{u}_{\mathcal{E}},\mathbf{u}_{\mathcal{C}},\mathbf{u}_{i,j}$,
the symmetric matrices
$Q_{\mathcal{E}}(\alpha),Q_{\mathcal{C}}(\alpha),Q_{i,j}(\alpha)$,
and the scalar $r$ are assembled mechanically from
\eqref{eq:autoopt-quad-class}, $\mathcal{E}$, and $\mathcal{C}$; for
instance, the initial condition
$\mathcal{C}=\norm{x_{0}-x_{\star}}^{2}-R^{2}$ has
$\mathbf{u}_{\mathcal{C}}=0$ and $r=R^{2}$. Under the large-scale assumption that $d$ is at least the
number of columns of $P$, the constraint $\mathop{\textbf{rank}}G\leq d$ is
vacuous. Dropping it makes \eqref{eq:autoopt-inner-sdp} a convex
semidefinite program (SDP) in $(F,G)$, free of the problem dimension
$d$. This mechanical assembly is exactly what
\textsf{bnb-pep-skill} automates: it fixes the exact contents of $P$ and
$F$ for the instance at hand, derives the corresponding vectors
$\mathbf{u}$ and matrices $Q(\alpha)$ symbolically, and records the
resulting Gram dimension together with the matching assumption on $d$. Appendix~\ref{sec:lemniscate-pep-design} carries out this assembly
explicitly for lemniscate acceleration.

\paragraph{Dualization: from worst-case values to proofs.}
SDP duality turns the inner maximization
\eqref{eq:autoopt-inner-sdp} into a minimization,
\begin{equation}\label{eq:autoopt-dual-sdp}
\overline{\mathcal{R}}(M)=\left(\begin{array}{l}
\textrm{minimize}\quad\nu r\\
\textrm{subject to}\\
\mathbf{u}_{\mathcal{E}}-\nu\mathbf{u}_{\mathcal{C}}-\sum_{i,j\in I^{\star}_{N}:i\neq j}\lambda_{i,j}\mathbf{u}_{i,j}=0,\\
\nu Q_{\mathcal{C}}(\alpha)-Q_{\mathcal{E}}(\alpha)+\sum_{i,j\in I^{\star}_{N}:i\neq j}\lambda_{i,j}Q_{i,j}(\alpha)=Z,\\
Z\succeq0,\\
\nu\geq0,\;\lambda_{i,j}\geq0,\quad i,j\in I^{\star}_{N}:i\neq j
\end{array}\right),
\end{equation}
whose decision variables
$\lambda=\{\lambda_{i,j}\}_{i,j\in I^{\star}_{N}:i\neq j}$, $\nu$, and
the positive-semidefinite slack matrix $Z$ we call the \emph{inner-dual variables}. By
weak duality, $\mathcal{R}(M)\leq\overline{\mathcal{R}}(M)$, and this
inequality carries the interpretive weight of the methodology: any
feasible inner-dual point is a certificate, an explicit nonnegative
aggregation of the sampled inequalities proving that
$\mathcal{E}\leq\nu r$ for every function in the class satisfying the
initial condition, that is, a machine-checkable convergence proof. Strong
duality holds under mild conditions which implies $\mathcal{R}(M)=\overline{\mathcal{R}}(M)$  and is assumed in this paper.
Each derivation that \textsf{bnb-pep-skill} writes states
this dual explicitly, with the sign conventions of its multipliers fixed,
and records whether the resulting bound invokes strong duality or only
the weak direction.
The certificate view is what the later stages of \textsf{AutoOPT} use internally: Stage~2
(Section~\ref{subsec:autoopt-stage-symbolic}) fits analytic expressions to
the numerically designed
$(\alpha^{\star},\lambda^{\star},\nu^{\star},Z^{\star})$ 
and then converts the fitted multipliers into the Lyapunov function based convergence proofs, and Stage~3
(Section~\ref{subsec:autoopt-stage-lean}) formalizes those Lyapunov
proofs in Lean.

\paragraph{Cholesky elimination and the BnB-PEP QCQP.}
Substituting the dual \eqref{eq:autoopt-dual-sdp} for the
inner maximization in \eqref{eq:autoopt-outer-design} collapses the
minimax into a single minimization, jointly over the stepsizes and the
inner-dual variables. 
In other words, the single minimization problem jointly optimizes over the algorithm and its convergence proof. 

However, this joint optimization is non-convex. The BnB-PEP methodology \cite{dasgupta2022BnBPEP}
eliminates the semidefinite constraint $Z\succeq 0$ through the parameterization $Z=VV^{\top}$ and, by introducing auxiliary variables to represent bilinear or quadratic terms when necessary, formulates the numerical optimization problem as a nonconvex quadratically constrained quadratic program (QCQP)
\begin{equation}\label{eq:autoopt-qcqp}
\mathcal{R}^{\star}(\mathcal{M}_{N})=\left(\begin{array}{l}
\textrm{minimize}\quad\nu r\\
\textrm{subject to}\\
\mathbf{u}_{\mathcal{E}}-\nu\mathbf{u}_{\mathcal{C}}-\sum_{i,j\in I^{\star}_{N}:i\neq j}\lambda_{i,j}\mathbf{u}_{i,j}=0,\\
\nu Q_{\mathcal{C}}(\alpha)-Q_{\mathcal{E}}(\alpha)+\sum_{i,j\in I^{\star}_{N}:i\neq j}\lambda_{i,j}Q_{i,j}(\alpha)=Z,\\
V\text{ is lower triangular with nonnegative diagonals},\\
Z=VV^{\top},\\
\nu\geq0,\;\lambda_{i,j}\geq0,\quad i,j\in I^{\star}_{N}:i\neq j
\end{array}\right),
\end{equation}
with decision variables $\{\alpha_{i,j}\}_{0\leq j<i\leq N}$, $\lambda$,
$\nu$, $Z$, and $V$. 
Any feasible point to this QCQP corresponds to an algorithm and its convergence proof, while an optimal point gives the best algorithm along with its convergence proof.

\paragraph{The three-stage solution strategy.}
Despite its nonconvexity, \eqref{eq:autoopt-qcqp} can be solved to
certifiable global optimality, up to prescribed tolerances, by the
three-stage BnB-PEP strategy of \cite{dasgupta2022BnBPEP}: a fixed-method
dual SDP supplies a feasible point, an interior-point solver refines it
to a local optimum, and a customized spatial branch-and-bound solver
performs the global certification. The final stage is made practical by
the QCQP-specific bounds and cuts developed in that work. Accordingly,
\textsf{bnb-pep-skill} runs the first two stages by default and invokes
the global stage only at the user's explicit request. The BnB-PEP methodology can be laborious to implement and debug and the manual coding is cumbersome and error-prone and requires substantial PEP expertise. This implementation barrier has previously limited its use.

\paragraph{Automation with \textsf{bnb-pep-skill}.}
The methodology described above, including several technical details omitted from the exposition, can be tedious to formulate and implement by hand. Our \textsf{bnb-pep-skill} automates this process through an executable and auditable protocol.

A run of the \textsf{bnb-pep-skill} opens as a structured interview that
assumes no PEP expertise on the user's part and produces a fully specified
instance: the function class with its inequalities, the method update
equations, the performance measure, the initial condition, the constants,
and the horizon $N$. Whenever a required item is missing or ambiguous, the
skill asks a targeted question rather than guessing, and it accepts the
instance only when another researcher could write down every constraint
and identify every decision variable without making further choices.

\paragraph{Code generation and guardrails.}
Once the user approves the mathematics, \textsf{bnb-pep-skill}  writes code. The
generated \texttt{Julia}/\texttt{JuMP} \cite{Lubin2023} program builds the primal and dual SDPs of the instance
together with the design problem \eqref{eq:autoopt-qcqp} and numerically solves the non-convex QCQP.

This protocol makes concrete the division of labor
behind the domain-specific harness of Section~\ref{sec:introduction}: the
LLM agent interviews the user,
instantiating the derivation, and writing code, while the harness supplies
the accumulated domain expertise, from exactness theorems for
interpolation conditions to warm-start numerics, that a general-purpose
model cannot be relied upon to reproduce unaided. The same division of
labor recurs in Stages~2 and~3.

\paragraph{The agentic workflow.}
In the first stage,
\textsf{AutoOPT} invokes \textsf{bnb-pep-skill} and interviews the user
about the problem setup for which an optimal first-order method is
sought. The interview assumes no PEP expertise on the user's part.
From the agreed setup, the skill writes out the complete
derivation, instantiating
\eqref{eq:autoopt-quad-class}--\eqref{eq:autoopt-qcqp} for the user's
instance, and stops for the user's approval of the mathematics before any
code is written.

Once the user approves the derivation, the skill generates and runs
\texttt{Julia}/\texttt{JuMP} code to obtain candidate stepsizes and
associated inner-dual certificates. It stores the derivation, the code,
the numerical solution, the solver status, and the logs in the research
workspace.
Figure~\ref{fig:bnb-pep-stage} summarizes the workflow; the complete
derivations produced this way for our two case studies
appear, lightly edited for presentation, in
Appendices~\ref{sec:lemniscate-pep-design}
and~\ref{sec:itemf-pep-design}.

\begin{figure}[H]
\centering
\begin{adjustbox}{max width=\textwidth, center}
\definecolor{bnbink}{RGB}{70,70,70}
\definecolor{bnbbluefill}{RGB}{224,235,245}
\definecolor{bnbbluebadgefill}{RGB}{246,250,253}
\definecolor{bnbblueline}{RGB}{61,103,148}
\definecolor{bnbtealfill}{RGB}{208,230,220}
\definecolor{bnbtealline}{RGB}{40,117,91}
\definecolor{bnbamberfill}{RGB}{243,226,196}
\definecolor{bnbamberline}{RGB}{166,118,29}
\definecolor{bnbvioletfill}{RGB}{226,218,233}
\definecolor{bnbvioletline}{RGB}{108,82,135}
\definecolor{bnbgrayfill}{RGB}{244,244,244}
\definecolor{bnbgrayline}{RGB}{95,95,95}
\begin{tikzpicture}[
  font=\sffamily,
  stage/.style={draw, rounded corners=4pt, line width=0.72pt, align=center,
    text width=2.00cm, minimum height=1.62cm, inner sep=4pt, text=bnbink,
    font=\sffamily\small},
  flow/.style={-{Stealth[length=5.7pt]}, line width=0.78pt, draw=bnbink},
  lab/.style={font=\sffamily\scriptsize, text=bnbink, align=center,
    inner sep=1.0pt, fill=white},
  skilltag/.style={draw, rounded corners=2.4pt, line width=0.58pt,
    font=\sffamily\scriptsize, text=bnbink, align=center,
    inner xsep=4.2pt, inner ysep=1.8pt},
]
\node[stage, draw=bnbblueline, fill=bnbbluefill] (q) at (0,0)
  {\textbf{Problem}\\\textbf{description}\\\textbf{by user}};
\node[stage, draw=bnbtealline, fill=bnbtealfill, right=1.28cm of q] (deriv)
  {\textbf{BnB-PEP}\\\textbf{derivation}};
\node[stage, draw=bnbamberline, fill=bnbamberfill, right=1.28cm of deriv] (julia)
   {\textbf{Julia}\\\textbf{code}\\\textbf{generation}};
\node[stage, draw=bnbvioletline, fill=bnbvioletfill, right=1.28cm of julia] (run)
  {\textbf{Run}\\\textbf{Julia}\\\textbf{code}};
\node[stage, draw=bnbgrayline, fill=bnbgrayfill, right=1.28cm of run] (out)
  {\textbf{Store}\\\textbf{derivation \&}\\\textbf{numerical solution}};

\node[skilltag, draw=bnbtealline, fill=bnbbluebadgefill, anchor=south]
  at ($(q.north)!0.5!(out.north)+(0,0.18)$) {\texttt{bnb-pep-skill}};

\draw[flow] (q) -- (deriv);
\draw[flow] (deriv) -- (julia);
\draw[flow] (julia) -- (run);
\draw[flow] (run) -- (out);

\foreach \a/\b/\labeltop/\labelbottom in {q/deriv/approve/edit,
    deriv/julia/approve/edit, julia/run/approve/edit}{
  \begin{scope}[shift={($(\a.east)!0.5!(\b.west)$)}]
    \draw[fill=white, draw=bnbink, line width=0.6pt] circle[radius=0.21];
    \fill[bnbink] (0,0.05) circle[radius=0.052];
    \fill[bnbink] (-0.082,-0.118)
      arc[start angle=180, end angle=0, radius=0.082] -- cycle;
    \node[lab, anchor=north] at (0,-0.30)
      {\labeltop/\\[-0.5pt]\labelbottom};
  \end{scope}
}
\end{tikzpicture}\end{adjustbox}
\caption{
Stage 1 of \textsf{AutoOPT}.
Numerical design of the optimal method via \textsf{bnb-pep-skill}.}
\label{fig:bnb-pep-stage}
\end{figure}
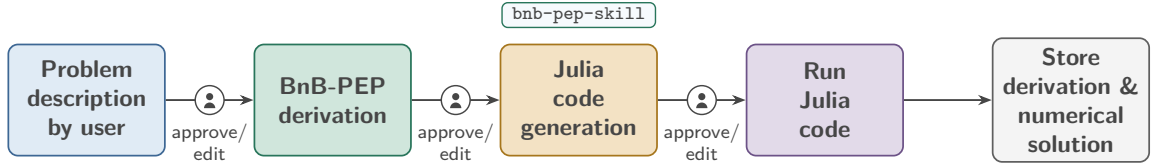

\subsection{Stage 2: Symbolic discovery of analytic descriptions and
convergence proofs via
LLMs}\label{subsec:autoopt-stage-symbolic}

\paragraph{Numerical to symbolic discovery.}
Once the numerically optimal stepsizes and the BnB-PEP mathematical
derivation are available from Stage 1, \textsf{AutoOPT} invokes the skill
\textsf{frontier-llm-consult} to discover analytic descriptions of the
numerically designed methods and candidate convergence proofs.
Figure~\ref{fig:symbolic-fitting-stage} summarizes the stage. 
The skill \textsf{frontier-llm-consult} supports
two routes. The external consultation route,
\textsf{chatgpt-pro-session}, runs a persistent ChatGPT Pro session
in a web browser and is the default for ordinary Stage 2 work; the
same route also covers bounded one-shot reviews, each conducted as a
single approved turn in a fresh session. For the consultation
route, the skill
assembles a deliberately small context bundle containing a
self-contained prompt, the numerical findings from the previous stage,
and the relevant mathematical background and derivations. The skill
then presents the bundle for user approval; the user may approve it as
is or edit it first, and only after approval is the bundle shared with
the frontier model. The persistent session enables symbolic discovery
from numerical findings through several rounds of pattern search,
algebraic refinement, and follow-up questions from the user. The native
route, \textsf{solve-with-highest-reasoning}, runs only on explicit
user invocation and is not an external consultation: the host agent
itself conducts a long native campaign on the problem, using the
strongest currently available Codex GPT model at its highest supported
reasoning setting, with direct repository exploration, local tools,
and adversarial verification, over at least eight wall-clock hours.
The user approves the repository scope, no problem material is sent to
any external model, and the full evidence trail is archived in the
research workspace. For the consultation route, the user can edit the
underlying \texttt{SKILL.md} file to substitute a different frontier
model.

Whichever route is selected, discovery proceeds in
three substeps. The frontier
model first identifies the analytic form of the numerically designed
method, that is, an analytic parametrization of the
numerically optimal FSFOM stepsizes from Stage 1, together with an
analytic feasible point of the inner dual
\eqref{eq:autoopt-dual-sdp}. By the certificate view of
Section~\ref{subsec:autoopt-stage-numerical}, feasibility of this point
already proves the worst-case bound, so the first substep yields a
candidate convergence proof by dual feasibility.
In the second substep, the model recasts the
method from this raw stepsize parametrization into a simple,
memory-efficient algorithmic form, recursive or momentum-based, which is
the form in which this paper states each method.
In the third substep, the
frontier model derives a second argument from the
same certificate: reading coefficients off the certificate's multipliers,
it constructs a Lyapunov sequence of potentials that is nonincreasing
along the iterates and whose endpoint values give the claimed
bound, with the potentials stated directly on the recast
form of the second substep. Both arguments are candidate proofs of
the same theorem. This paper presents the Lyapunov proofs,
the form that a classical optimization audience reads most naturally; the
dual-feasibility certificates remain intermediate artifacts of the
research workspace, and it is their multipliers that the printed proofs
are built from.

Figure~\ref{fig:stage2-prompt} makes the workflow concrete for the
lemniscate acceleration case study of Section~\ref{sec:lemniscate}. It shows an
illustrative Stage-2 prompt sequence, consultation-shaped because both
published runs used the external \textsf{chatgpt-pro-session} route: one
initial consultation prompt requesting the first substep, followed by
two short turns in the same session requesting the remaining two. Only
the author-written task text is printed; the consultation package that
wraps it and the approve-before-send gate are those described above.

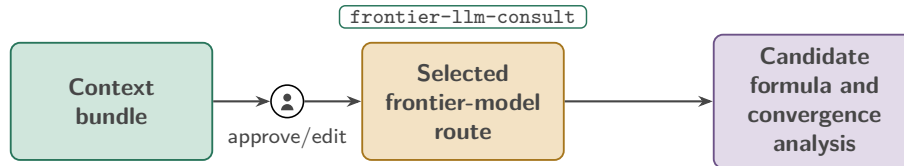
\begin{figure}[H]
\centering
\definecolor{symink}{RGB}{70,70,70}
\definecolor{symtealfill}{RGB}{208,230,220}
\definecolor{symtealline}{RGB}{40,117,91}
\definecolor{symamberfill}{RGB}{243,226,196}
\definecolor{symamberline}{RGB}{166,118,29}
\definecolor{symvioletfill}{RGB}{226,218,233}
\definecolor{symvioletline}{RGB}{108,82,135}
\begin{tikzpicture}[
  font=\sffamily,
  stage/.style={draw, rounded corners=3.5pt, line width=0.7pt, align=center,
    text width=2.35cm, minimum height=1.55cm, inner sep=4.5pt, text=symink,
    font=\sffamily\footnotesize},
  gate/.style={draw, circle, line width=0.6pt, fill=white, minimum size=0.43cm,
    inner sep=0pt},
  flow/.style={-{Stealth[length=5.5pt]}, line width=0.8pt, draw=symink},
  lab/.style={font=\sffamily\scriptsize, text=symink, align=center,
    inner sep=1.5pt},
  skilltag/.style={draw, rounded corners=2.2pt, line width=0.55pt,
    font=\sffamily\scriptsize, text=symink, align=center,
    inner xsep=4pt, inner ysep=1.7pt},
]
\node[stage, draw=symtealline, fill=symtealfill] (bundle) at (0,0)
  {\textbf{Context}\\\textbf{bundle}};
\node[stage, draw=symamberline, fill=symamberfill] (consult) at (4.65,0)
  {\textbf{Selected}\\\textbf{frontier-model}\\\textbf{route}};
\node[stage, draw=symvioletline, fill=symvioletfill] (candidate) at (9.3,0)
  {\textbf{Candidate}\\\textbf{formula and}\\\textbf{convergence}\\\textbf{analysis}};
\node[gate] (gate) at ($(bundle.east)!0.5!(consult.west)$) {};

\node[skilltag, draw=symtealline, fill=white, anchor=south]
  at ($(bundle.north)!0.5!(candidate.north)+(0,0.12)$)
  {\texttt{frontier-llm-consult}};
\node[lab, anchor=north] at ($(gate.south)+(0,-0.06)$) {approve/edit};
\fill[symink] ($(gate.center)+(0,0.05)$) circle[radius=0.052];
\fill[symink] ($(gate.center)+(-0.082,-0.118)$)
  arc[start angle=180, end angle=0, radius=0.082] -- cycle;

\draw[flow] (bundle.east) -- (gate.west);
\draw[flow] (gate.east) -- (consult.west);
\draw[flow] (consult.east) -- (candidate.west);
\end{tikzpicture}\caption{Stage 2 of \textsf{AutoOPT}. Symbolic discovery of analytic descriptions of the numerically designed
methods and their candidate convergence proofs via \textsf{frontier-llm-consult}.}
\label{fig:symbolic-fitting-stage}
\end{figure}

\begin{figure}[tp]
\centering
\begin{mdframed}[linewidth=0.7pt, roundcorner=3.5pt, innerleftmargin=8pt,
  innerrightmargin=8pt, innertopmargin=6pt, innerbottommargin=6pt]

\footnotesize\ttfamily
Attached are the Stage 1 outputs of \textsf{AutoOPT} for the design of an
$N$-step FSFOM that minimizes the worst-case squared gradient norm
$\norm{\nabla f(x_{N})}^{2}$ over $L$-smooth convex functions with
$\norm{x_{0}-x_{\star}}\leq R$: the BnB-PEP derivation
(\texttt{derivation.md}), the \textup{\texttt{Julia}}/\textup{\texttt{JuMP}}
code that solved it, and a CSV sweep
(\verb|stepsize_certificate_sweep.csv|) of the numerically optimal
stepsizes $h$ and inner-dual certificates $(\lambda,\nu,Z)$ for the
horizons \textup{\texttt{N = 1,...,25}} with first 5  being globally optimal via 
spatial branch-and-bound and the last 20 being locally optimal. 

From these data, propose a candidate analytic parametrization of the
stepsizes $h$ valid for every horizon $N$, together with a candidate
analytic feasible point $(\lambda,\nu,Z)$ of the inner dual SDP stated in
the derivation. Feasibility of that point already certifies the worst-case
bound, so also state the candidate rate it implies in closed form. Every
proposed formula is a candidate, not a proof; report which entries of the
sweep you checked it against and to what precision. Please ask me clarifying 
questions if anything is not clear to you.

\medskip\hrule\medskip

Your parametrization stores the full lower-triangular stepsize array $h$.
Recast the method into an equivalent memory-efficient form, recursive or
momentum-based, that reproduces the same iterates $x_{0},\ldots,x_{N}$
while carrying only a constant number of vectors, and give closed-form
expressions for the coefficients of the recast update. Check the recast
against the $h$-form on the sweep horizons and report the agreement.

\medskip\hrule\medskip

Reading the coefficients off the multipliers $(\lambda,\nu)$ of your
candidate certificate, construct a candidate Lyapunov proof of the same
bound: a sequence of potentials $\mathcal V_{0},\ldots,\mathcal V_{N}$,
stated directly on the recast form of your previous reply, that is
nonincreasing along the iterates and whose endpoint values yield the
claimed bound. Express each one-step decrement as a nonnegative
combination of the interpolation inequalities from the derivation, and
flag any step you cannot close. 

\end{mdframed}
\caption{An illustrative Stage-2 prompt sequence for the
lemniscate case study: one initial prompt followed by two
follow-up prompts corresponding to the three
substeps of Section~\ref{subsec:autoopt-stage-symbolic}.
}
\label{fig:stage2-prompt}
\end{figure}

\paragraph{Preparation for machine verification.}
At the end of this stage, \textsf{AutoOPT} saves the outputs from
the frontier LLMs in the research workspace so that the reasoning
trace is inspectable by the user. \textsf{AutoOPT}
marks the frontier LLM output as a ``candidate'' analytical solution:
it may give the researcher something concrete to test, simplify, formalize,
or prove, but it still needs independent validation, which is
supplied, for the results advanced to the paper, by the
Lean verification stage described next.

\subsection{Stage 3: Formal verification of candidate theorems via
\textsf{lean-verify}}\label{subsec:autoopt-stage-lean}

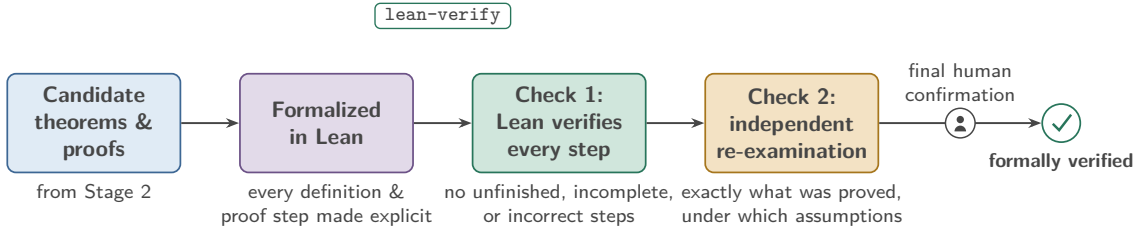
\begin{figure}[H]
\centering
\begin{adjustbox}{max width=\textwidth, center}
\definecolor{lvink}{RGB}{70,70,70}
\definecolor{lvbluefill}{RGB}{224,235,245}
\definecolor{lvblueline}{RGB}{61,103,148}
\definecolor{lvtealfill}{RGB}{208,230,220}
\definecolor{lvtealline}{RGB}{40,117,91}
\definecolor{lvamberfill}{RGB}{243,226,196}
\definecolor{lvamberline}{RGB}{166,118,29}
\definecolor{lvvioletfill}{RGB}{226,218,233}
\definecolor{lvvioletline}{RGB}{108,82,135}
\definecolor{lvrosefill}{RGB}{248,228,224}
\definecolor{lvroseline}{RGB}{162,92,82}
\definecolor{lvgrayfill}{RGB}{244,244,244}
\definecolor{lvgrayline}{RGB}{95,95,95}
\begin{tikzpicture}[
  font=\sffamily,
  stage/.style={draw, rounded corners=3.5pt, line width=0.7pt, align=center,
    text width=2.16cm, minimum height=1.40cm, inner sep=4.2pt, text=lvink,
    font=\sffamily\footnotesize},
  flow/.style={-{Stealth[length=5.3pt]}, line width=0.78pt, draw=lvink},
  lab/.style={font=\sffamily\scriptsize, text=lvink, align=center,
    inner sep=1.4pt, fill=white},
  sub/.style={font=\sffamily\scriptsize, text=lvink, align=center,
    inner sep=1.4pt},
  skilltag/.style={draw, rounded corners=2.2pt, line width=0.55pt,
    font=\sffamily\scriptsize, text=lvink, align=center,
    inner xsep=4pt, inner ysep=1.7pt},
]
\node[stage, draw=lvblueline, fill=lvbluefill] (candidate) at (0,0)
  {\textbf{Candidate}\\\textbf{theorems \&}\\\textbf{proofs}};
\node[stage, draw=lvvioletline, fill=lvvioletfill, right=0.82cm of candidate] (formalize)
  {\textbf{Formalized}\\\textbf{in Lean}};
\node[stage, draw=lvtealline, fill=lvtealfill, right=0.82cm of formalize] (check1)
  {\textbf{Check 1:}\\\textbf{Lean verifies}\\\textbf{every step}};
\node[stage, draw=lvamberline, fill=lvamberfill, right=0.82cm of check1] (check2)
  {\textbf{Check 2:}\\\textbf{independent}\\\textbf{re-examination}};
\node[draw=lvtealline, circle, fill=white, line width=0.7pt,
  minimum size=0.54cm, inner sep=0pt, right=2.30cm of check2] (verified) {};
\draw[lvtealline, line width=1.0pt]
  ($(verified.center)+(-0.12,-0.01)$) -- ($(verified.center)+(-0.03,-0.10)$) --
  ($(verified.center)+(0.14,0.12)$);
\node[lab, anchor=north, fill=white] at ($(verified.south)+(0,-0.08)$)
  {\textbf{formally verified}};

\node[sub, anchor=north] at ($(candidate.south)+(0,-0.10)$)
  {from Stage 2};
\node[sub, anchor=north] at ($(formalize.south)+(0,-0.10)$)
  {every definition \&\\proof step made explicit};
\node[sub, anchor=north] at ($(check1.south)+(0,-0.10)$)
  {no unfinished, incomplete,\\or incorrect steps};
\node[sub, anchor=north] at ($(check2.south)+(0,-0.10)$)
  {exactly what was proved,\\under which assumptions};

\node[skilltag, draw=lvtealline, fill=white, anchor=south]
  at ($(formalize.north)!0.5!(check1.north)+(0,0.62)$) {\texttt{lean-verify}};

\draw[flow] (candidate) -- (formalize);
\draw[flow] (formalize) -- (check1);
\draw[flow] (check1) -- (check2);
\draw[flow] (check2) -- (verified);

\begin{scope}[shift={($(check2.east)!0.5!(verified.west)$)}]
  \draw[fill=white, draw=lvink, line width=0.6pt] circle[radius=0.21];
  \fill[lvink] (0,0.05) circle[radius=0.052];
  \fill[lvink] (-0.082,-0.118)
    arc[start angle=180, end angle=0, radius=0.082] -- cycle;
  \node[lab, anchor=south] at (0,0.26) {final human\\confirmation};
\end{scope}
\end{tikzpicture}\end{adjustbox}
\caption{Formal verification via \textsf{lean-verify}: candidate theorems
and proofs from Stage 2 are formalized in Lean, checked by Lean itself, and
re-examined by an independent verifier; the result is recorded as formally
verified only after the human authors confirm that the formal theorems
faithfully express the paper's claims.}
\label{fig:lean-verification-stage}
\end{figure}

\paragraph{From candidate to formalization.}
Stage 3 of \textsf{AutoOPT} is the formal verification stage. Here,
\textsf{AutoOPT} invokes the skill \textsf{lean-verify}, which uses
machine verification in the open-source proof assistant Lean to
determine whether the chain of logic in the candidate theorems and
proposed proofs discovered in Stage 2 is correct. During machine
verification, Lean records every definition, theorem statement, and
proof precisely, examines the logic of every step of a proof, and
accepts a proof only when each step is correct and explicit.

The \textsf{lean-verify} skill first formalizes in Lean the candidate theorems together with the proofs selected
for presentation in the paper, which for our case studies are the
Lyapunov proofs derived in Stage 2, and reconstructs the proof steps
there in detail, so that any gap, hidden assumption, or algebraic slip
in the candidate argument is exposed rather than granted on trust. A
proof is recorded as \emph{formally verified} only after two
successive independent checks followed by a final human confirmation. The first
check is Lean's own internal verification, which confirms that the
accepted proof contains no unfinished, incomplete, or incorrect steps.
The second independent check is through a separate verifier called the
comparator that reexamines the finished formalization and confirms exactly
which statements have been proved and which assumptions the proofs
were permitted to use, so that what has been machine verified is
pinned down without ambiguity. Once the two checks are complete, the
human authors confirm that the formal theorem statements in Lean
faithfully express the corresponding theorem statements in the paper,
and only then record the result as formally verified.
Figure~\ref{fig:lean-verification-stage} summarizes the workflow of
the skill \textsf{lean-verify}.

An illustrative Stage-3 prompt, the file structure of
each resulting Lean project, and per-project summary statistics are
presented with the case studies, in Sections~\ref{sec:lemniscate-lean}
and~\ref{sec:itemf-lean}.

\subsection{Stage 4: Human interpretation and
write-up}\label{subsec:autoopt-stage-human}

The goal of this paper is to provide a new methodology for automating what used to be a highly time-consuming and technical component of optimization research. As researchers, we value accelerating progress and solving harder and more interesting problems.

However, we hope that this pursuit of automation does not cause us to lose sight of the fact that mathematical research is a human endeavor and that the value of research is realized when it is conveyed and used by other researchers. Therefore, the final part of the pipeline is for the human authors to interpret the results, frame why they are of interest to the community, and present insights that go beyond the mere statement of the result.

In what follows, we briefly offer our personal views on the best practices for humans using such methodologies.

\paragraph{Publication restraint.}
In this era of AI accelerating research, proofs of novel results are no longer scarce. We therefore believe authors should now ask whether a novel and correct result is sufficiently interesting to the research community, and exercise judgment and parsimony by forgoing publication of results that do not meet this new higher bar.

\paragraph{Distilling generalizable insights.}
While optimization is an applied field, worst-case complexity guarantees are theoretical analyses whose value lies in expanding our collective understanding of algorithm design and eventually informing the optimization practice. Merely presenting a new algorithm with a verifiable but opaque analysis delivers little of this value, especially when AI assistance means the human authors themselves lack strong intuition about the result.

It is therefore essential that authors extract compact, generalizable insights from the formal discovery and present them in a form that is maximally useful to human researchers. For example, a machine-discovered optimization algorithm may first arrive in a dense stepsize-matrix representation. In our pipeline, Stage 2 already recasts such a method into an equivalent memory-efficient, recursive, or momentum-based formulation (Section~\ref{subsec:autoopt-stage-symbolic}); what remains a matter of human judgment is which of the available forms serves readers best. The authors should make that choice deliberately, potentially with further AI assistance, and present the formulation that a human reader can most readily implement, interpret, and analyze.

\paragraph{A calibrated standard of evidence.}
In the final presentation, the authors should clearly assign each claim an evidentiary status (human-verified, formally verified, LLM-verified, numerical evidence, or conjecture), so that no intermediate object or computational observation is presented as stronger evidence than it actually provides.

We also emphasize that human authors must exercise their own due diligence in verification. Even a Lean formalization may fail to establish the intended result: the formal theorem statement may encode a claim weaker than the paper's natural-language statement, or the proof may remain incomplete because it contains \texttt{sorry} placeholders. Even with thorough harnesses, AI systems still make mistakes, and the human authors are ultimately accountable for the work they publish.

\section{New algorithms discovered, analyzed, and verified using AutoOPT}
\label{subsec:notation}
In this section, we present two new first-order algorithms found using \textsf{AutoOPT}: \emph{lemniscate acceleration} and \emph{analytic ITEM-f}. We present these algorithms both as algorithmic advances in their own right and as demonstrations of the strength of \textsf{AutoOPT}.

Note that the order of discovery via \textsf{AutoOPT} differs from the order of presentation of those results in this paper. For each method, \textsf{AutoOPT} first solved the BnB-PEP
design problem numerically, next identified the method's analytic form
together with an analytic dual certificate of its worst-case guarantee,
then recast the method from the raw stepsize
parametrization of the design problem into the memory-efficient form in
which this paper states it, and finally converted the certificate's
multipliers into a Lyapunov convergence proof for the presented form. The paper presents each method in the order a reader
expects: algorithm, convergence theorem, and proof. The printed proofs
are the Lyapunov proofs from Stage 2, with Lean formalization discussed
in Sections~\ref{sec:lemniscate-lean} and~\ref{sec:itemf-lean}; the
subsections titled ``\textsf{AutoOPT} design'' recount how each method
was found, and Appendices~\ref{sec:lemniscate-pep-design}
and~\ref{sec:itemf-pep-design} reproduce the Stage-1 PEP formulations.

We use the following notation and notions in the rest of the paper. Write $\mathbb{R}^{d}$ for the underlying Euclidean space and we
write $\left\langle \cdot,\cdot\right\rangle $ and $\|\cdot\|$ to
denote the standard inner product and norm on $\mathbb{R}^{d}$. Write
$\mathbb{R}^{m\times n}$ for the set of $m\times n$ matrices, $\mathbb{S}^{n}$
for the set of $n\times n$ symmetric matrices, and $\mathbb{S}^{n}_{+}$
for the set of $n\times n$ positive-semidefinite matrices. We use
the standard notation $e_{i}\in\mathbb{R}^{d}$ for the unit vector
having a single 1 as its $i$-th component. Write $\left(\cdot\odot\cdot\right)\colon\mathbb{R}^{d}\times\mathbb{R}^{d}\to\mathbb{R}^{d\times d}$
to denote the symmetric outer product, that is
$x\odot y=\left(xy^{\intercal}+yx^{\intercal}\right)/2$ for every $x,y\in\mathbb{R}^{d}$, and $\delta_{i,j}$
denotes the Kronecker delta, equal to $1$ if $i=j$ and $0$ otherwise. We
write $\mathbf{tr}$ for the trace of a square matrix.
For $0\leq\mu<L<\infty$, we write $\mathcal{F}_{\mu,L}$ for the class of
$L$-smooth $\mu$-strongly convex functions
$f\colon\mathbb{R}^{d}\to\mathbb{R}$, that is, differentiable functions for
which $\nabla f$ is $L$-Lipschitz and $f-(\mu/2)\norm{\cdot}^{2}$ is convex;
the case $\mu=0$ gives the class $\mathcal{F}_{0,L}$ of $L$-smooth convex
functions. Every function $f\in\mathcal{F}_{0,L}$ satisfies
\begin{equation}\label{eq:F_0L_formula}
	f(x)-f(y)-\langle\nabla f(y),x-y\rangle
	-\frac{1}{2L}\|\nabla f(x)-\nabla f(y)\|^{2}\geq 0
\end{equation}
for all $x,y\in\mathbb{R}^{d}$
\cite[Theorem A.1]{AspremontEtAl2021Acceleration}.

Throughout, $x_{\star}$ denotes a minimizer of the function $f$
under consideration, $f_{\star}\triangleq f(x_{\star})$, and the subscript
$\star$ refers to this optimal point. For a method that runs for $N$
iterations and produces the iterates $x_{0},x_{1},\ldots,x_{N}$, we collect the
iteration indices together with the optimal point in the index set
$I^{\star}_{N}\triangleq\{0,1,\ldots,N,\star\}$.

\subsection{Lemniscate acceleration for efficient gradient-norm minimization}\label{sec:lemniscate}

Consider the unconstrained minimization problem
\begin{equation}\label{eq:main-opt}
\begin{array}{ll}
\underset{x\in\mathbb{R}^{d}}{\mbox{minimize}} & f(x),\end{array}
\end{equation}
where $f\colon\mathbb{R}^{d}\to\mathbb{R}$ is an $L$-smooth convex function. We seek an algorithm that uses $N$ gradient evaluations to produce an output $x_N$ with a small worst-case guarantee on $\norm{\nabla f(x_N)}$.

Throughout this section, we use the single-gradient-step notation
\[
z^{+}\triangleq z-\frac{1}{L}\nabla f(z),\qquad\forall\,z\in\mathbb{R}^d.
\]

\subsubsection*{Algorithm statement}
The \emph{lemniscate accelerated gradient method} \eqref{eq:lemniscate-acceleration} is
\begin{equation*}\tag{LemniAcc}\label{eq:lemniscate-acceleration}
\begin{aligned}
 z_{k+1}
 &=z_{k}-\Omega_{N}\left(
    \frac{1+\rho_{k+1}^{2}}{2\rho_{k+1}}
    -\frac{1+\rho_{k}^{2}}{2\rho_{k}}
    \right)\frac{1}{L}\nabla f(x_{k})\\
 x_{k+1}
 &=x_{k}^{+}+\left(
    \frac{1+\rho_{k+1}^{2}}{1-\rho_{k+1}^{2}}
    -\frac{1+\rho_{k+2}^{2}}{1-\rho_{k+2}^{2}}
    \right)z_{k+1}
\end{aligned}
\end{equation*}
for $k=0,1,2,\dots, N-1$,
where $N\geq1$ is an integer-valued  prespecified iteration budget, $x_{0}\in\mathbb{R}^{d}$ is the starting point, $z_{0}=0$, and $\Omega_N>0$ and $1=\rho_{0}>\rho_{1}>\cdots>\rho_{N}>\rho_{N+1}=0$ are constants that we define in Lemma~\ref{lem:lemniscate-recurrence}.

We can also express \eqref{eq:lemniscate-acceleration} equivalently in the so-called momentum form
\begin{equation}
\begin{aligned} & x_{1}=x^{+}_{0}+\alpha^{\text{(init)}}(x^{+}_{0}-x_{0}),\\
 & x_{k+1}=x^{+}_{k}+\alpha^{\text{(mom)}}_{k}(x^{+}_{k}-x^{+}_{k-1})+\alpha^{\text{(corr)}}_{k}(x^{+}_{k}-x_{k}),\qquad k=1,\dots,N-1,
\end{aligned}
\label{eq:lemniscate-momentum-form}
\end{equation}
where
\begin{align*}
\alpha^{\text{(init)}}&=\Omega_{N}
\left(
\frac{1+\rho_{1}^{2}}{2\rho_{1}}
-\frac{1+\rho_{0}^{2}}{2\rho_{0}}
\right)
\left(
\frac{1+\rho_{N}^{2}}{2\rho_{N}}
-\frac{1+\rho_{N-1}^{2}}{2\rho_{N-1}}
\right),
\\
\alpha^{\text{(mom)}}_k&=\frac{
\frac{1+\rho_{N-k}^{2}}{2\rho_{N-k}}
-\frac{1+\rho_{N-k-1}^{2}}{2\rho_{N-k-1}}
}{
\frac{1+\rho_{N-k+1}^{2}}{2\rho_{N-k+1}}
-\frac{1+\rho_{N-k}^{2}}{2\rho_{N-k}}
},\\
\alpha^{\text{(corr)}}_k&=\left(
\frac{1+\rho_{N-k}^{2}}{2\rho_{N-k}}
-\frac{1+\rho_{N-k-1}^{2}}{2\rho_{N-k-1}}
\right)
\left[
\Omega_{N}
\left(
\frac{1+\rho_{k+1}^{2}}{2\rho_{k+1}}
-\frac{1+\rho_{k}^{2}}{2\rho_{k}}
\right)
-
\frac{1}{
\displaystyle
\frac{1+\rho_{N-k+1}^{2}}{2\rho_{N-k+1}}
-\frac{1+\rho_{N-k}^{2}}{2\rho_{N-k}}
}
\right].
\end{align*}
The two forms are equivalent in the sense that they produce the same iterates $x_{0},\ldots,x_{N}$; a proof is given in Appendix~\ref{sec:lemniscate-form-equivalence}.

\begin{table}[t]
\centering
\begin{tabular}{c|c|ccccccc}
\toprule
$N$ & $\Omega_{N}$ & $\rho_{1}$ & $\rho_{2}$ & $\rho_{3}$ & $\rho_{4}$ & $\rho_{5}$ & $\rho_{6}$ & $\rho_{7}$\\
\midrule
1 & 2.4142 & 0.4142 & & & & & & \\
2 & 4.1628 & 0.6126 & 0.2402 & & & & & \\
3 & 6.2479 & 0.7241 & 0.4142 & 0.1601 & & & & \\
4 & 8.6645 & 0.7931 & 0.5380 & 0.3004 & 0.1154 & & & \\
5 & 11.4075 & 0.8388 & 0.6277 & 0.4142 & 0.2287 & 0.0877 & & \\
6 & 14.4730 & 0.8707 & 0.6942 & 0.5052 & 0.3287 & 0.1805 & 0.0691 & \\
7 & 17.8578 & 0.8939 & 0.7446 & 0.5780 & 0.4142 & 0.2675 & 0.1464 & 0.0560\\
\bottomrule
\end{tabular}
\vspace{0.05in}
\caption{Values of $\Omega_{N}$ and $\rho_{1},\dots,\rho_{N}$ for $N=1,\ldots,7$,
computed by the bisection method.
}
\label{tab:lemniscate-parameters}
\end{table}

\subsubsection*{Construction of the coefficients}

To fully define \eqref{eq:lemniscate-acceleration}, it remains to define the coefficients $\rho_0,\rho_1,\dots,\rho_N,\rho_{N+1}$ and $\Omega_N$.

\begin{lem}\label{lem:lemniscate-recurrence}
Let $N$ be a nonnegative integer. There exists a unique $\Omega_{N}>0$ for which there is a strictly decreasing nonnegative sequence $1=\rho_{0} > \rho_{1} > \ldots  > \rho_{N+1}=0$ satisfying the recurrence equation
\begin{equation}\label{eq:lemniscate-recurrence}
    \Omega_{N}\left(\rho_{k}-\rho_{k+1}\right)^{2}=\rho_{k}\left(1-{\rho_{k+1}}^{2}\right), \quad k=0,\dots,N.
\end{equation}
Moreover, for this value of $\Omega_{N}$, the sequence $\rho_{0},\rho_{1},\dots,\rho_{N+1}$ is uniquely determined.
\end{lem}

We can calculate the numerical values of $\{\rho_{k}\}_{k=0}^{N+1}$ and $\Omega_{N}$ using bisection based on $1\le \Omega_N\le (N+1)^2$.
Appendix~\ref{sec:app-recurrence-existence} gives the exact shooting decision rule, its termination case, and its correctness proof. Table~\ref{tab:lemniscate-parameters} reports the values of
$\Omega_{N}$ and of the nontrivial entries $\rho_{1},\dots,\rho_{N}$ for $N=1,\ldots,7$.

We note that the sequence $\{\rho_{k}\}_{k=0}^{N+1}$ has a symmetry: the map $T(\rho)\triangleq(1-\rho)/(1+\rho)$, which satisfies $T\circ T=\mathrm{id}$, relates mirrored entries through $\rho_{N+1-k}=T(\rho_{k})$ for $0\leq k\leq N+1$. This follows from the uniqueness in Lemma~\ref{lem:lemniscate-recurrence} together with the fact that $(\rho_{k},\rho_{k+1})=(a,b)$ satisfies the recurrence \eqref{eq:lemniscate-recurrence} if and only if $(T(b),T(a))$ does. In particular, for odd $N$, the middle index $k=(N+1)/2$ is its own mirror image, so $\rho_{(N+1)/2}$ must be the point such that $T(\rho)=\rho$, namely $\rho_{(N+1)/2}=\sqrt{2}-1\approx0.4142$.

\subsubsection*{Convergence theorem}

\begin{thm}\label{thm:lemniscate-convergence}
Let $N\geq1$ be an integer, and let $f\colon\mathbb{R}^{d} \to \mathbb{R}$ be an $L$-smooth convex function with a minimizer $x_\star$. Then, the final iterate $x_{N}$ of~\eqref{eq:lemniscate-acceleration} satisfies
\[       \norm{\nabla f(x_N)}^{2} \leq \frac{L^{2} \norm{x_{0}-x_{\star}}^{2}}{\Omega_{N}^{2}} \leq \frac{\varpi^{4} L^{2} \norm{x_{0}-x_{\star}}^{2}}{(N+1)^{4}}.
  \]
\end{thm}
Here, $\varpi = 2\int_0^1 1/\sqrt{1-x^{4}}\, dx$ is the lemniscate constant, and the latter relaxation comes from $ (N+1)^{2}/\varpi^{2}<\Omega_{N} $ (Corollary~\ref{cor:app-omega-bounds} in Appendix~\ref{sec:app-recurrence-existence}). We present a convergence proof soon in Section~\ref{sec:lemniscate-lyapunov}.

\eqref{eq:lemniscate-acceleration} is not the first method to attain an
$O(L^{2}\norm{x_{0}-x_{\star}}^{2}/N^{4})$ rate on the squared gradient norm
$\norm{\nabla f(x_{N})}^{2}$ of $L$-smooth convex functions.
Nesterov et al.~\cite[Remark~2.1]{NesterovGasnikovGuminovDvurechensky2020}
observed that running Nesterov's accelerated gradient method (AGM)
\cite{nesterov1983method}
(or Kim--Fessler's optimized gradient method (OGM) \cite{kim2016optimized})
for the first half of the iterations and OGM-G
\cite{kim2021optimizing} for the second half achieves this rate, which
matches the $\Omega(L^{2}\norm{x_{0}-x_{\star}}^{2}/N^{4})$ lower bound of
\cite{nemirovsky1991optimality,nemirovsky1992information} and is therefore
optimal up to a constant factor.

Lemniscate acceleration improves upon this two-phase scheme in two ways.
First, it attains the optimal rate as a single method, with one coefficient
sequence and no switch between two algorithms. Second, it improves the
constant of the guarantee. For even $N$, chaining the standard guarantee of
OGM
with that of OGM-G
yields
\[
\norm{\nabla f(x_{N})}^{2}
\leq\frac{64\,L^{2}\norm{x_{0}-x_{\star}}^{2}}{(N+2)^{4}}.
\]
In comparison,
Theorem~\ref{thm:lemniscate-convergence} attains the constant
$\varpi^{4}\approx47.27$, improving the leading asymptotic constant  by a factor of approximately $1.35$.

\subsubsection{Convergence proof} \label{sec:lemniscate-lyapunov}

In this section, we prove
Theorem~\ref{thm:lemniscate-convergence}. Throughout,
$f\colon\mathbb{R}^{d}\to\mathbb{R}$ is an $L$-smooth convex function
with a minimizer $x_{\star}$, and $x_{k}$ and $z_{k}$ denote the
iterates of \eqref{eq:lemniscate-acceleration}; in particular
$z_{0}=0$. The argument has three steps. We first define the
interpolation gaps and, from the multipliers of the analytic dual
certificate found in Stage 2 of the pipeline, a Lyapunov sequence
$\mathcal V_{0},\mathcal V_{1},\ldots,\mathcal V_{N}$ along the
iterates. We then evaluate the two endpoints: $\mathcal V_{0}$ is at
most $\left(L/(2\Omega_{N})\right)\norm{x_{0}-x_{\star}}^{2}$, and the
terminal value $\mathcal V_{N}$ dominates
$\left(\Omega_{N}/(2L)\right)\norm{\nabla f(x_{N})}^{2}$
(Lemma~\ref{lem:lyap-terminal}). Finally,
Lemma~\ref{lem:lyap-closed-decrement} shows that the sequence is
nonincreasing along the iterates, and chaining the three facts proves
Theorem~\ref{thm:lemniscate-convergence}; the same machinery yields a
function-value guarantee (Corollary~\ref{cor:lyap-function-value}).

Write $g_{i}\triangleq\nabla f(x_{i})$ and $g_{\star}=0$, so that
$x_{\star}^{+}=x_{\star}$, and denote the interpolation gap by
\[
\mathcal D_{i,j}
 \triangleq
 f(x_{i})-f(x_{j})-\langle g_{j},x_{i}-x_{j}\rangle
 -\frac{1}{2L}\norm{g_{i}-g_{j}}^{2}\geq0,
 \quad i,j\in I_{N}^{\star},
\]
which is nonnegative by \eqref{eq:F_0L_formula}. The three-point identity
\begin{equation}
 \mathcal D_{i,\star}
 -\left\langle x_{i}^{+}-x_{\star},g_{j}\right\rangle
 =\mathcal D_{i,j}-\mathcal D_{\star,j}
 \label{eq:lyap-three-point}
\end{equation}
follows from direct calculation.
For notational simplicity, we also set
$\mathcal{D}_{-1,0}\triangleq0$, $\mathcal{D}_{-1,\star}\triangleq0$,
and $x_{-1}^{+}\triangleq x_{0}$.
Evaluating the recurrence
\eqref{eq:lemniscate-recurrence} at $k=N$, where $\rho_{N+1}=0$, and at
$k=0$, where $\rho_{0}=1$, gives the two boundary identities
$\rho_{N}=1/\Omega_{N}$ and $\Omega_{N}=(1+\rho_{1})/(1-\rho_{1})>1$,
which we use repeatedly below.

For $0\leq k\leq N$, define the Lyapunov sequence
\begin{equation}
\label{eq:lyap-sequence-def}
 \mathcal V_{k}\triangleq
 \frac{1-\rho_{k}^{2}}{2\rho_{k}}
 \mathcal D_{k-1,\star}
 -\frac{(1-\rho_{k})^{2}}{2\rho_{k}}
 \mathcal D_{N,\star}
 +\frac{L}{2\Omega_{N}}\norm{
 x_{k}-x_{\star}
 +\frac{1+\rho_{k+1}^{2}}{1-\rho_{k+1}^{2}}z_{k}}^{2}
 -\frac{L}{2\Omega_{N}}\norm{
 x_{N}^{+}-x_{\star}
 +z_{k}}^{2}.
\end{equation}
At $k=0$, since $\mathcal{D}_{-1,\star}=0$,
$\rho_{0}=1$, and $z_{0}=0$, we get
\begin{equation*}
 \mathcal V_{0}
 =\frac{L}{2\Omega_{N}}\norm{x_{0}-x_{\star}}^{2}
 -\frac{L}{2\Omega_{N}}\norm{x_{N}^{+}-x_{\star}}^{2}
 \leq \frac{L}{2\Omega_{N}}\norm{x_{0}-x_{\star}}^{2}.
\end{equation*}

Next, we evaluate the terminal value $\mathcal V_{N}$.
\begin{lem}[Terminal decomposition]
\label{lem:lyap-terminal}
The terminal value of the Lyapunov sequence
\eqref{eq:lyap-sequence-def} satisfies
\begin{equation*}
 \mathcal V_{N}
 =\frac{\Omega_{N}}{2L}\norm{\nabla f(x_{N})}^{2}
 +\mathcal D_{N,\star}
 +\frac{1}{\Omega_{N}}\mathcal D_{\star,N}
 +\frac{\Omega_{N}^{2}-1}{2\Omega_{N}}\mathcal D_{N-1,N}
 \geq\frac{\Omega_{N}}{2L}\norm{\nabla f(x_{N})}^{2}.
\end{equation*}
\end{lem}

\begin{proof}[Proof of Lemma~\ref{lem:lyap-terminal}]
Solving the $x$-update at $k=N-1$, whose coefficient evaluates to
$(1+\rho_{N}^{2})/(1-\rho_{N}^{2})
-(1+\rho_{N+1}^{2})/(1-\rho_{N+1}^{2})
=2/(\Omega_{N}^{2}-1)$,
for $z_{N}$ and adding $x_{N}-x_{\star}$ gives
\begin{equation}\label{eq:lyap-terminal-z}
 x_{N}-x_{\star}+z_{N}
 =-\frac{\Omega_{N}^{2}-1}{2}(x_{N-1}^{+}-x_{\star})
 +\frac{\Omega_{N}^{2}+1}{2}(x_{N}-x_{\star}).
\end{equation}
Substituting $\rho_{N}=1/\Omega_{N}$, $\rho_{N+1}=0$ into
the Lyapunov sequence at $k=N$ and expanding,
\begin{equation*}
\begin{aligned}
 \mathcal V_{N}
 ={}&\frac{\Omega_{N}^{2}-1}{2\Omega_{N}}\mathcal D_{N-1,\star}
 -\frac{(\Omega_{N}-1)^{2}}{2\Omega_{N}}\mathcal D_{N,\star}
 +\frac{L}{2\Omega_{N}}\norm{x_{N}-x_{\star}+z_{N}}^{2}
 -\frac{L}{2\Omega_{N}}\norm{x_{N}^{+}-x_{\star}+z_{N}}^{2}\\
 ={}&\frac{\Omega_{N}^{2}-1}{2\Omega_{N}}\mathcal D_{N-1,\star}
 -\frac{(\Omega_{N}-1)^{2}}{2\Omega_{N}}\mathcal D_{N,\star}
 +\frac{1}{\Omega_{N}}\left\langle x_{N}-x_{\star}+z_{N},g_{N}\right\rangle
 -\frac{1}{2L\Omega_{N}}\norm{g_{N}}^{2}\\
 ={}&\frac{\Omega_{N}^{2}-1}{2\Omega_{N}}
 \left(
 \mathcal D_{N-1,\star}
 -\left\langle x_{N-1}^{+}-x_{\star},g_{N}\right\rangle
 \right)
 -\frac{(\Omega_{N}-1)^{2}}{2\Omega_{N}}\mathcal D_{N,\star}\\
 &{}+\frac{\Omega_{N}^{2}+1}{2\Omega_{N}}\left\langle x_{N}-x_{\star},g_{N}\right\rangle
 -\frac{1}{2L\Omega_{N}}\norm{g_{N}}^{2}\\
 ={}&\frac{\Omega_{N}^{2}-1}{2\Omega_{N}}
 \left(\mathcal D_{N-1,N}-\mathcal D_{\star,N}\right)
 -\frac{(\Omega_{N}-1)^{2}}{2\Omega_{N}}\mathcal D_{N,\star}\\
 &{}+\frac{\Omega_{N}^{2}+1}{2\Omega_{N}}\left\langle x_{N}^{+}-x_{\star},g_{N}\right\rangle
 +\frac{\Omega_{N}}{2L}\norm{g_{N}}^{2}\\
 ={}&\frac{\Omega_{N}}{2L}\norm{g_{N}}^{2}
 +\mathcal D_{N,\star}
 +\frac{1}{\Omega_{N}}\mathcal D_{\star,N}
 +\frac{\Omega_{N}^{2}-1}{2\Omega_{N}}\mathcal D_{N-1,N},
\end{aligned}
\end{equation*}
where the second equality expands the difference of the two squared
norms, the third substitutes \eqref{eq:lyap-terminal-z} and groups the
terms by their coefficients, the fourth applies
\eqref{eq:lyap-three-point} at $(i,j)=(N-1,N)$ and substitutes
$x_{N}=x_{N}^{+}+\left(1/L\right) g_{N}$, and the last applies
\eqref{eq:lyap-three-point} at $(i,j)=(N,N)$ in the form
$\left\langle x_{N}^{+}-x_{\star},g_{N}\right\rangle
=\mathcal D_{N,\star}+\mathcal D_{\star,N}$
using $\mathcal D_{N,N}=0$, and collects terms. The three interpolation gaps are nonnegative and their multipliers are
positive by $\Omega_{N}=(1+\rho_{1})/(1-\rho_{1})>1$, which gives the second claim.
\end{proof}

\begin{lem}[Monotonicity]
\label{lem:lyap-closed-decrement}
The Lyapunov sequence \eqref{eq:lyap-sequence-def},
evaluated on the iterates of \eqref{eq:lemniscate-acceleration},
satisfies for $0\leq k\leq N-1$
\begin{equation}
 \mathcal V_{k}-\mathcal V_{k+1}
 =\frac{1-\rho_{k}^{2}}{2\rho_{k}}
 \mathcal D_{k-1,k}
 +(\rho_{k}-\rho_{k+1})\mathcal D_{\star,k}
 +\left(
 \frac{1+\rho_{k+1}^{2}}{2\rho_{k+1}}
 -\frac{1+\rho_{k}^{2}}{2\rho_{k}}
 \right)\mathcal D_{N,k}\geq0.
 \label{eq:lyap-closed-decrement}
\end{equation}
Consequently,
$\mathcal V_{0}\geq\mathcal V_{1}\geq\cdots\geq\mathcal V_{N}$.
\end{lem}

\begin{proof}
We calculate the difference $\mathcal V_{k}-\mathcal V_{k+1}$
and group the resulting terms by the three-point identity
\eqref{eq:lyap-three-point}; we first prepare the supporting
identities.

We use the following two scalar identities. Rearranging the recurrence
\eqref{eq:lemniscate-recurrence} gives for $0\leq k\leq N-1$,
\begin{equation}\label{eq:lyap-scalar-identities}
\begin{aligned}
&\Omega_{N}\left(
\frac{1+\rho_{k+1}^{2}}{2\rho_{k+1}}-\frac{1+\rho_{k}^{2}}{2\rho_{k}}
\right)\frac{1+\rho_{k+1}^{2}}{1-\rho_{k+1}^{2}}
+1
=\Omega_{N}\left(
\frac{1-\rho_{k+1}^{2}}{2\rho_{k+1}}-\frac{1-\rho_{k}^{2}}{2\rho_{k}}
\right),\\
 &\frac{\Omega_{N}}{2}\Bigg[
 \left(
 \frac{1-\rho_{k+1}^{2}}{2\rho_{k+1}}
 -\frac{1-\rho_{k}^{2}}{2\rho_{k}}
 \right)^{2}
 -\left(
 \frac{1+\rho_{k+1}^{2}}{2\rho_{k+1}}
 -\frac{1+\rho_{k}^{2}}{2\rho_{k}}
 \right)^{2}
 \Bigg]
 =\frac{1-\rho_{k+1}^{2}}{2\rho_{k+1}}.
\end{aligned}
\end{equation}

The update equations of
\eqref{eq:lemniscate-acceleration} give the two one-step relations
\begin{equation}\label{eq:lyap-onestep}
\begin{aligned}
 x_{k+1}-x_{\star}
 +\frac{1+\rho_{k+2}^{2}}{1-\rho_{k+2}^{2}}z_{k+1}
 &=x_{k}-x_{\star}
 +\frac{1+\rho_{k+1}^{2}}{1-\rho_{k+1}^{2}}z_{k}
 -\frac{\Omega_{N}}{L}\left(
 \frac{1-\rho_{k+1}^{2}}{2\rho_{k+1}}
 -\frac{1-\rho_{k}^{2}}{2\rho_{k}}
 \right)g_{k},\\
 x_{N}^{+}-x_{\star}+z_{k+1}
 &=x_{N}^{+}-x_{\star}+z_{k}
 -\frac{\Omega_{N}}{L}\left(
 \frac{1+\rho_{k+1}^{2}}{2\rho_{k+1}}
 -\frac{1+\rho_{k}^{2}}{2\rho_{k}}
 \right)g_{k},
\end{aligned}
\end{equation}
both for $0\leq k\leq N-1$.  
For the first relation, the $x$-update gives
\begin{equation*}
x_{k+1}-x_{\star}
+\frac{1+\rho_{k+2}^{2}}{1-\rho_{k+2}^{2}}z_{k+1}
=x_{k}-x_{\star}-\frac{1}{L}g_{k}
+\frac{1+\rho_{k+1}^{2}}{1-\rho_{k+1}^{2}}z_{k+1},
\end{equation*}
and substituting the $z$-update for $z_{k+1}$ on the right-hand side
and evaluating the resulting coefficient of $g_{k}$ by
the first line of \eqref{eq:lyap-scalar-identities} proves it.
The second relation of \eqref{eq:lyap-onestep} follows by adding
$x_{N}^{+}-x_{\star}$ to both sides of the $z$-update.

We also state a useful vector identity:
for $0\leq k\leq N-1$,
\begin{equation}
\begin{aligned}
 &\left(
 \frac{1-\rho_{k+1}^{2}}{2\rho_{k+1}}
 -\frac{1-\rho_{k}^{2}}{2\rho_{k}}
 \right)
 \left(
 x_{k}-x_{\star}
 +\frac{1+\rho_{k+1}^{2}}{1-\rho_{k+1}^{2}}z_{k}
 \right)
 -\left(
 \frac{1+\rho_{k+1}^{2}}{2\rho_{k+1}}
 -\frac{1+\rho_{k}^{2}}{2\rho_{k}}
 \right)
 \left(
 x_{N}^{+}-x_{\star}+z_{k}
 \right)\\
 &\quad=-\frac{1-\rho_{k}^{2}}{2\rho_{k}}(x_{k-1}^{+}-x_{\star})
 +\frac{1-\rho_{k+1}^{2}}{2\rho_{k+1}}(x_{k}-x_{\star})
 -\left(
 \frac{1+\rho_{k+1}^{2}}{2\rho_{k+1}}
 -\frac{1+\rho_{k}^{2}}{2\rho_{k}}
 \right)(x_{N}^{+}-x_{\star}).
\end{aligned}
 \label{eq:lyap-direct-vector-identity}
\end{equation}
At $k=0$, both sides equal
$\left((1-\rho_{1}^{2})/(2\rho_{1})\right)(x_{0}-x_{\star})
-\left(\left((1+\rho_{1}^{2})/(2\rho_{1})\right)-1\right)(x_{N}^{+}-x_{\star})$,
by $z_{0}=0$, $(1-\rho_{0}^{2})/(2\rho_{0})=0$, and
$x_{-1}^{+}=x_{0}$.  For $1\leq k\leq N-1$, distribute the two
products on the left-hand side: the coefficient of $z_{k}$ collects to
\begin{equation*}
\left(
\frac{1-\rho_{k+1}^{2}}{2\rho_{k+1}}-\frac{1-\rho_{k}^{2}}{2\rho_{k}}
\right)\frac{1+\rho_{k+1}^{2}}{1-\rho_{k+1}^{2}}
-\left(
\frac{1+\rho_{k+1}^{2}}{2\rho_{k+1}}-\frac{1+\rho_{k}^{2}}{2\rho_{k}}
\right)
=\frac{1-\rho_{k}^{2}}{2\rho_{k}}\left(
\frac{1+\rho_{k}^{2}}{1-\rho_{k}^{2}}
-\frac{1+\rho_{k+1}^{2}}{1-\rho_{k+1}^{2}}
\right),
\end{equation*}
and the $x$-update at index $k-1$ gives
\begin{equation*}
\frac{1-\rho_{k}^{2}}{2\rho_{k}}\left(
\frac{1+\rho_{k}^{2}}{1-\rho_{k}^{2}}
-\frac{1+\rho_{k+1}^{2}}{1-\rho_{k+1}^{2}}
\right)z_{k}
=\frac{1-\rho_{k}^{2}}{2\rho_{k}}
\left(x_{k}-x_{k-1}^{+}\right).
\end{equation*}
Substituting this and splitting
$x_{k}-x_{k-1}^{+}=(x_{k}-x_{\star})-(x_{k-1}^{+}-x_{\star})$
verifies \eqref{eq:lyap-direct-vector-identity}.

Fix $0\leq k\leq N-1$.  Expanding the difference $\mathcal V_{k}-\mathcal V_{k+1}$ using
the one-step relations \eqref{eq:lyap-onestep},
\begin{equation*}
\begin{aligned}
 \mathcal V_{k}-\mathcal V_{k+1}
 ={}&\frac{1-\rho_{k}^{2}}{2\rho_{k}}
 \mathcal D_{k-1,\star}
 -\frac{1-\rho_{k+1}^{2}}{2\rho_{k+1}}
 \mathcal D_{k,\star}
 +\left(
 \frac{1+\rho_{k+1}^{2}}{2\rho_{k+1}}
 -\frac{1+\rho_{k}^{2}}{2\rho_{k}}
 \right)\mathcal D_{N,\star}\\
 &{}+\left(
 \frac{1-\rho_{k+1}^{2}}{2\rho_{k+1}}
 -\frac{1-\rho_{k}^{2}}{2\rho_{k}}
 \right)\cdot\left\langle
 x_{k}-x_{\star}
 +\frac{1+\rho_{k+1}^{2}}{1-\rho_{k+1}^{2}}z_{k},
 g_{k}\right\rangle\\
 &{}-\left(
 \frac{1+\rho_{k+1}^{2}}{2\rho_{k+1}}
 -\frac{1+\rho_{k}^{2}}{2\rho_{k}}
 \right)\cdot\left\langle
 x_{N}^{+}-x_{\star}+z_{k},
 g_{k}\right\rangle\\
 &{}-\frac{\Omega_{N}}{2L}\Bigg[
 \left(
 \frac{1-\rho_{k+1}^{2}}{2\rho_{k+1}}
 -\frac{1-\rho_{k}^{2}}{2\rho_{k}}
 \right)^{2}
 -\left(
 \frac{1+\rho_{k+1}^{2}}{2\rho_{k+1}}
 -\frac{1+\rho_{k}^{2}}{2\rho_{k}}
 \right)^{2}
 \Bigg]\norm{g_{k}}^{2}\\
 ={}&\frac{1-\rho_{k}^{2}}{2\rho_{k}}
 \left(
 \mathcal D_{k-1,\star}
 -\left\langle x_{k-1}^{+}-x_{\star},g_{k}\right\rangle
 \right)\\
 &{}-\frac{1-\rho_{k+1}^{2}}{2\rho_{k+1}}
 \left(
 \mathcal D_{k,\star}
 -\left\langle x_{k}-x_{\star},g_{k}\right\rangle
 +\frac{1}{L}\norm{g_{k}}^{2}
 \right)\\
 &{}+\left(
 \frac{1+\rho_{k+1}^{2}}{2\rho_{k+1}}
 -\frac{1+\rho_{k}^{2}}{2\rho_{k}}
 \right)
 \left(
 \mathcal D_{N,\star}
 -\left\langle x_{N}^{+}-x_{\star},g_{k}\right\rangle
 \right)\\
 ={}&\frac{1-\rho_{k}^{2}}{2\rho_{k}}
 \left(\mathcal D_{k-1,k}-\mathcal D_{\star,k}\right)
 -\frac{1-\rho_{k+1}^{2}}{2\rho_{k+1}}
 \left(\mathcal D_{k,k}-\mathcal D_{\star,k}\right)
 +\left(
 \frac{1+\rho_{k+1}^{2}}{2\rho_{k+1}}
 -\frac{1+\rho_{k}^{2}}{2\rho_{k}}
 \right)
 \left(\mathcal D_{N,k}-\mathcal D_{\star,k}\right)\\
 ={}&\frac{1-\rho_{k}^{2}}{2\rho_{k}}
 \mathcal D_{k-1,k}
 +(\rho_{k}-\rho_{k+1})\mathcal D_{\star,k}
 +\left(
 \frac{1+\rho_{k+1}^{2}}{2\rho_{k+1}}
 -\frac{1+\rho_{k}^{2}}{2\rho_{k}}
 \right)\mathcal D_{N,k},
\end{aligned}
\end{equation*}
where the second equality substitutes the vector identity \eqref{eq:lyap-direct-vector-identity} and
the second line of \eqref{eq:lyap-scalar-identities},
the third applies \eqref{eq:lyap-three-point} at $j=k$
to each of the three groups, where at $k=0$ the first group vanishes since its factor $(1-\rho_{0}^{2})/(2\rho_{0})$ is zero,
and the fourth uses $\mathcal D_{k,k}=0$.  This proves the equality in
\eqref{eq:lyap-closed-decrement}.

Finally, the right-hand side of \eqref{eq:lyap-closed-decrement} is
nonnegative: the interpolation gaps are nonnegative, and the three multipliers
are nonnegative, since Lemma~\ref{lem:lemniscate-recurrence} gives
$1=\rho_{0}>\rho_{1}>\cdots>\rho_{N+1}=0$ and since the map
$\rho\mapsto(1+\rho^{2})/(2\rho)$ is strictly
decreasing for $0<\rho\leq1$, 
so that $(1+\rho_{k+1}^{2})/(2\rho_{k+1})>(1+\rho_{k}^{2})/(2\rho_{k})$ for $0\leq k\leq N-1$.
\end{proof}

Combining Lemma~\ref{lem:lyap-terminal} and
Lemma~\ref{lem:lyap-closed-decrement}, the proof of
Theorem~\ref{thm:lemniscate-convergence} is now immediate.
\begin{proof}[Proof of Theorem~\ref{thm:lemniscate-convergence}]
Combining Lemma~\ref{lem:lyap-terminal}, the
monotonicity of Lemma~\ref{lem:lyap-closed-decrement}, and the
evaluation of $\mathcal V_{0}$ gives
\[
\frac{\Omega_{N}}{2L}\norm{\nabla f(x_{N})}^{2}
 \leq\mathcal V_{N}
 \leq\mathcal V_{N-1}
\le\cdots
 \leq\mathcal V_{1}
\leq\mathcal V_{0}
 \leq\frac{L}{2\Omega_{N}}\norm{x_{0}-x_{\star}}^{2},
\]
and multiplying through by $2L/\Omega_{N}$ yields the first
inequality; the second inequality then follows from the lower bound
$\Omega_{N}>(N+1)^{2}/\varpi^{2}$ of
Corollary~\ref{cor:app-omega-bounds}.
\end{proof}

The Lyapunov sequence also yields a guarantee on the function value
of the final iterate.

\begin{cor}[Function-value bound]
\label{cor:lyap-function-value}
Under the assumptions of Theorem~\ref{thm:lemniscate-convergence},
the final iterate $x_{N}$ of~\eqref{eq:lemniscate-acceleration}
satisfies
\begin{equation*}
 f(x_{N})-f(x_{\star})
 \leq\frac{L}{2\Omega_{N}}\norm{x_{0}-x_{\star}}^{2}
 \leq\frac{\varpi^{2}L\norm{x_{0}-x_{\star}}^{2}}{2(N+1)^{2}}.
\end{equation*}
\end{cor}

\begin{proof}
Dropping the two nonnegative terms
$(1/\Omega_{N})\mathcal D_{\star,N}$ and
$\left((\Omega_{N}^{2}-1)/(2\Omega_{N})\right)\mathcal D_{N-1,N}$ in
Lemma~\ref{lem:lyap-terminal} and expanding
$\mathcal D_{N,\star}$,
\begin{equation*}
\begin{aligned}
 \mathcal V_{N}
 &\geq\frac{\Omega_{N}}{2L}\norm{\nabla f(x_{N})}^{2}
 +\mathcal D_{N,\star}\\
 &=f(x_{N})-f(x_{\star})
 +\frac{\Omega_{N}-1}{2L}\norm{\nabla f(x_{N})}^{2}
 \geq f(x_{N})-f(x_{\star}),
\end{aligned}
\end{equation*}
where the last inequality uses $\Omega_{N}=(1+\rho_{1})/(1-\rho_{1})>1$.  Combining with
$\mathcal V_{N}\leq\mathcal V_{0}
\leq \left(L/(2\Omega_{N})\right)\norm{x_{0}-x_{\star}}^{2}$ proves the first
inequality; the second follows from
$(N+1)^{2}/\varpi^{2}<\Omega_{N}$
(Corollary~\ref{cor:app-omega-bounds}), as in the proof of
Theorem~\ref{thm:lemniscate-convergence}.
\end{proof}

\subsubsection{\textsf{AutoOPT} design: Optimal method for reducing the gradient magnitude from an initial-distance bound}
\label{sec:lemniscate-autooptdesign}

The method \eqref{eq:lemniscate-acceleration} was designed by \textsf{AutoOPT} as the optimal fixed-step first-order method such that for all $L$-smooth convex functions $f$ with minimizer $x_\star$, the guarantee
\[
\|\nabla f(x_N)\|^2
\leq
C_N\|x_0-x_\star\|^2
\]
is optimal, i.e., $C_N$ is minimized.
Specifically, among all \(N\)-step fixed-step first-order methods of the form
\[
x_i
=
x_{i-1}
-
\frac{1}{L}
\sum_{j=0}^{i-1} h_{i,j}\nabla f(x_j),
\qquad
i=1,\ldots,N,
\]
\textsf{AutoOPT} jointly searches over the coefficients
\(\{h_{i,j}\}_{0\leq j<i\leq N}\) and a convergence certificate to minimize $C_N$. In other words, \textsf{AutoOPT} jointly optimizes over the algorithm and its convergence proof to find the algorithm with the optimal (smallest) convergence guarantee.

In Stage 1 of the pipeline, the skill
\textsf{bnb-pep-skill} wrote out the BnB-PEP formulation of this design
problem. At a high level, performance estimation programming replaces the
maximization over the infinite-dimensional function class
$\mathcal{F}_{0,L}$ by a finite-dimensional semidefinite program
involving only the function values, gradients, and pairwise inner
products observed along the trajectory; dualizing this inner semidefinite
program and jointly optimizing its dual variables and the method
coefficients produces a nonconvex quadratically constrained quadratic
program. Upon our approval of the formulation, reproduced with light
editing in Appendix~\ref{sec:lemniscate-pep-design}, the skill solved
this program for the horizons $N=1,2,3, \ldots , 25$, running spatial
branch-and-bound at our request to certify global optimality for $N=1,\ldots,5$ and local optimality for $N=6, \ldots, 25$, and returned the corresponding candidate
stepsizes together with the candidate dual certificates.

These numerical solutions, together with the Stage-1
derivation, became the context for Stage 2, in which the skill
\textsf{frontier-llm-consult} placed them before a frontier LLM. The
consultation first identified an analytic description valid for
arbitrary $N\in\mathbb{N}$, the parametrization in terms of $\Omega_N$
and the sequence $\rho_0,\ldots,\rho_{N+1}$, together with an analytic
feasible point of the dual certifying the guarantee. At
this point the method was still written as the dense stepsize array
$\{h_{i,j}\}$, its entries now in closed form; the consultation next
recast it into the memory-efficient two-sequence form
\eqref{eq:lemniscate-acceleration} in which Section~\ref{sec:lemniscate}
states the method. Continuing the same
conversation, Stage 2 then converted the multipliers of this certificate
into the Lyapunov proof of Theorem~\ref{thm:lemniscate-convergence}
presented in Section~\ref{sec:lemniscate-lyapunov}, a
proof stated directly on this two-sequence form; its Lean
formalization is discussed in Section~\ref{sec:lemniscate-lean}.

\subsubsection{Background: Lemniscate constant and lemniscate elliptic functions}\label{sec:lemniscate-constant}
The lemniscate constant $\varpi\in\mathbb{R}$ is a transcendental number \cite{Waldschmidt2008} defined by
\[
\varpi=2 \int_{0}^{1} \frac{1}{\sqrt{1-x^{4}}}\, dx \approx 2.62205755.
\]
It is also characterized as the ratio of perimeter to diameter of the unit
lemniscate.
See
Figure~\ref{fig:lemniscate-curve}.

\begin{figure}[htb]
\centering
\definecolor{lemnired}{RGB}{228,26,28}
\begin{tikzpicture}[x=3.1cm, y=3.1cm]
\draw[black, -{Stealth[length=2mm]}, line width=0.5pt]
    (-1.22,0) -- (1.28,0) node[anchor=west, font=\scriptsize, black] {$x$};
\draw[black, -{Stealth[length=2mm]}, line width=0.5pt]
    (0,-0.42) -- (0,0.40) node[anchor=south west, font=\scriptsize, black, inner sep=1pt] {$y$};
\draw[black, line width=1.4pt] plot[domain=0:360, samples=400, smooth, variable=\t]
    ({cos(\t)/(1+sin(\t)*sin(\t))},
     {cos(\t)*sin(\t)/(1+sin(\t)*sin(\t))});
\node[font=\small, text=black] at (-0.69,0.52) {$(x^{2}+y^{2})^{2}=x^{2}-y^{2}$};
\fill[black] ({1/sqrt(2)},0) circle [radius=1.2pt];
\fill[black] ({-1/sqrt(2)},0) circle [radius=1.2pt];
\node[font=\scriptsize, anchor=north, text=black] at ({1/sqrt(2)},-0.032) {$1/\sqrt{2}$};
\node[font=\scriptsize, anchor=north, text=black] at ({-1/sqrt(2)},-0.032) {$-1/\sqrt{2}$};
\fill[black] (1,0) circle [radius=1.4pt];
\fill[black] (-1,0) circle [radius=1.4pt];
\draw[black, line width=0.6pt, dashed] (-1,0) -- (-1,-0.52);
\draw[black, line width=0.6pt, dashed] (1,0) -- (1,-0.52);
\draw[lemnired, {Stealth[length=2.2mm]}-{Stealth[length=2.2mm]}, line width=0.9pt]
    (-1,-0.48) -- (1,-0.48)
    node[midway, anchor=north, font=\small, text=lemnired] {diameter $2$};
\node[lemnired, font=\small, anchor=south] (perim) at (0.80,0.51)
    {perimeter $2\varpi$};
\draw[lemnired, -{Stealth[length=2mm]}, line width=0.9pt, shorten >=2pt]
    (perim.south) -- ({2*sqrt(3)/5},{sqrt(3)/5});
\end{tikzpicture}\caption{The unit lemniscate $(x^{2}+y^{2})^{2}=x^{2}-y^{2}$ with foci at $(\pm1/\sqrt{2},0)$ and tips at
$(\pm1,0)$. Its perimeter is $2\varpi$,
and its diameter is $2$. Thus $\varpi\approx2.62206$ is its
perimeter-to-diameter ratio.}
\label{fig:lemniscate-curve}
\end{figure}
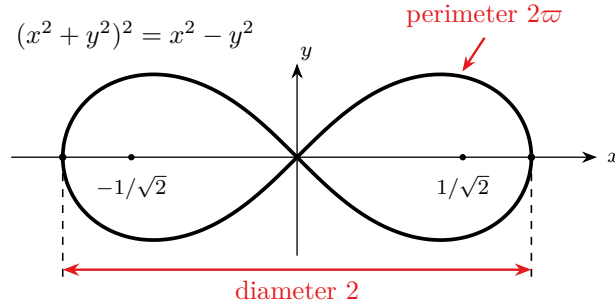

The lemniscate elliptic functions are defined through the arc-length integral of the lemniscate. Specifically, let
\[
\operatorname{arcsl}(u)=
\int_{0}^{u}\frac{dx}{\sqrt{1-x^{4}}},
\]
for $u\in[0,1]$ and define the lemniscate sine $\operatorname{sl}(x)$ on
$[0,\varpi/2]$ as the inverse of $\operatorname{arcsl}$. The lemniscate
cosine $\operatorname{cl}(x)$ is defined on $[0,\varpi/2]$ by
\[
\operatorname{cl}(x)
=
\operatorname{sl}\left(\frac{\varpi}{2}-x\right).
\]

Finally, we briefly note that while $\Omega_{N}$ is an algebraic number, since a finite optimal value of an QCQP with rational coefficients, 
the limit $N^{2}/\Omega_{N} \rightarrow \varpi^{2}$
(see Lemma~\ref{lem:app-omega-limit} in
Appendix~\ref{sec:app-recurrence-existence}) is transcendental.

\subsubsection{Continuous-time analysis}
\label{sec:lemniscate-continuous-time}

In this section, we present the continuous time limit of lemniscate acceleration as an ODE, and state an analogous convergence theorem. The details for this section can be found in Appendix~\ref{sec:app-continuous-time}.

\paragraph{Continuous-time ODE.}
For a finite time interval $[0,T]$, define $\rho$ on $[0,T]$ and $\sigma$ on $[0,T)$ by
\begin{align*}
\rho(t)&\triangleq\operatorname{cl}^{2}\left(\frac{\varpi t}{2T}\right),\\
\sigma(t)& \triangleq\sqrt{1/\rho(t)-\rho(t)} = \sqrt{\operatorname{cl}^{-2}\left(\varpi t/(2T)\right)-\operatorname{cl}^{2}\left(\varpi t/(2T)\right)}
\end{align*}
where  $\varpi$ is the lemniscate
constant.
Figure~\ref{fig:lemniscate-continuous-coefficients} shows these profiles and
the discrete coefficients $\rho_k$ approaching the continuous curve $\rho(t)$.
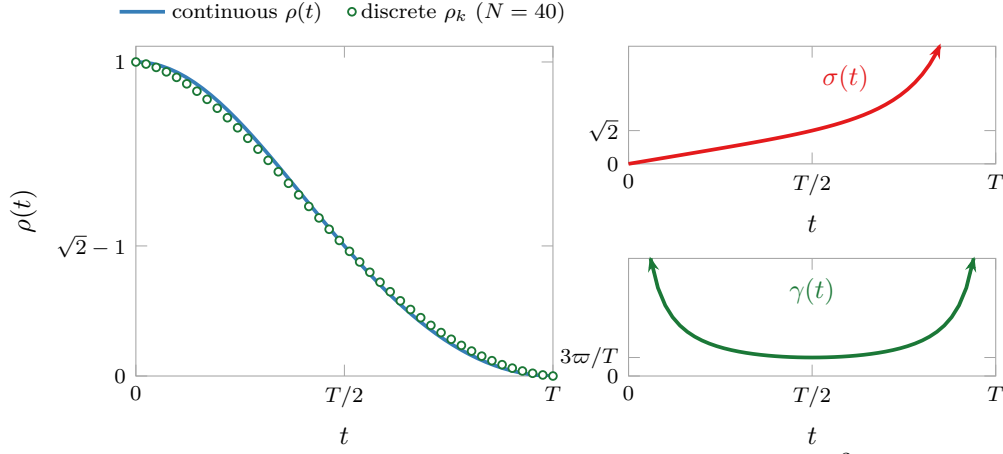
\begin{figure}[t]
\centering
\definecolor{lemnictblue}{RGB}{55,126,184}\definecolor{lemnictred}{RGB}{228,26,28}\definecolor{lemnictgreen}{RGB}{27,120,55}\begin{tikzpicture}
\begin{axis}[
    name=rhoplot,
    width=0.43\textwidth,
    height=0.36\textwidth,
    xmin=0, xmax=1,
    xlabel={$t$},
    xtick={0,0.5,1},
    xticklabels={$0$,$T/2$,$T$},
    ylabel={$\rho(t)$},
    ymin=0, ymax=1.05,
    ytick={0,0.4142135624,1},
    yticklabels={$0$,$\sqrt{2}-1$,$1$},
    label style={font=\small},
    tick label style={font=\scriptsize},
    axis line style={gray!70},
    tick style={gray!70},
    legend style={
        at={(axis description cs:0.5,1.04)},
        anchor=south,
        draw=none,
        fill=none,
        font=\scriptsize,
        legend columns=-1,
        /tikz/column 2/.style={column sep=0.7em},
    },
    legend cell align={left},
]
\addplot[lemnictblue, line width=1.4pt] coordinates {
    (0.00,1.0000000000)
    (0.02,0.9986259077)
    (0.04,0.9945149440)
    (0.06,0.9877008627)
    (0.08,0.9782393061)
    (0.10,0.9662069022)
    (0.12,0.9517000364)
    (0.14,0.9348333332)
    (0.16,0.9157378870)
    (0.18,0.8945592914)
    (0.20,0.8714555154)
    (0.22,0.8465946805)
    (0.24,0.8201527895)
    (0.26,0.7923114589)
    (0.28,0.7632557003)
    (0.30,0.7331717945)
    (0.32,0.7022452938)
    (0.34,0.6706591839)
    (0.36,0.6385922274)
    (0.38,0.6062175079)
    (0.40,0.5737011828)
    (0.42,0.5412014505)
    (0.44,0.5088677291)
    (0.46,0.4768400425)
    (0.48,0.4452486005)
    (0.50,0.4142135624)
    (0.52,0.3838449657)
    (0.54,0.3542428039)
    (0.56,0.3254972331)
    (0.58,0.2976888903)
    (0.60,0.2708893034)
    (0.62,0.2451613746)
    (0.64,0.2205599212)
    (0.66,0.1971322573)
    (0.68,0.1749188012)
    (0.70,0.1539536971)
    (0.72,0.1342654384)
    (0.74,0.1158774833)
    (0.76,0.0988088536)
    (0.78,0.0830747111)
    (0.80,0.0686869036)
    (0.82,0.0556544781)
    (0.84,0.0439841555)
    (0.86,0.0336807650)
    (0.88,0.0247476368)
    (0.90,0.0171869490)
    (0.92,0.0110000311)
    (0.94,0.0061876199)
    (0.96,0.0027500702)
    (0.98,0.0006875185)
    (1.00,0.0000000000)
};
\addlegendentry{continuous $\rho(t)$}

\addplot[
    only marks,
    mark=*,
    mark size=1.45pt,
    mark options={draw=lemnictgreen, fill=white, line width=0.7pt},
] coordinates {
    (0.00000000,1.0000000000)
    (0.02439024,0.9933780970)
    (0.04878049,0.9827560973)
    (0.07317073,0.9685361886)
    (0.09756098,0.9509977519)
    (0.12195122,0.9303947570)
    (0.14634146,0.9069798913)
    (0.17073171,0.8810115953)
    (0.19512195,0.8527552267)
    (0.21951220,0.8224817832)
    (0.24390244,0.7904655359)
    (0.26829268,0.7569812131)
    (0.29268293,0.7223010805)
    (0.31707317,0.6866921226)
    (0.34146341,0.6504134547)
    (0.36585366,0.6137140455)
    (0.39024390,0.5768307964)
    (0.41463415,0.5399870006)
    (0.43902439,0.5033911850)
    (0.46341463,0.4672363229)
    (0.48780488,0.4316993980)
    (0.51219512,0.3969412872)
    (0.53658537,0.3631069302)
    (0.56097561,0.3303257462)
    (0.58536585,0.2987122613)
    (0.60975610,0.2683669070)
    (0.63414634,0.2393769551)
    (0.65853659,0.2118175566)
    (0.68292683,0.1857528551)
    (0.70731707,0.1612371511)
    (0.73170732,0.1383160987)
    (0.75609756,0.1170279237)
    (0.78048780,0.0974046591)
    (0.80487805,0.0794734087)
    (0.82926829,0.0632576668)
    (0.85365854,0.0487787570)
    (0.87804878,0.0360575176)
    (0.90243902,0.0251165067)
    (0.92682927,0.0159833543)
    (0.95121951,0.0086969359)
    (0.97560976,0.0033219503)
    (1.00000000,0.0000000000)
};
\addlegendentry{discrete $\rho_k$ ($N=40$)}
\end{axis}

\begin{axis}[
    name=sigmaplot,
    at={($(rhoplot.north east)+(1.0cm,0)$)},
    anchor=north west,
    width=0.39\textwidth,
    height=0.19\textwidth,
    xmin=0, xmax=1,
    xlabel={$t$},
    xtick={0,0.5,1},
    xticklabels={$0$,$T/2$,$T$},
    ymin=0, ymax=5,
    ytick={0,1.4142135624},
    yticklabels={$0$,$\sqrt{2}$},
    label style={font=\small},
    tick label style={font=\scriptsize},
    axis line style={gray!70},
    tick style={gray!70},
    clip=false,
]
\addplot[lemnictred, line width=1.4pt] coordinates {
    (0.000,0.0000000000)
    (0.020,0.0524411610)
    (0.040,0.1048826195)
    (0.060,0.1573258627)
    (0.080,0.2097747579)
    (0.100,0.2622367435)
    (0.120,0.3147240228)
    (0.140,0.3672547653)
    (0.160,0.4198543195)
    (0.180,0.4725564478)
    (0.200,0.5254045947)
    (0.220,0.5784532057)
    (0.240,0.6317691175)
    (0.260,0.6854330480)
    (0.280,0.7395412201)
    (0.300,0.7942071628)
    (0.320,0.8495637429)
    (0.340,0.9057654941)
    (0.360,0.9629913261)
    (0.380,1.0214477164)
    (0.400,1.0813725144)
    (0.420,1.1430395214)
    (0.440,1.2067640550)
    (0.460,1.2729097660)
    (0.480,1.3418970572)
    (0.500,1.4142135624)
    (0.520,1.4904272941)
    (0.540,1.5712032804)
    (0.560,1.6573248032)
    (0.580,1.7497207773)
    (0.600,1.8495014191)
    (0.620,1.9580052584)
    (0.640,2.0768619049)
    (0.660,2.2080770498)
    (0.680,2.3541494288)
    (0.700,2.5182346542)
    (0.720,2.7043793444)
    (0.740,2.9178633942)
    (0.760,3.1657134617)
    (0.780,3.4574966138)
    (0.800,3.8065902356)
    (0.820,4.2322986162)
    (0.840,4.7635570415)
    (0.847545421,5.0000000000)
};
\draw[lemnictred, densely dashed, line width=0.9pt, -{Stealth[length=2mm]}]
    (axis cs:0.84,4.7636) -- (axis cs:0.850,5.12);
\node[lemnictred, anchor=north west, font=\small]
    at (axis description cs:0.5,0.93) {$\sigma(t)$};
\end{axis}

\begin{axis}[
    name=gammaplot,
    at={($(rhoplot.south east)+(1.0cm,0)$)},
    anchor=south west,
    width=0.39\textwidth,
    height=0.19\textwidth,
    xmin=0, xmax=1,
    xlabel={$t$},
    xtick={0,0.5,1},
    xticklabels={$0$,$T/2$,$T$},
    ymin=0, ymax=50,
    ytick={0,7.8661726629},
    yticklabels={$0$,$3\varpi/T$},
    label style={font=\small},
    tick label style={font=\scriptsize},
    axis line style={gray!70},
    tick style={gray!70},
    clip=false,
]
\addplot[lemnictgreen, line width=1.4pt] coordinates {
    (0.060003677,50.0000000000)
    (0.080000000,37.5072606267)
    (0.100000000,30.0141815713)
    (0.120000000,25.0245078284)
    (0.140000000,21.4674943755)
    (0.160000000,18.8081131491)
    (0.180000000,16.7494355464)
    (0.200000000,15.1135868760)
    (0.220000000,13.7876364791)
    (0.240000000,12.6965448298)
    (0.260000000,11.7885989569)
    (0.280000000,11.0270925579)
    (0.300000000,10.3853372041)
    (0.320000000,9.8435479427)
    (0.340000000,9.3868322391)
    (0.360000000,9.0038540528)
    (0.380000000,8.6859253345)
    (0.400000000,8.4263765686)
    (0.420000000,8.2201148149)
    (0.440000000,8.0633113943)
    (0.460000000,7.9531819915)
    (0.480000000,7.8878350171)
    (0.500000000,7.8661726629)
    (0.520000000,7.8878350171)
    (0.540000000,7.9531819915)
    (0.560000000,8.0633113943)
    (0.580000000,8.2201148149)
    (0.600000000,8.4263765686)
    (0.620000000,8.6859253345)
    (0.640000000,9.0038540528)
    (0.660000000,9.3868322391)
    (0.680000000,9.8435479427)
    (0.700000000,10.3853372041)
    (0.720000000,11.0270925579)
    (0.740000000,11.7885989569)
    (0.760000000,12.6965448298)
    (0.780000000,13.7876364791)
    (0.800000000,15.1135868760)
    (0.820000000,16.7494355464)
    (0.840000000,18.8081131491)
    (0.860000000,21.4674943755)
    (0.880000000,25.0245078284)
    (0.900000000,30.0141815713)
    (0.920000000,37.5072606267)
    (0.939996323,50.0000000000)
};
\draw[lemnictgreen, densely dashed, line width=0.9pt, -{Stealth[length=2mm]}]
    (axis cs:0.08,37.5073) -- (axis cs:0.058,50.8);
\draw[lemnictgreen, densely dashed, line width=0.9pt, -{Stealth[length=2mm]}]
    (axis cs:0.92,37.5073) -- (axis cs:0.942,50.8);
\node[lemnictgreen, anchor=north, font=\small]
    at (axis description cs:0.5,0.92) {$\gamma(t)$};
\end{axis}
\end{tikzpicture}
\caption{Continuous-time coefficient profiles.
The left panel shows
$\rho(t)=\operatorname{cl}^{2}\!\left(\varpi t/(2T)\right)$ together with the
discrete coefficients $\rho_k$ for $N=40$, placed at
$t_k=k\,T/(N+1)$. We see that $\rho(t)$ is the continuous limit of the discrete coefficients.
The right panels show
$\sigma(t)=\sqrt{1/\rho(t)-\rho(t)}$ and
$\gamma(t)=3\dot{\sigma}(t)/\sigma(t)$.  
The arrows indicate endpoint divergences to $+\infty$.
}
\label{fig:lemniscate-continuous-coefficients}
\end{figure}

Taking the small-stepsize limit of \eqref{eq:lemniscate-acceleration} under
the quantization $t=k\,T/N$ and $L=N^{2}/T^{2}$ with $N\rightarrow\infty$
suggests that
the continuous-time counterparts $X$ and $Z$ of $x_k$ and
$\bigl(N/(\sqrt{\Omega_N}\,T)\bigr)z_k$, respectively, satisfy the ODE
\begin{align*}
      \dot{Z}(t)&=-\frac{\sigma(t)^{3}}{2} \nabla f(X(t)),\\
      \dot{X}(t)&=\frac{4}{\sigma(t)^{3}} Z(t)
\end{align*}
for $t\in(0,T)$,
with initial conditions $X(0)=x_0$ and $\dot{X}(0)=0$.
Eliminating $Z(t)$ and defining the friction coefficient
$\gamma(t)\triangleq 3\dot{\sigma}(t)/\sigma(t)$ yields the second-order ODE
\begin{equation*}
    \ddot{X} + \gamma(t)\dot{X} + 2 \nabla f(X) = 0,
\end{equation*}
which, in terms of the lemniscate functions $\operatorname{cl}$ and $\operatorname{sl}$,
has the form
\begin{equation*}
\Ddot{X}(t)+
\underbrace{\frac{3\varpi}{2T}\left[\frac{\operatorname{cl}\!\left(\frac{\varpi t}{2T}\right)}{\operatorname{sl}\!\left(\frac{\varpi t}{2T}\right)}+\frac{\operatorname{sl}\!\left(\frac{\varpi t}{2T}\right)}{\operatorname{cl}\!\left(\frac{\varpi t}{2T}\right)}\right]
}_{=\gamma(t)}
\dot{X}(t)+2\nabla f\left(X(t)\right)=0
\end{equation*}
for $t\in(0,T)$.
The endpoint expansions of $\gamma$ give $\gamma(t)\sim3/t$ as $t\downarrow0$ and $\gamma(t)\sim3/(T-t)$
as $t\uparrow T$ so near $t=0$ the ODE asymptotically matches the continuous-time limit of OGM and near
$t=T$ that of OGM-G \cite{suh2022continuous}.

We present an analogous Lyapunov analysis for the continuous-time ODE.

\begin{thm}\label{thm:continuous-time-lyap}
Fix $T>0$, and let $f\colon\mathbb{R}^{d} \to \mathbb{R}$ be an $L$-smooth convex
function with a minimizer $x_\star$.  Any trajectory
$X\in C^{2}([0,T];\mathbb{R}^{d})$ satisfying
$\ddot{X}+\gamma(t)\dot{X}+2\nabla f(X)=0$ on $(0,T)$, with
$X(0)=x_0$ and $\dot{X}(0)=0$, has the following bounds on the endpoint $X(T)$:
\begin{align*}
     & \norm{\nabla f(X(T))}^{2}
    \leq
    \frac{\varpi^{4}}{T^{4}}\norm{x_{0}-x_{\star}}^{2},\\
     & f(X(T))-f(x_{\star})
    \leq
    \frac{\varpi^{2}}{2T^{2}}\norm{x_{0}-x_{\star}}^{2}.
\end{align*}
\end{thm}
\begin{proof}[Proof sketch]
    We denote $x_{T}\triangleq X(T)$, set
    $Z(t)\triangleq\sigma(t)^3\dot X(t)/4$ for $t\in(0,T)$, and define
    the continuous-time Lyapunov function $\mathcal V(t)$ on $(0,T)$ by
    \begin{align*}
        \mathcal V(t) &\triangleq \frac{1-\rho^2}{2\rho}\left(f(X(t))-f(x_\star)\right) - \frac{(1-\rho)^2}{2\rho}\left(f(x_T)-f(x_\star)\right)\\
        &\quad{}+\frac{\varpi^2}{2T^2}\left(\norm{X(t)-x_\star+\frac{1+\rho^2}{1-\rho^2}\frac{T}{\varpi}Z(t)}^2-\norm{x_{T}-x_\star+\frac{T}{\varpi}Z(t)}^2\right).
    \end{align*}
    The derivative can be expressed as the following sum of Bregman divergences $D_{f}(x,y)\triangleq f(x)-f(y)-\left\langle x-y, \nabla f(y)\right\rangle$, thus is nonpositive:
    \begin{equation*}
        \frac{d}{dt}\mathcal V(t)=-\frac{\varpi}{T}\left(\sigma\rho D_{f}\left(x_\star,X(t)\right)+\frac12\sigma^3 D_{f}\left(x_{T},X(t)\right)\right)\leq 0.
    \end{equation*}
    The limits of the endpoints $t\downarrow 0$, $t\uparrow T$ are
    \begin{equation*}
    \begin{aligned}
        \mathcal V(0+)&=\frac{\varpi^2}{2T^2}\left(\norm{x_0-x_\star}^2-\norm{x_{T}-x_\star}^2\right),\\
        \mathcal V(T-)&=\frac{T^2}{2\varpi^2}\norm{\nabla f(x_T)}^2+f(x_T)-f(x_\star),
    \end{aligned}
    \end{equation*}
    and the bounds follow by chaining $\left(\varpi^2/(2T^2)\right)\norm{x_0-x_\star}^2 \geq \mathcal V(0+) \geq \mathcal V(T-)$ and $\mathcal V(T-) \geq \left( T^2 / (2\varpi^2)\right)\norm{\nabla f(x_T)}^2$ for the gradient bound and $\mathcal V(T-) \geq f(x_T)-f(x_\star)$ for the function value bound.
\end{proof}

\subsubsection{Lean formalization}\label{sec:lemniscate-lean}

\begin{figure}[t]
\centering
\begin{mdframed}[linewidth=0.7pt, roundcorner=3.5pt, innerleftmargin=8pt,
  innerrightmargin=8pt, innertopmargin=6pt, innerbottommargin=6pt]
\footnotesize\ttfamily
Using the \textsf{lean-verify} skill, formalize and verify the lemniscate acceleration results that are
located in the \texttt{ResearchLog} folder. Targets for the verification are the coefficient construction,
the Lyapunov identities,
the convergence theorem, the function-value bound, the continuous-time theorem.
State the results in the
manuscript's Euclidean setting or a generalization of it. 
Constraints are Lean 4.32.0 with mathlib 4.32.0, pinned by commit; no
\textup{\texttt{sorry}}, \textup{\texttt{axiom}}, \textup{\texttt{admit}},
or \textup{\texttt{unsafe}} anywhere in the development.

Deliverables are a Lake project passing \textup{\texttt{lake build}} and the
hygiene scan; wrappers \path{Challenge.lean}, \path{Solution.lean}, and
\path{config.json} naming the agreed public theorem surface for comparator
replay; build, axiom-audit, and replay logs under \path{artifacts/}.

Please ask me clarifying questions if anything is not clear to you.
\end{mdframed}
\caption{An illustrative Stage-3 prompt for the formal verification
of the lemniscate accelerated gradient method using the
\textsf{lean-verify} skill.}
\label{fig:stage3-prompt}
\end{figure}

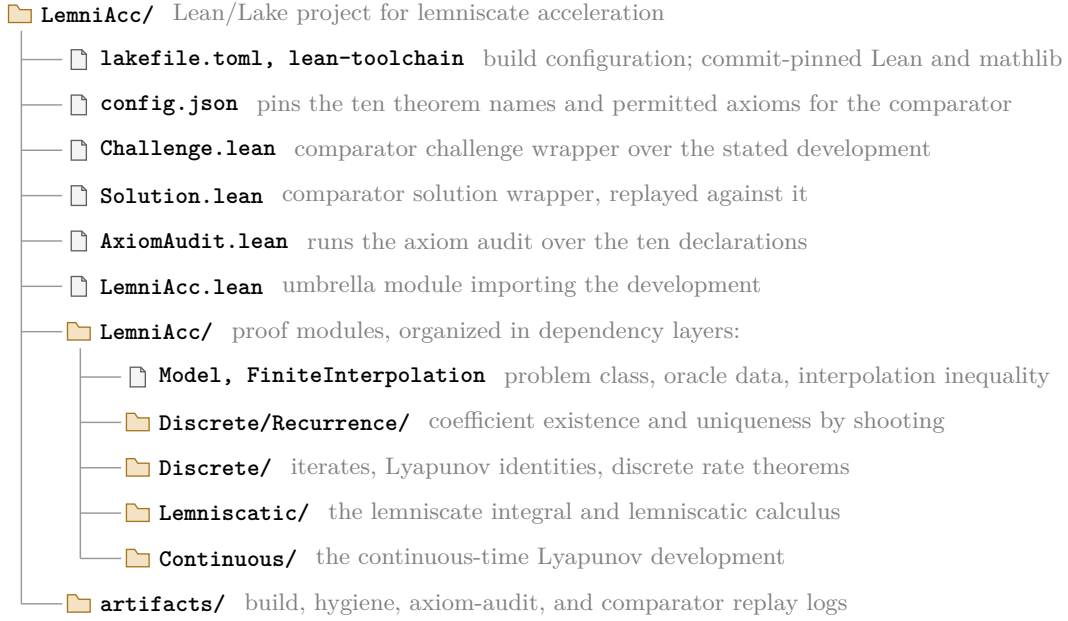
\begin{figure}[t]
\centering
\begin{adjustbox}{max width=\textwidth, center}
\definecolor{lfink}{RGB}{70,70,70}
\definecolor{lfamberfill}{RGB}{243,226,196}
\definecolor{lfamberline}{RGB}{166,118,29}
\definecolor{lfgrayfill}{RGB}{244,244,244}
\definecolor{lfgrayline}{RGB}{95,95,95}
\begin{tikzpicture}[
  fname/.style={font=\ttfamily\footnotesize\bfseries, text=black,
    inner sep=1.6pt, anchor=west},
  fnote/.style={font=\footnotesize, text=annotgray, inner sep=1.6pt,
    anchor=west},
  treeline/.style={draw=lfink!55, line width=0.55pt},
  pics/folder/.style={code={
    \fill[lfamberfill, draw=lfamberline, line width=0.5pt]
      (-0.15,0.04) -- (-0.15,0.12) -- (-0.04,0.12) -- (0.00,0.06)
      -- (0.15,0.06) -- (0.15,-0.11) -- (-0.15,-0.11) -- cycle;
  }},
  pics/lfile/.style={code={
    \fill[lfgrayfill, draw=lfgrayline, line width=0.5pt]
      (-0.10,-0.13) -- (-0.10,0.13) -- (0.04,0.13) -- (0.10,0.07)
      -- (0.10,-0.13) -- cycle;
    \draw[lfgrayline, line width=0.5pt] (0.04,0.13) -- (0.04,0.07)
      -- (0.10,0.07);
  }},
]
\newcommand{\treerow}[6]{  \pic at (#2,#3) {#4};
  \node[fname] (row#1) at (#2+0.22,#3) {#5};
  \node[fnote] at ($(row#1.east)+(0.12,0)$) {#6};
}
\treerow{a}{0}{0}{folder}{LemniAcc/}{Lean/Lake project for lemniscate
  acceleration}
\treerow{b}{0.78}{-0.60}{lfile}{lakefile.toml, lean-toolchain}{build
  configuration; commit-pinned Lean and mathlib}
\treerow{c}{0.78}{-1.20}{lfile}{config.json}{pins the ten theorem
  names and permitted axioms for the comparator}
\treerow{d}{0.78}{-1.80}{lfile}{Challenge.lean}{comparator challenge
  wrapper over the stated development}
\treerow{e}{0.78}{-2.40}{lfile}{Solution.lean}{comparator solution
  wrapper, replayed against it}
\treerow{f}{0.78}{-3.00}{lfile}{AxiomAudit.lean}{runs the axiom audit
  over the ten declarations}
\treerow{g}{0.78}{-3.60}{lfile}{LemniAcc.lean}{umbrella module importing
  the development}
\treerow{h}{0.78}{-4.20}{folder}{LemniAcc/}{proof modules, organized in
  dependency layers:}
\treerow{i}{1.56}{-4.80}{lfile}{Model, FiniteInterpolation}{problem
  class, oracle data, interpolation inequality}
\treerow{j}{1.56}{-5.40}{folder}{Discrete/Recurrence/}{coefficient
  existence and uniqueness by shooting}
\treerow{k}{1.56}{-6.00}{folder}{Discrete/}{iterates, Lyapunov
  identities, discrete rate theorems}
\treerow{l}{1.56}{-6.60}{folder}{Lemniscatic/}{the lemniscate integral
  and lemniscatic calculus}
\treerow{m}{1.56}{-7.20}{folder}{Continuous/}{the continuous-time
  Lyapunov development}
\treerow{n}{0.78}{-7.80}{folder}{artifacts/}{build, hygiene,
  axiom-audit, and comparator replay logs}
\draw[treeline] (0.02,-0.22) -- (0.02,-7.80);
\foreach \y in {-0.60,-1.20,-1.80,-2.40,-3.00,-3.60,-4.20,-7.80}
  {\draw[treeline] (0.02,\y) -- (0.56,\y);}
\draw[treeline] (0.80,-4.42) -- (0.80,-7.20);
\foreach \y in {-4.80,-5.40,-6.00,-6.60,-7.20}
  {\draw[treeline] (0.80,\y) -- (1.34,\y);}
\end{tikzpicture}\end{adjustbox}
\caption{File structure of the Lean project for the lemniscate accelerated gradient method
(module files grouped by folder). The wrappers \texttt{Challenge.lean},
\texttt{Solution.lean}, and \texttt{config.json} pin the ten public
theorems and the permitted axioms for the comparator, and
\texttt{artifacts/} records the build, audit, and replay logs.}
\label{fig:lemniacc-file-tree}
\end{figure}

\begin{table}[t]
\centering
\begin{tabular}{lc}
\toprule
Public theorem declarations & 10\\
Dependency nodes in the closure & 22\\
Lean source files & 22\\
Lines of Lean & 7{,}384\\
\texttt{lake build} jobs, native macOS & 8{,}687\\
Comparator replay time, native macOS & 109\,s\\
Comparator replay time, WSL2 real Landrun & 155\,s\\
Toolchain & Lean 4.32.0, mathlib 4.32.0, commit-pinned\\
Axioms used & \texttt{propext}, \texttt{Quot.sound}, \texttt{Classical.choice}\\
\bottomrule
\end{tabular}
\vspace{0.05in}
\caption{Summary statistics of the lemniscate acceleration Lean
project: it formalizes a fixed 22-node dependency closure of manuscript
claims and their supporting steps through ten public theorem
declarations, builds cleanly under the commit-pinned toolchain, and is
accepted by the comparator replay under exactly the three standard
axioms listed.}
\label{tab:lemniacc-project}
\end{table}

The convergence theory of this section is machine-checked in Lean~4 with
mathlib \cite{demoura2021lean,mathlib2020lean} by Stage~3 of
\textsf{AutoOPT} (Section~\ref{subsec:autoopt-stage-lean});
Figure~\ref{fig:stage3-prompt} shows an illustrative Stage-3 prompt.
The resulting Lean project proves ten public
theorem declarations, stating the smooth convex
inequality over the Euclidean space $\mathbb{R}^{d}$ and the
algorithmic and continuous-time results over an arbitrary complete real
inner-product space, which subsumes the manuscript's $\mathbb{R}^{d}$
setting. The declaration
\path{LemniAcc.gradient_rate} is the Lean counterpart of
Theorem~\ref{thm:lemniscate-convergence}, proving both displayed bounds;
\path{LemniAcc.functionValue_rate} proves both bounds of
Corollary~\ref{cor:lyap-function-value}; and
\path{LemniAcc.continuousTime_lyapunov}
proves both endpoint bounds
of Theorem~\ref{thm:continuous-time-lyap}, with the $C^{2}$ trajectory
of the ODE supplied as a hypothesis exactly as stated there. The supporting declarations follow the paper's proofs:
\path{recurrence_existsUnique} is
Lemma~\ref{lem:lemniscate-recurrence}; \path{omega_bounds} gives the
strict bounds $(N+1)^{2}/\varpi^{2}<\Omega_{N}<(N+1)^{2}$ of
Corollary~\ref{cor:app-omega-bounds}; \path{finite_interpolation}
establishes \eqref{eq:F_0L_formula} that the
proofs use; \path{iterates} constructs the trajectory of
\eqref{eq:lemniscate-acceleration}; \path{lyapunov_terminal} and
\path{lyapunov_decrement} are Lemmas~\ref{lem:lyap-terminal}
and~\ref{lem:lyap-closed-decrement}; and
\path{Lemniscatic.sigma_identities} combines the lemniscatic calculus
of Lemma~\ref{lem:lemniscatic-calculus} with the
$\sigma$-identities of Lemma~\ref{lem:sigma-identities}. As a
representative statement, \path{LemniAcc.gradient_rate} reads as
follows, where \texttt{E} is an arbitrary complete real inner-product
space:
\begin{center}
\begin{minipage}{0.92\textwidth}
\ttfamily\footnotesize
theorem gradient\_rate\\
\hspace*{2em}(N : Nat) (hN : 1 $\leq$ N)\\
\hspace*{2em}(M : SmoothConvexModel E) (x0 xStar : E)\\
\hspace*{2em}(hxStar : M.IsMinimizer xStar) :\\
\hspace*{2em}$\Vert$M.grad (Discrete.canonicalX N M x0 N)$\Vert$ \^{} 2 $\leq$\\
\hspace*{4em}((M.L : $\mathbb{R}$) \^{} 2 / omega N \^{} 2) * $\Vert$x0 - xStar$\Vert$ \^{} 2\\
\hspace*{3em}$\wedge$\\
\hspace*{3em}((M.L : $\mathbb{R}$) \^{} 2 / omega N \^{} 2) * $\Vert$x0 - xStar$\Vert$ \^{} 2 $\leq$\\
\hspace*{4em}(Lemniscatic.varpi \^{} 4 * (M.L : $\mathbb{R}$) \^{} 2 *\\
\hspace*{6em}$\Vert$x0 - xStar$\Vert$ \^{} 2) /\\
\hspace*{5em}(((N + 1 : Nat) : $\mathbb{R}$) \^{} 4)
\end{minipage}
\end{center}

Figure~\ref{fig:lemniacc-file-tree} shows the layout of the project,
and Table~\ref{tab:lemniacc-project} reports its summary statistics.
The project comprises 22 Lean source files and 7{,}384 lines of
Lean, organized in dependency layers that parallel the paper: the
problem class and the interpolation inequality, the recurrence
construction, the discrete Lyapunov development, and the lemniscatic
calculus supporting the continuous-time theorem. It builds on mathlib's
theories of inner-product and Euclidean spaces, differentiability, and
one-variable calculus; in particular, the shooting argument of
Appendix~\ref{sec:app-recurrence-existence} rests on the intermediate
value theorem, and the lemniscatic calculus of
Appendix~\ref{sec:app-continuous-time} on interval integration.

The project passes \texttt{lake build} under the commit-pinned toolchain
Lean 4.32.0 with mathlib 4.32.0 (8{,}687 build jobs on native macOS), a
hygiene scan
confirms that no \texttt{sorry}, \texttt{axiom}, \texttt{admit}, or
\texttt{unsafe} appears anywhere in the development, and an axiom audit
reports that every public declaration depends on exactly the three
standard axioms \texttt{propext}, \texttt{Quot.sound}, and
\texttt{Classical.choice}. The comparator then independently certified
the ten public theorems: the approved wrappers \texttt{Challenge.lean},
\texttt{Solution.lean}, and \texttt{config.json} were replayed end to
end, rebuilding both wrappers, exporting their environments, and
re-checking the solution with the Lean default kernel, in 109 seconds;
in this first project both wrappers import the proof development
itself, and the ITEM-f project of Section~\ref{sec:itemf-lean}
strengthens the wrapper design with a proof-free statement layer.
Each project was certified twice, first replayed on native macOS, and then in a 
Linux environment with Landlock isolation and systemd restrictions.
The project was accepted by the Lean-kernel of the sandboxed environment 
in 155 seconds, thus carries the evidence label 
\path{Lean+comparator verified (WSL2 real Landrun; systemd address-family restriction)}.
The projects and their verification artifacts accompany the paper.

\subsection{ITEM-f for efficient function-value ratio contraction}\label{sec:itemf}

Consider the unconstrained minimization problem
\begin{equation}
\begin{array}{ll}
\underset{x\in\mathbb{R}^{d}}{\mbox{minimize}} & f(x),\end{array}\tag{\ensuremath{\mathcal{P}'}}\label{eq:itemf-opt}
\end{equation}
where $f\colon\mathbb{R}^{d}\to\mathbb{R}$ is an $L$-smooth and $\mu$-strongly
convex function with $0<\mu<L$.
Strong convexity guarantees that $f$ has a unique minimizer $x_{\star}$. Denote $f_{\star}\triangleq f(x_{\star})$.

\subsubsection*{Algorithm statement}

The method \eqref{eq:item-f-momentum} is
\begin{equation}
	\tag{ITEM-f}\begin{aligned}x_{k+1}={} & x^{+}_{k}+\dfrac{a_{k}\phi_{N-k-1}}{(1-q)a_{k+1}\phi_{N-k}}(x^{+}_{k}-x^{+}_{k-1})+\dfrac{\phi_{N-k-1}}{\phi_{N-k}}(x^{+}_{k}-x_{k}),\end{aligned}
	\label{eq:item-f-momentum}
\end{equation}
for $k=0,1,2,\dots,N-1$, where $N$ is a prespecified iteration count,
$x_{0}\in\mathbb{R}^{d}$ is the starting point, $x^{+}_{-1}=x_{0}$
by convention, $x^{+}_{k}=x_{k}-\frac{1}{L}\nabla f(x_{k})$, $q=\mu/L$,
$\phi_{0}=\sqrt{1-q}$, $\phi_{N}=\left((1-q)\Upsilon_{N}+a_{N}\right)/\sqrt{1-q}$,
$\phi_{k}=\sqrt{2\Upsilon_{N}a_{k}-q}$ for $1\le k\le N-1$ and $a_{0},a_{1},\dots,a_{N}$
and $\Upsilon_{N}$ are positive constants that we define in Lemma~\ref{lem:itemf-construction}.

Taylor and Drori
\cite[Appendix~E]{taylor2021optimal}
reported numerically optimized step sizes of ITEM-f for $N=1,\dots,5$ in the setting $L=1$ and $\mu=0.1$, but did not assign a name to the resulting method.
They also derived a closed-form method with a convergence analysis, called the Information-Theoretic Exact Method (ITEM), for a related performance criterion.
We use the name \emph{ITEM-f} to acknowledge their numerical construction while distinguishing it from ITEM.
The optimization objectives underlying ITEM and ITEM-f are discussed further in Section~\ref{sec:itemf-autooptdesign}.

\subsubsection*{Construction of the coefficients}

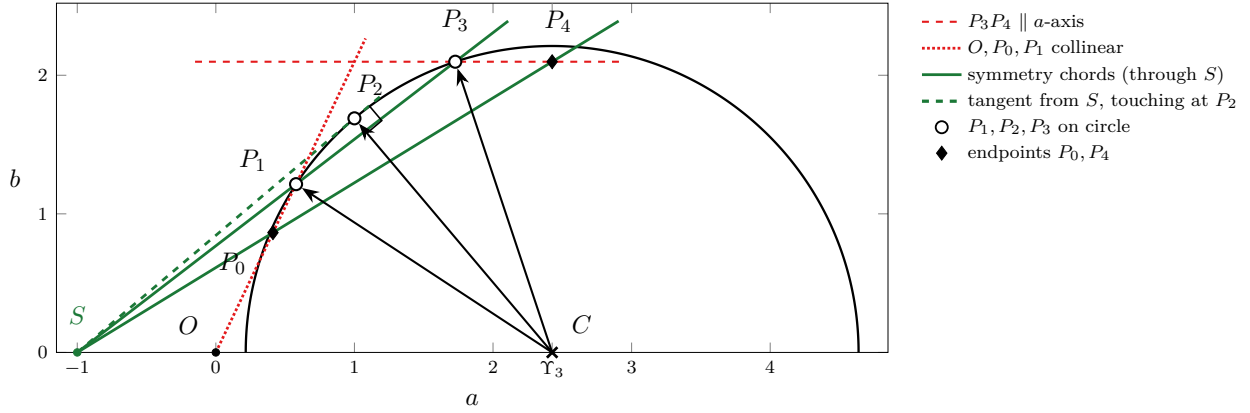
\begin{figure}[t]
\centering
\begin{adjustbox}{max width=\textwidth}
\definecolor{itemfblue}{RGB}{55,126,184}\definecolor{itemfred}{RGB}{228,26,28}\definecolor{itemfdgreen}{RGB}{27,120,55}\begin{tikzpicture}
\begin{axis}[
    width=0.86\textwidth,
    axis equal image,
    xmin=-1.15, xmax=4.85,
    ymin=0, ymax=2.52,
    clip=false,
    xlabel={$a$},
    ylabel={$b$},
    ylabel style={rotate=-90},
    label style={font=\normalsize},
    tick label style={font=\scriptsize},
    xtick={-1,0,1,2,3,4},
    ytick={0,1,2},
    extra x ticks={2.426697},
    extra x tick labels={$\Upsilon_{3}$},
    legend pos=outer north east,
    legend style={draw=none, font=\scriptsize},
    legend cell align={left},
]
\addlegendimage{itemfred, line width=0.9pt, dashed}
\addlegendentry{$P_3 P_4$ $\parallel a$-axis}
\addlegendimage{itemfred, line width=1.2pt, densely dotted}
\addlegendentry{$O,P_0,P_1$ collinear}
\addlegendimage{itemfdgreen, line width=1.2pt}
\addlegendentry{symmetry chords (through $S$)}
\addlegendimage{itemfdgreen, line width=1.4pt, dashed}
\addlegendentry{tangent from $S$, touching at $P_2$}
\addlegendimage{only marks, mark=*, mark options={fill=white, draw=black, line width=0.8pt}, mark size=2.6pt}
\addlegendentry{$P_1,P_2,P_3$ on circle}
\addlegendimage{only marks, mark=diamond*, black, mark size=3pt}
\addlegendentry{endpoints $P_0,P_4$}
\addplot[domain=0:180, samples=181, smooth, black, line width=1.0pt, forget plot]
    ({2.426697+2.211076*cos(x)}, {2.211076*sin(x)});
\draw[itemfdgreen, line width=1.2pt] (-1,0) -- (2.109342,2.391277);
\draw[itemfdgreen, line width=1.2pt] (-1,0) -- (2.906434,2.391277);
\draw[itemfdgreen, line width=1.2pt, dashed] (-1,0) -- (1.200000,1.858119);
\draw[black, line width=0.7pt] (1.106956,1.779534) -- (1.197291,1.672578)
    -- (1.090335,1.582242);
\draw[itemfred, line width=1.2pt, densely dotted] (0,0) -- (1.080000,2.265446);
\draw[itemfred, line width=0.9pt, dashed] (-0.15,2.097611) -- (2.93,2.097611);
\draw[black, line width=0.9pt, -{Stealth[length=2.6mm]}, shorten >=3pt]
    (2.426697,0) -- (0.578874,1.214251);
\draw[black, line width=0.9pt, -{Stealth[length=2.6mm]}, shorten >=3pt]
    (2.426697,0) -- (1.000000,1.689199);
\draw[black, line width=0.9pt, -{Stealth[length=2.6mm]}, shorten >=3pt]
    (2.426697,0) -- (1.727493,2.097611);
\addplot[only marks, mark=x, black, mark size=3.2pt,
    mark options={line width=1.2pt}, forget plot] coordinates {(2.426697,0)};
\addplot[only marks, mark=*, black, mark size=1.6pt, forget plot] coordinates {(0,0)};
\addplot[only marks, mark=*, itemfdgreen, mark size=1.6pt, forget plot] coordinates {(-1,0)};
\addplot[only marks, mark=diamond*, black, mark size=3pt, forget plot]
    coordinates {(0.412083,0.864389) (2.426697,2.097611)};
\addplot[only marks, mark=*, mark options={fill=white, draw=black, line width=0.9pt},
    mark size=2.6pt, forget plot]
    coordinates {(0.578874,1.214251) (1.000000,1.689199) (1.727493,2.097611)};
\node[anchor=east]       at (0.282,0.644)  {$P_0$};
\node[anchor=south east] at (0.429,1.244)  {$P_1$};
\node[anchor=south west] at (0.950,1.790)  {$P_2$};
\node[anchor=south]      at (1.727,2.258)  {$P_3$};
\node[anchor=south]      at (2.470,2.250)  {$P_4$};
\node[anchor=south east] at (-0.06,0.07)   {$O$};
\node[anchor=south west] at (2.505,0.075)  {$C$};
\node[itemfdgreen, anchor=south] at (-1.0,0.14) {$S$};
\end{axis}
\end{tikzpicture}\end{adjustbox}
\caption{The planar construction underlying Lemma~\ref{lem:itemf-construction}, shown for $N=3$, $q=0.1$.
The intermediate points $P_{1},\dots,P_{N}$ ($\circ$) lie on the circle of radius
$R_{N}=\sqrt{\Upsilon_{N}^{2}-1}$ centered at $C=(\Upsilon_{N},0)$; the endpoints
$P_{0}$ and $P_{N+1}$ ($\diamond$) lie strictly inside; successive points satisfy
$(P_k-C)\cdot(P_{k+1}-C)=R_{N}^{2}(1-q\,a_{k+1}/\Upsilon_{N})$. 
\emph{Shooting (red).} 
The construction is run backward from the horizontal final edge $b_{N}=b_{N+1}=\sqrt{1-q}\,R_{N}$
(dashed); $\Upsilon_{N}$ is the value for which $O$, $P_{0}$, $P_{1}$ are collinear in this order (dotted).
\emph{Symmetry (green).}
The map $T(a,b)=(1/a,b/a)$ satisfies $T(P_k)=P_{N+1-k}$; hence every chord
$P_kP_{N+1-k}$ passes through $S=(-1,0)$. For odd $N$, the tangent from $S$
touches the circle at the fixed point $P_{(N+1)/2}$.}
\label{fig:itemf-points}
\end{figure}

To fully define \ref{eq:item-f-momentum}, it remains to
specify the coefficients \(a_0,a_1,\ldots,a_N\) and \(\Upsilon_N\).
Lemma~\ref{lem:itemf-construction} formally defines these coefficients,
and Figure~\ref{fig:itemf-points} illustrates the underlying geometric
construction.

\begin{lem}\label{lem:itemf-construction}
For any integer $N\geq1$ and $0<q<1$, there exist a unique
real $\Upsilon_{N}>1$, unique reals
\[
1/\Upsilon_{N}=a_{0}<a_{1}<\cdots<a_{N}<a_{N+1}=\Upsilon_{N},
\]
and unique positive reals $b_{0},b_{1},\dots,b_{N},b_{N+1}$
satisfying the following conditions.
Write $C=(\Upsilon_{N},0)$ and $R_{N}=\sqrt{\Upsilon_{N}^{2}-1}$ and
write $P_{k}=(a_{k},b_{k})\in\mathbb{R}^2$ for $0\leq k\leq N+1$. The requirements are:
\begin{equation}
\begin{aligned} & \norm{P_k-C}^{2}=R_{N}^{2},\quad1\leq k\leq N,\\
 & (P_k-C)\cdot(P_{k+1}-C)=R_{N}^{2}\Big(1-\tfrac{q\,a_{k+1}}{\Upsilon_{N}}\Big),\quad0\leq k\leq N,\\
 & P_{0}=\Big(\tfrac{1}{\Upsilon_{N}},\tfrac{\sqrt{1-q}\,R_{N}}{\Upsilon_{N}}\Big),\\
 & P_{N+1}=\big(\Upsilon_{N},\sqrt{1-q}\,R_{N}\big).
\end{aligned}
\label{eq:itemf-construction-conditions}
\end{equation}
\end{lem}
The proof of Lemma~\ref{lem:itemf-construction} is provided
in Appendix~\ref{sec:app-itemf-existence}.
To clarify, $a_{N+1}$ is not used in the definition of \ref{eq:item-f-momentum} but we defined it for convenience. We can compute the numerical values of $a_0,\dots,a_N$ and \(\Upsilon_N\) again by a one-dimensional bisection, as discussed in Appendix~\ref{sec:app-itemf-existence}.

The geometric construction carries a symmetry analogous to the lemniscate recurrence in Section~\ref{sec:lemniscate}: the map $T(a,b)\triangleq(1/a,\,b/a)$, defined for a>0, maps the circle $a^{2}+b^{2}-2\Upsilon_{N}a+1=0$ to itself and reverses the construction, $T(P_{k})=P_{N+1-k}$ for $0\leq k\leq N+1$, so that $a_{k}a_{N+1-k}=1$ and $a_{k}b_{N+1-k}=b_{k}$ (Lemma~\ref{lem:itemf-involution-symmetry} in Appendix~\ref{sec:app-itemf-existence}); in Figure~\ref{fig:itemf-points} this symmetry is visible as the chords $P_{k}P_{N+1-k}$ passing through the point $S\triangleq(-1,0)$. This symmetry supplies the identities $a_{k+1}a_{N-k}=1$ and $a_{1}a_{N}=1$ used in the convergence proof of Section~\ref{sec:itemf-lyapunov}.

\subsubsection*{Convergence theorem}

\begin{thm}\label{thm:itemf-convergence}
Let $N\geq1$ and $d\geq1$ be integers, and let
$f\colon\mathbb{R}^{d}\to\mathbb{R}$ be $L$-smooth and $\mu$-strongly convex
with $q=\mu/L\in(0,1)$, let $\Upsilon_{N}$ and
$a_{0},\dots,a_{N}$ be as in Lemma~\ref{lem:itemf-construction}, and
let $x_{\star}=\operatorname*{argmin}f$. Then, for any
starting point $x_{0}\in\mathbb{R}^{d}$, the
final iterate $x_{N}$ of~\eqref{eq:item-f-momentum} satisfies
\begin{equation}\label{eq:itemf-rate}
f(x_{N})-f_{\star}\;\leq\;\frac{1}{\Upsilon_{N}^{2}}\,\big(f(x_{0})-f_{\star}\big)
\;\leq\;4 (1-\sqrt q)^{2N}\,\big(f(x_{0})-f_{\star}\big).
\end{equation}
\end{thm}
We prove Theorem~\ref{thm:itemf-convergence} in the following section.

\subsubsection{Convergence proof}\label{sec:itemf-lyapunov}
In this section, we prove
Theorem~\ref{thm:itemf-convergence}. Throughout,
$f\colon\mathbb{R}^{d}\to\mathbb{R}$ is an $L$-smooth and
$\mu$-strongly convex function with the minimizer $x_{\star}$,
$q=\mu/L$, and $x_{k}$ denotes the iterates of
\eqref{eq:item-f-momentum} with the convention $x_{-1}^{+}=x_{0}$.
The argument has three steps. We first define a Lyapunov sequence
$\mathcal V_{1},\ldots,\mathcal V_{N}$ along the iterates, whose
coefficients are read off from the multipliers of the analytic dual
certificate found in Stage 2 of the pipeline, and compress it into a
compact form built from transformed interpolation gaps and vectors
$W_{k}\in\mathbb{R}^{2d}$. We then compare the two endpoints:
$\mathcal V_{1}$ is at most
$\left(f(x_{0})-f_{\star}\right)/\Upsilon_{N}$, and $\mathcal V_{N}$
dominates $\Upsilon_{N}\left(f(x_{N})-f_{\star}\right)$. Finally,
Lemma~\ref{lem:itemf-lyap-monotone} shows that the sequence is
nonincreasing, and chaining the three facts proves
Theorem~\ref{thm:itemf-convergence}. The norm identities that drive
these computations (Lemma~\ref{lem:itemf-lyap-norms}) are stated
below and proved in Appendix~\ref{sec:app-itemf-lyap-norms}.

For $1\leq k\leq N-1$, we define the Lyapunov sequence
\begin{align*}
\mathcal{V}_{k} &\triangleq a_{k}\left(f(x_{k-1})-f_{\star}-\frac{\mu}{2}\norm{x_{k-1}-x_{\star}}^{2}-\frac{1}{2(L-\mu)}\norm{\nabla f(x_{k-1})-\mu(x_{k-1}-x_{\star})}^{2}\right)\\
 &\quad {}+\frac{\mu(1-q)a^{2}_{k+1}}{2\Upsilon_N}\norm{x_{k}-x_{\star}}^{2}\\
 &\quad {}+\frac{\mu}{2\Upsilon_N(1-q)(1-c^{2}_{k})}
 \norm{(1-q)c_ka_{k+1}(x_k-x_\star)-a_k(x_{k-1}^+-x_\star)}^{2},
\end{align*}
and for the final iteration $N$, define
\begin{align*}
\mathcal{V}_{N} &\triangleq a_{N}\left(f(x_{N-1})-f_{\star}-\frac{\mu}{2}\norm{x_{N-1}-x_{\star}}^{2}-\frac{1}{2(L-\mu)}\norm{\nabla f(x_{N-1})-\mu(x_{N-1}-x_{\star})}^{2}\right)\\
 &\quad {}+\frac{\mu\Upsilon_{N}}{2}\norm{x_{N}-x_{\star}}^{2}\\
 &\quad {}+\frac{L}{2\Upsilon_N(1-q)}
 \norm{(1-q)\Upsilon_N(x_N-x_\star)-a_N(x_{N-1}^+-x_\star)}^{2},
\end{align*}
where we define $c_{k}\triangleq1-(qa_{k+1})/\Upsilon_{N}$ for $1\leq k\leq N$.

Following the approach of \cite{taylor2021optimal}, we introduce the convex $(L-\mu)$-smooth function and its gradient:
\begin{align*}
 & \widetilde{f}(x)\triangleq f(x)-f_{\star}-\frac{\mu}{2}\norm{x-x_{\star}}^{2},\\
 & \widetilde{g}_{k}\triangleq\nabla\widetilde{f}(x_{k})=\nabla f(x_{k})-\mu(x_{k}-x_{\star}).
\end{align*}
This makes $\widetilde g_{\star}=0$ and $\widetilde f(x_{\star})=0$.
For $i,j\in I_{N}^{\star}$, denote its interpolation gap by
\begin{align*}
 \mathcal G_{i,j}
 \triangleq{}&\widetilde f(x_{i})-\widetilde f(x_{j})
 -\langle \widetilde g_{j},x_{i}-x_{j}\rangle
 -\frac{1}{2(L-\mu)}\norm{\widetilde g_{i}-\widetilde g_{j}}^{2}\geq0,
\end{align*}
which is nonnegative by \eqref{eq:F_0L_formula} applied to $\widetilde f$.
From the relation
$x_i^+-x_\star=(1-q)(x_i-x_\star)-L^{-1}\widetilde g_i$
(which follows from $x_{i}^{+}=x_{i}-L^{-1}\nabla f(x_i)$,
$\nabla f(x_i)=\widetilde g_{i}+\mu(x_{i}-x_{\star})$, and $\mu=qL$), the
three-point identity follows:
\begin{equation}
 \mathcal G_{i,\star}
 -\frac{1}{1-q}\left\langle x_i^+-x_\star,\widetilde g_j\right\rangle
 =\mathcal G_{i,j}-\mathcal G_{\star,j}.
\label{eq:itemf-lyap-three-point}
\end{equation}

To simplify the analysis of \eqref{eq:item-f-momentum}, for $1 \leq k \leq N$, we define
\begin{equation}
 W_{k}\triangleq
 \begin{bmatrix}
  a_k(x_{k-1}^+-x_\star)\\[2mm]
  \dfrac{(1-q)a_{k+1}(x_k-x_\star)
  -c_ka_k(x_{k-1}^+-x_\star)}{s_k}
 \end{bmatrix}\in\mathbb{R}^{2d},
\label{eq:itemf-lyap-W-def}
\end{equation}
where we define $s_k\triangleq\sqrt q\,a_{k+1}\phi_{N-k}/\Upsilon_N$.
At $k=N$, these definitions give the boundary values
$c_{N}=1-q$ and $s_{N}=\sqrt{q(1-q)}$, using $a_{N+1}=\Upsilon_{N}$
and $\phi_{0}=\sqrt{1-q}$.

For $1\leq k\leq N-1$, the relation $s_k^2+c_k^2=1$ can be proved using
the symmetry identity $a_{k+1}a_{N-k}=1$ of the construction:
\begin{align*}
s^{2}_{k}+c^{2}_{k} & =\frac{qa^{2}_{k+1}}{\Upsilon^{2}_{N}}(2\Upsilon_{N}a_{N-k}-q)+\left(1-\frac{qa_{k+1}}{\Upsilon_{N}}\right)^{2}=\frac{qa_{k+1}}{\Upsilon^{2}_{N}}(2\Upsilon_{N}-qa_{k+1})+\left(1-\frac{qa_{k+1}}{\Upsilon_{N}}\right)^{2}=1.
\end{align*}

For $1\leq k\leq N$, define
\begin{equation*}
  \mathcal{O}_k\triangleq
  \begin{bmatrix}
    c_k I_d&s_k I_d\\
    -s_k I_d&c_k I_d
  \end{bmatrix}
  \in\mathbb R^{2d\times2d}.
\end{equation*}
For $1\leq k\leq N-1$, the identity $s_k^2+c_k^2=1$ makes
$\mathcal O_k$ orthogonal, and multiplying it by $W_k$ gives
\begin{equation}
 \mathcal{O}_kW_k=
 \begin{bmatrix}
  (1-q)a_{k+1}(x_k-x_\star)\\[2mm]
  \dfrac{(1-q)c_ka_{k+1}(x_k-x_\star)
  -a_k(x_{k-1}^+-x_\star)}{s_k}
 \end{bmatrix},\quad 1\leq k\leq N-1.
\label{eq:itemf-lyap-W-rotation}
\end{equation}
Since $\mathcal{O}_k$ is orthogonal for $k<N$, this gives the two
quadratic terms in $\mathcal V_k$. At $k=N$, where $s_N^2+c_N^2=1-q$
and the terminal block $\mathcal{O}_N$ is not orthogonal, the quadratic terms in the final
$\mathcal V_N$ follow instead from the direct expansion in the third
identity of Lemma~\ref{lem:itemf-lyap-norms}, using $c_N=1-q$,
$a_{N+1}=\Upsilon_N$, and $s_N=\sqrt{q(1-q)}$. Thus the Lyapunov sequence has the compact form
\begin{equation}
\label{eq:itemf-lyap-compact}
 \mathcal V_{k}
 =a_{k}\mathcal G_{k-1,\star}
 +\frac{\mu}{2(1-q)\Upsilon_N}\norm{W_{k}}^{2},
 \quad 1 \leq k \leq N,
\end{equation}
where the norm is the Euclidean norm on $\mathbb R^{2d}$.

\begin{lem}[Norm identities]\label{lem:itemf-lyap-norms}
The squared norms $\norm{W_k}^{2}$ for the iterates of \eqref{eq:item-f-momentum} satisfy the following identities:
\begin{align*}
	& \norm{W_{1}}^{2}=(1-q)\norm{x_{0}-x_{\star}}^{2}-\frac{2\Upsilon_{N}}{\mu}(a_{1}-a_{0})\left\langle x^{+}_{0}-x_{\star},\widetilde{g}_{0}\right\rangle +\frac{\Upsilon_{N}a_{0}}{\mu L}\norm{\widetilde{g}_{0}}^{2},\\
	& \norm{W_{k+1}}^{2}=\norm{W_{k}-\frac{\Upsilon_{N}}{\mu}\begin{bmatrix}(c_{k}-1)\widetilde{g}_{k}\\
			s_{k}\widetilde{g}_{k}
	\end{bmatrix}}^{2},\quad1\leq k\leq N-1,\\
	& \norm{W_{N}}^{2}=(1-q)\Upsilon^{2}_{N}\norm{x_{N}-x_{\star}}^{2}+\frac{1}{q}\norm{(1-q)\Upsilon_{N}(x_{N}-x_{\star})-a_{N}(x^{+}_{N-1}-x_{\star})}^{2}.
\end{align*}
\end{lem}
The proof is tedious but straightforward algebra on
the coordinate relations of the construction; it is deferred to
Appendix~\ref{sec:app-itemf-lyap-norms}.

\begin{lem}[Monotonicity]
\label{lem:itemf-lyap-monotone}
For $1 \leq k \leq N-1$,
\begin{equation*}
 \mathcal V_{k}-\mathcal V_{k+1}
 =a_{k}\mathcal G_{k-1,k}
 +(a_{k+1}-a_{k})\mathcal G_{\star,k}
 \geq0.
\end{equation*}
Consequently, $\mathcal V_1\geq\mathcal V_2\geq\cdots\geq\mathcal V_N$.
\end{lem}

\begin{proof}
From the definition of $W_k$ and
$c_{k}=1-qa_{k+1}/\Upsilon_{N}$, we have
\begin{align*}
 \left\langle
 \begin{bmatrix}(c_k-1)\widetilde g_k\\ s_k\widetilde g_k\end{bmatrix},
 W_k\right\rangle_{\mathbb R^{2d}}
 &=
  \left\langle (1-q)a_{k+1}(x_{k}-x_{\star})
  -a_{k}(x_{k-1}^+-x_\star),\widetilde g_k
  \right\rangle,\\
 \norm{\begin{bmatrix}(c_k-1)\widetilde g_k\\ s_k\widetilde g_k\end{bmatrix}}^{2}
 &=\left((1-c_k)^2+s_k^2\right)\norm{\widetilde g_k}^{2}
 =2(1-c_k)\norm{\widetilde g_k}^{2}
 =\frac{2qa_{k+1}}{\Upsilon_N}\norm{\widetilde g_k}^{2}.
\end{align*}

Using the compact form \eqref{eq:itemf-lyap-compact} and expanding
the second identity of Lemma~\ref{lem:itemf-lyap-norms}, we obtain
\begin{align*}
 \mathcal V_{k}-\mathcal V_{k+1}
 ={}&a_k\mathcal G_{k-1,\star}
 -a_{k+1}\mathcal G_{k,\star}
 +\frac{\mu}{2(1-q)\Upsilon_N}
 \left(\norm{W_k}^{2}-\norm{W_{k+1}}^{2}\right)\\
 ={}&a_k\mathcal G_{k-1,\star}
 -a_{k+1}\mathcal G_{k,\star}
 \\
 &\quad
 {}+\frac{1}{1-q}\left\langle
  (1-q)a_{k+1}(x_k-x_\star)
    -a_k(x_{k-1}^+-x_\star),\widetilde g_k
  \right\rangle
 -\frac{a_{k+1}}{L(1-q)}\norm{\widetilde g_k}^{2}\\
 ={}&a_k\left(
  \mathcal G_{k-1,\star}
  -\frac{1}{1-q}\left\langle x_{k-1}^+-x_\star,\widetilde g_k\right\rangle
 \right)
 +a_{k+1}\mathcal G_{\star,k}\\
 ={}&a_k\left(\mathcal G_{k-1,k}-\mathcal G_{\star,k}\right)
 +a_{k+1}\mathcal G_{\star,k}\\
 ={}&a_k\mathcal G_{k-1,k}
 +(a_{k+1}-a_k)\mathcal G_{\star,k},
\end{align*}
where the second equality uses the above two identities and the third and fourth use
the three-point identity~\eqref{eq:itemf-lyap-three-point} at
$(i,j)=(k,k)$ and $(k-1,k)$, respectively.  Since $a_{k+1}\geq a_k>0$ and
$\mathcal G_{i,j}\geq0$, the last line is nonnegative.
\end{proof}

We now provide the proof of Theorem~\ref{thm:itemf-convergence}.
\begin{proof}[Proof of Theorem~\ref{thm:itemf-convergence}]
The relation $a_0=1/\Upsilon_N$, the first identity in
Lemma~\ref{lem:itemf-lyap-norms}, and the compact form
\eqref{eq:itemf-lyap-compact} of $\mathcal V_1$ give
\begin{align*}
 \frac{f(x_{0})-f_{\star}}{\Upsilon_N}-\mathcal V_{1}
 ={}&(a_{1}-a_{0})\left(
  \frac{1}{1-q}\left\langle x_0^+-x_\star,\widetilde g_0\right\rangle
  -\mathcal G_{0,\star}
 \right)\\
 ={}&(a_{1}-a_{0})\mathcal G_{\star,0}\geq0,
\end{align*}
where the last equality applies \eqref{eq:itemf-lyap-three-point} at
$(i,j)=(0,0)$.

The relation $a_{N+1}=\Upsilon_N$, the third identity in
Lemma~\ref{lem:itemf-lyap-norms}, and the compact form
\eqref{eq:itemf-lyap-compact} of $\mathcal V_{N}$ give
\begin{align*}
 \mathcal V_{N}
 &-\Upsilon_{N}\bigl(f(x_{N})-f_{\star}\bigr)\\
 &{}-\frac{L}{2(1-q)\Upsilon_N}
  \norm{(1-q)\Upsilon_{N}(x_{N}-x_{\star})-a_{N}(x_{N-1}^{+}-x_{\star})
  -\frac{\Upsilon_{N}}{L}\widetilde g_{N}}^{2}\\
 ={}&a_N\left(
  \mathcal G_{N-1,\star}
  -\frac{1}{1-q}
   \left\langle x_{N-1}^+-x_\star,\widetilde g_N\right\rangle
 \right)
 +\Upsilon_N\mathcal G_{\star,N}\\
 ={}&a_N\left(\mathcal G_{N-1,N}-\mathcal G_{\star,N}\right)
 +\Upsilon_N\mathcal G_{\star,N}\\
 ={}&a_{N}\mathcal G_{N-1,N}
 +(\Upsilon_{N}-a_{N})\mathcal G_{\star,N}\geq0,
\end{align*}
where the second equality applies \eqref{eq:itemf-lyap-three-point} at
$(i,j)=(N-1,N)$.
The completed square is nonnegative, so
$\mathcal V_{N}\geq\Upsilon_{N}\left(f(x_{N})-f_{\star}\right)$;
combining the two endpoint bounds with
Lemma~\ref{lem:itemf-lyap-monotone} gives
\begin{equation*}
 \frac{f(x_{0})-f_{\star}}{\Upsilon_N}
 \geq\mathcal V_{1}\geq\cdots\geq\mathcal V_{N}
 \geq\Upsilon_{N}\bigl(f(x_{N})-f_{\star}\bigr).
\end{equation*}
The explicit relaxation in the second inequality then follows from Lemma~\ref{lem:itemf-explicit-rate}.
\end{proof}

\subsubsection{\textsf{AutoOPT} design: Optimal method for contracting function-value suboptimality}
\label{sec:itemf-autooptdesign}

The method \eqref{eq:item-f-momentum} was designed by \textsf{AutoOPT} as the optimal fixed-step first-order method for contracting the function-value suboptimality. Specifically, \textsf{AutoOPT} minimizes the contraction factor $C_N$ in the guarantee
\[
f(x_N)-f_\star
\leq
C_N\left(f(x_0)-f_\star\right)
\]
and attains the optimal value $C_N=1/\Upsilon_N^2$.
The pipeline ran as for lemniscate acceleration
(Section~\ref{sec:lemniscate-autooptdesign}): \textsf{bnb-pep-skill}
wrote the BnB-PEP formulation of this design problem, reproduced with
light editing in Appendix~\ref{sec:itemf-pep-design}, and, upon our
approval, solved the resulting QCQP numerically over a range of
horizons.

Taylor and Drori \cite[Appendix~E]{taylor2021optimal} previously considered the same optimization problem and numerically obtained ITEM-f for $N=1,\dots,5$ in the setting $L=1$ and $\mu=0.1$. From the Stage-1 formulation and its numerical
solutions, \textsf{frontier-llm-consult} identified a closed-form
expression for the algorithm valid for arbitrary $N$, $\mu$, and $L$,
together with an analytic feasible point of the dual certifying the
contraction factor $1/\Upsilon_N^2$, and then rewrote
the method, first obtained in the stepsize variables of the design
problem, as the momentum form \eqref{eq:item-f-momentum} in which this
paper states it; continued consultation converted the
certificate's multipliers into the Lyapunov proof of
Section~\ref{sec:itemf-lyapunov}, stated on that
momentum form, and Section~\ref{sec:itemf-lean}
discusses its Lean formalization.

\subsubsection{Continuous-time analysis}
For a fixed finite horizon $T>0$ and strong convexity parameter $\mu>0$,
let $\Upsilon>1$ be the unique real number such that, with
$s\triangleq\sqrt{1-\Upsilon^{-2}}$,
\begin{equation*}
 \sqrt{2\mu}\,T=\int_0^{\arcsin s}\frac{dx}{\sqrt{1-s\sin x}}.
\end{equation*}
We note the continuous-time limit concerns $L=N^2/T^2$, and with $q_N=\mu T^2 / N^2$, we have $ \Upsilon_N(q_N)\rightarrow\Upsilon$ as $N\rightarrow\infty$.
Define $a,b\colon[0,T]\to\mathbb R$ by
\begin{align*}
\dot a(t)&=b(t)\sqrt{\frac{2\mu a(t)}{\Upsilon}},\\
\dot b(t)&=(\Upsilon-a(t))\sqrt{\frac{2\mu a(t)}{\Upsilon}},\\
a(0)&=\frac{1}{\Upsilon},\\
b(0)&=\frac{\sqrt{\Upsilon^2-1}}{\Upsilon}.
\end{align*}

Then the continuous-time ODE of ITEM-f is
\begin{equation*}
\ddot X(t)+\frac{3 \dot a(t)}{2 a(t)}\dot X(t)+2\nabla f(X(t))=0,\quad 0\leq t \leq T.
\end{equation*}
\begin{remark}[Continuous-time convergence of ITEM-f]\label{rem:itemf-ct}
For an $L$-smooth and $\mu$-strongly convex function $f\colon \mathbb{R}^d \to \mathbb{R}$, any solution $X\in C^{2}([0,T];\mathbb{R}^{d})$ of the ODE above with $X(0)=x_0$ and $\dot X(0)=0$ satisfies
\begin{equation*}
f(X(T))-f_\star \leq \frac{1}{\Upsilon^2} \left(f(X(0))-f_\star\right).
\end{equation*}
This can be shown with a continuous-time Lyapunov function
\begin{equation*}
     \mathcal{V}(t)=a(t)\left(f(X(t))-f_\star-\frac{\mu}{2}\norm{X(t)-x_\star}^2\right)
     +\frac{\mu a(t)^2}{2\Upsilon}\norm{X(t)-x_\star}^2
     +\frac{1}{4a(t)}\norm{\dot a(t)(X(t)-x_\star)+a(t)\dot X(t)}^2,
\end{equation*}
by the monotonicity $\dot{\mathcal{V}}(t)\leq 0$ and the endpoints $\mathcal{V}(0)=\Upsilon^{-1}\left(f(X(0))-f_\star\right)$ and $\mathcal{V}(T)\geq \Upsilon\left(f(X(T))-f_\star\right)$.
We do not pursue a detailed proof or Lean verification of
this remark; both can be generated by running the \textsf{AutoOPT}
pipeline on this continuous-time claim.
\end{remark}

\subsubsection{Lean formalization}\label{sec:itemf-lean}

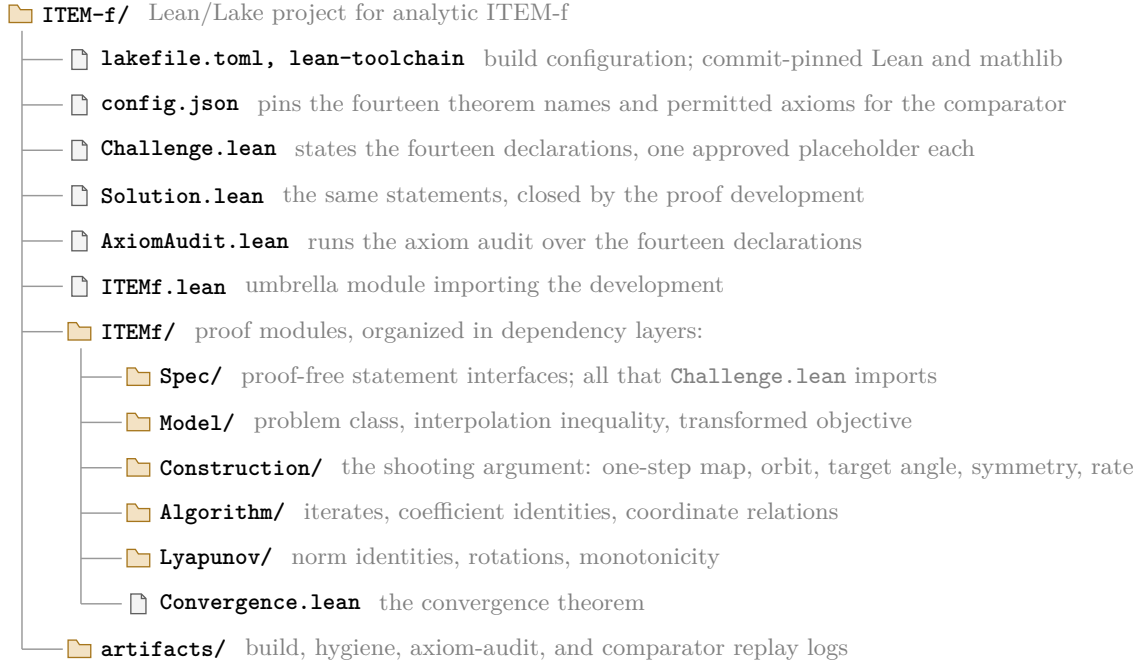
\begin{figure}[t]
\centering
\begin{adjustbox}{max width=\textwidth, center}
\definecolor{lfink}{RGB}{70,70,70}
\definecolor{lfamberfill}{RGB}{243,226,196}
\definecolor{lfamberline}{RGB}{166,118,29}
\definecolor{lfgrayfill}{RGB}{244,244,244}
\definecolor{lfgrayline}{RGB}{95,95,95}
\begin{tikzpicture}[
  fname/.style={font=\ttfamily\footnotesize\bfseries, text=black,
    inner sep=1.6pt, anchor=west},
  fnote/.style={font=\footnotesize, text=annotgray, inner sep=1.6pt,
    anchor=west},
  treeline/.style={draw=lfink!55, line width=0.55pt},
  pics/folder/.style={code={
    \fill[lfamberfill, draw=lfamberline, line width=0.5pt]
      (-0.15,0.04) -- (-0.15,0.12) -- (-0.04,0.12) -- (0.00,0.06)
      -- (0.15,0.06) -- (0.15,-0.11) -- (-0.15,-0.11) -- cycle;
  }},
  pics/lfile/.style={code={
    \fill[lfgrayfill, draw=lfgrayline, line width=0.5pt]
      (-0.10,-0.13) -- (-0.10,0.13) -- (0.04,0.13) -- (0.10,0.07)
      -- (0.10,-0.13) -- cycle;
    \draw[lfgrayline, line width=0.5pt] (0.04,0.13) -- (0.04,0.07)
      -- (0.10,0.07);
  }},
]
\newcommand{\treerow}[6]{  \pic at (#2,#3) {#4};
  \node[fname] (row#1) at (#2+0.22,#3) {#5};
  \node[fnote] at ($(row#1.east)+(0.12,0)$) {#6};
}
\treerow{a}{0}{0}{folder}{ITEM-f/}{Lean/Lake project for analytic
  ITEM-f}
\treerow{b}{0.78}{-0.60}{lfile}{lakefile.toml, lean-toolchain}{build
  configuration; commit-pinned Lean and mathlib}
\treerow{c}{0.78}{-1.20}{lfile}{config.json}{pins the fourteen theorem
  names and permitted axioms for the comparator}
\treerow{d}{0.78}{-1.80}{lfile}{Challenge.lean}{states the fourteen
  declarations, one approved placeholder each}
\treerow{e}{0.78}{-2.40}{lfile}{Solution.lean}{the same statements,
  closed by the proof development}
\treerow{f}{0.78}{-3.00}{lfile}{AxiomAudit.lean}{runs the axiom audit
  over the fourteen declarations}
\treerow{g}{0.78}{-3.60}{lfile}{ITEMf.lean}{umbrella module importing
  the development}
\treerow{h}{0.78}{-4.20}{folder}{ITEMf/}{proof modules, organized in
  dependency layers:}
\treerow{i}{1.56}{-4.80}{folder}{Spec/}{proof-free statement
  interfaces; all that \texttt{Challenge.lean} imports}
\treerow{j}{1.56}{-5.40}{folder}{Model/}{problem class, interpolation
  inequality, transformed objective}
\treerow{k}{1.56}{-6.00}{folder}{Construction/}{the shooting argument:
  one-step map, orbit, target angle, symmetry, rate}
\treerow{l}{1.56}{-6.60}{folder}{Algorithm/}{iterates, coefficient
  identities, coordinate relations}
\treerow{m}{1.56}{-7.20}{folder}{Lyapunov/}{norm identities, rotations,
  monotonicity}
\treerow{n}{1.56}{-7.80}{lfile}{Convergence.lean}{the convergence
  theorem}
\treerow{o}{0.78}{-8.40}{folder}{artifacts/}{build, hygiene,
  axiom-audit, and comparator replay logs}
\draw[treeline] (0.02,-0.22) -- (0.02,-8.40);
\foreach \y in {-0.60,-1.20,-1.80,-2.40,-3.00,-3.60,-4.20,-8.40}
  {\draw[treeline] (0.02,\y) -- (0.56,\y);}
\draw[treeline] (0.80,-4.42) -- (0.80,-7.80);
\foreach \y in {-4.80,-5.40,-6.00,-6.60,-7.20,-7.80}
  {\draw[treeline] (0.80,\y) -- (1.34,\y);}
\end{tikzpicture}\end{adjustbox}
\caption{File structure of the ITEM-f Lean project
(module files grouped by folder). The proof-free \texttt{Spec/} layer
is all that \texttt{Challenge.lean} imports, so the proofs reach
the comparator only through \texttt{Solution.lean};
\texttt{config.json} pins the fourteen public theorems and the
permitted axioms, and \texttt{artifacts/} records the build, audit, and
replay logs.}
\label{fig:itemf-file-tree}
\end{figure}

\begin{table}[t]
\centering
\begin{tabular}{lc}
\toprule
Public theorem declarations & 14\\
Dependency nodes in the closure & 20\\
Lean source files & 38\\
Lines of Lean & 7{,}985\\
\texttt{lake build} jobs & 8{,}694\\
Comparator replay time, native macOS & 105\,s\\
Comparator replay time, WSL2 real Landrun & 115\,s\\
Toolchain & Lean 4.32.0, mathlib 4.32.0, commit-pinned\\
Axioms used & \texttt{propext}, \texttt{Quot.sound}, \texttt{Classical.choice}\\
\bottomrule
\end{tabular}
\vspace{0.05in}
\caption{Summary statistics of the ITEM-f Lean
project: it formalizes a fixed 20-node dependency closure of manuscript
claims and their supporting steps through fourteen public theorem
declarations, builds cleanly under the commit-pinned toolchain, and is
accepted by the comparator replay under exactly the three standard
axioms listed.}
\label{tab:itemf-project}
\end{table}

The convergence theory of ITEM-f is machine-checked in the same manner
as that of lemniscate acceleration
(Section~\ref{sec:lemniscate-lean}), by a second Lean project whose
fourteen public declarations correspond to the results of this
section and its appendix. The
root declaration \texttt{ITEMf.convergence} is the Lean counterpart of
Theorem~\ref{thm:itemf-convergence}, stated over the Euclidean space
$\mathbb{R}^{d}$ and proving both inequalities of \eqref{eq:itemf-rate};
it reads:
\begin{center}
\begin{minipage}{0.92\textwidth}
\ttfamily\footnotesize
theorem convergence\\
\hspace*{2em}\{d N : Nat\} (hd : 1 $\leq$ d) (hN : 1 $\leq$ N)\\
\hspace*{2em}(M : StronglyConvexSmoothModel (Euclidean d)) \{xStar : Euclidean d\}\\
\hspace*{2em}(hxStar : M.IsMinimizer xStar)\\
\hspace*{2em}(C : CoeffData N) (hC : ValidCoefficients M.q C)\\
\hspace*{2em}(x0 : Euclidean d) :\\
\hspace*{2em}M.f (itemfIterate M C x0 N) - M.f xStar $\leq$\\
\hspace*{5em}(1 / C.Upsilon \^{} 2) * (M.f x0 - M.f xStar) $\wedge$\\
\hspace*{3em}(1 / C.Upsilon \^{} 2) * (M.f x0 - M.f xStar) $\leq$\\
\hspace*{5em}4 * (1 - Real.sqrt M.q) \^{} (2 * N) *\\
\hspace*{6em}(M.f x0 - M.f xStar)
\end{minipage}
\end{center}
The supporting declarations follow the paper's proofs.
\path{coefficientsExistUnique} is the construction
Lemma~\ref{lem:itemf-construction}, obtained through the shooting
argument of Appendix~\ref{sec:app-itemf-existence} via
\path{oneStepMap} (Lemma~\ref{lem:app-itemf-one-step}),
\texttt{orbitComparison} (Lemma~\ref{lem:app-itemf-comparison}),
\texttt{targetAngle} (Lemma~\ref{lem:app-itemf-target-angle}), and
\path{shootingIffAdmissible}
(Lemma~\ref{lem:app-itemf-shooting-equivalence});
\path{involutionSymmetry} is
Lemma~\ref{lem:itemf-involution-symmetry};
\path{explicitRate} is Lemma~\ref{lem:itemf-explicit-rate}; and
\path{finiteInterpolation} is \eqref{eq:F_0L_formula}. The Lyapunov argument of
Section~\ref{sec:itemf-lyapunov} is formalized by
\texttt{transformedMemF0}, \texttt{threePointIdentity},
\texttt{coordinateRelations} (the identities
\eqref{eq:itemf-lyap-coordinate-relations}), \texttt{normIdentities}
(Lemma~\ref{lem:itemf-lyap-norms}), and \texttt{lyapunovMonotone}
(Lemma~\ref{lem:itemf-lyap-monotone}).

The project comprises 38 Lean source files and 7{,}985 lines of Lean,
organized as in Figure~\ref{fig:itemf-file-tree} in a layout parallel
to the LemniAcc project, with one addition: a proof-free
specification layer
(\path{ITEMf/Spec}) contains the definitions and statements of
the fourteen declarations, and the comparator-facing
\path{Challenge.lean}, which restates each declaration with exactly
one approved placeholder, imports only this layer, so the proofs reach
the comparator only through
\texttt{Solution.lean}, which closes every placeholder. The project
passes \texttt{lake build} under the same commit-pinned toolchain
(8{,}694 build jobs); the hygiene scan confirms that no \texttt{sorry},
\texttt{axiom}, \texttt{admit}, or \texttt{unsafe} appears anywhere in
the proof development or in \texttt{Solution.lean}; and the axiom audit
again reports exactly \texttt{propext}, \texttt{Quot.sound}, and
\texttt{Classical.choice} for every public declaration. The comparator
accepted the fourteen public theorems twice, under the two-platform
protocol,
evidence label, and disclosures of Section~\ref{sec:lemniscate-lean}:
in 105 seconds on native macOS and in 115 seconds under the isolated
WSL2 real-Landrun replay;
Table~\ref{tab:itemf-project} summarizes the project.

\subsection{Self-H-duality of \ref{eq:lemniscate-acceleration} and \ref{eq:item-f-momentum}}
\label{sec:lemniscate-time-reversal}

The notion of \emph{H-duality} was introduced in \cite{kim2023time} and
further developed in
\cite{kim2023mirror,yoon2024optimal,yoon2025hinvariance,yoon2026composition}.
Recall that a fixed-step first-order method (FSFOM) with stepsizes
$\{h_{i,j}\}_{0\leq j<i\leq N}$ is a method of the form
\[
x_{i}=x_{i-1}-\frac{1}{L}\sum_{j=0}^{i-1}h_{i,j}\,\nabla f(x_{j}),
\qquad i=1,\ldots,N.
\]
The \emph{H-dual} of this method is the FSFOM with stepsizes $h^{\text{A}}_{i,j}\triangleq h_{N-j,N-i}$ for $0\leq j<i\leq N$. Viewing the lower-triangular array $\{h_{i,j}\}_{0\leq j<i\leq N}$ as an $(N+1)\times(N+1)$ matrix, taking the H-dual corresponds to taking the anti-diagonal transpose, i.e., the entries are reflected
across the anti-diagonal rather than the usual diagonal. These prior
works observe that H-duality exchanges the type of guarantee while
preserving its worst-case value: for example, OGM, which reduces the
function value under an initial-distance condition, maps to OGM-G, which
reduces the squared gradient norm under an initial function-gap
condition, with exactly the same constant \cite{kim2023time}.

\begin{remark}[Self-H-duality]\label{rem:self-h-duality}
In this work, we point out that \eqref{eq:lemniscate-acceleration} and
\eqref{eq:item-f-momentum} are \emph{self-H-dual} in the sense that each method is its own
H-dual. Moreover, the corresponding continuous-time ODEs are also self-H-dual in the sense of \cite{kim2023time}: the ODEs have the form $\ddot{X}(t)+\gamma(t)\dot{X}(t)+2\nabla f(X(t))=0$, with the symmetry $\gamma(t)=\gamma(T-t)$ for $0<t<T$.
We record this observation without a detailed proof or
Lean verification; a reader who wants both can obtain them by running
the \textsf{AutoOPT} pipeline on this claim.
\end{remark}

\section{Conclusion}\label{sec:conclusion}

In this work, we present \textsf{AutoOPT}, a framework that automates much of the labor-intensive parts in the design and analysis of optimal first-order methods. \textsf{AutoOPT} combines the domain-specific methodology of performance estimation programming with the general-purpose intelligence of frontier LLMs. We demonstrate the strength of this harness with two case studies, producing two optimal algorithms of independent interest: lemniscate acceleration (Section~\ref{sec:lemniscate}) and analytic ITEM-f (Section~\ref{sec:itemf}). More broadly, this work serves as a case study of how a domain-specific harness can augment the general-purpose intelligence of LLMs.

\textsf{AutoOPT} has significant limitations. For example, it does not handle stochastic or second-order optimization. We expect \textsf{AutoOPT} itself to be extensible to such settings, and we believe the general principle of the domain-specific harness extends further still.

Finally, we offer this paper itself as one data point on what scholarship can look like in the era of LLM intelligence. As stated in the Preface, we believe domain experts can contribute productively in an AI-driven research workflow by augmenting the general-purpose intelligence of LLMs with harnesses derived from their expertise. Automation lets us take on harder and more interesting problems and elevate our scholarship, and it frees the researcher to concentrate on the parts of research that are irreducibly human: choosing the problems, interpreting the answers, and conveying the insight to others.

\section*{Acknowledgments}

S.\ Das Gupta and E.\ K.\ Ryu acknowledge support by AFOSR Grant Number FA9550-25-1-0183.

\bibliographystyle{plainurl}
\bibliography{references}

\appendix
\section{Deferred proofs and details for lemniscate acceleration}
\label{sec:appendix-lemniscate}

\subsection{Existence and uniqueness of the lemniscate recurrence}
\label{sec:app-recurrence-existence}

In this section we prove Lemma~\ref{lem:lemniscate-recurrence}, establish
the bounds on $\Omega_{N}$ used in
Theorem~\ref{thm:lemniscate-convergence}, and give a complete bisection
rule for computing $\Omega_{N}$.

\paragraph{The one-step map.}
For $\Omega>0$ and $\rho\in(0,1]$, a pair of consecutive terms of
\eqref{eq:lemniscate-recurrence} solves
\begin{equation}\label{eq:app-one-step}
\Omega\left(\rho-t\right)^{2}=\rho\left(1-t^{2}\right)
\end{equation}
in the unknown $t$, which rearranges to the quadratic equation $(\Omega+\rho)\,t^{2}-2\Omega\rho\,t+\Omega\rho^{2}-\rho=0$. We work on the admissible set $\mathcal{A}\triangleq\left\{ (\Omega,\rho):\Omega>0,\;0<\rho\leq1,\;\Omega\rho\geq1\right\}$.

\begin{lem}[One-step map]\label{lem:app-one-step-map}
For $(\Omega,\rho)\in\mathcal{A}$, equation \eqref{eq:app-one-step} has
exactly one solution $t$ with $0\leq t<\rho$, namely
\begin{equation}\label{eq:app-F-def}
F_{\Omega}(\rho)\triangleq
\frac{\Omega\rho-\sqrt{\rho^{2}+\Omega\rho(1-\rho^{2})}}{\Omega+\rho},
\end{equation}
and the other solution is greater than or equal to $\rho$, with equality
only when $\rho=1$. Moreover:
\begin{enumerate}
\item $F_{\Omega}(\rho)=0$ if and only if $\Omega\rho=1$;
\item $F$ is continuous on $\mathcal{A}$ and strictly increasing in each of
$\Omega$ and $\rho$ on the fibers of $\mathcal{A}$ (the other variable
being fixed).
\end{enumerate}
\end{lem}

\begin{proof}
The quarter-discriminant of the quadratic equals
$\Omega^{2}\rho^{2}-(\Omega+\rho)(\Omega\rho^{2}-\rho)
=\rho^{2}+\Omega\rho(1-\rho^{2})>0$, so \eqref{eq:app-one-step} has the two
real solutions
\begin{equation*}
t_{\pm}=\frac{\Omega\rho\pm\sqrt{\rho^{2}+\Omega\rho(1-\rho^{2})}}{\Omega+\rho}.
\end{equation*}
Since $\Omega\rho(1-\rho^{2})\geq0$, the square root is at least $\rho$,
hence
\begin{equation*}
t_{+}\geq\frac{\Omega\rho+\rho}{\Omega+\rho}
\geq\frac{\Omega\rho+\rho^{2}}{\Omega+\rho}=\rho,
\end{equation*}
and both inequalities hold with equality only when $\rho=1$. On the other
hand, $t_{-}<\Omega\rho/(\Omega+\rho)<\rho$. The product of the two
solutions is $t_{+}t_{-}=\rho(\Omega\rho-1)/(\Omega+\rho)\geq0$ on
$\mathcal{A}$; since $t_{+}\geq\rho>0$, this gives $t_{-}\geq0$, with
$t_{-}=0$ exactly when $\Omega\rho=1$. Therefore $F_{\Omega}(\rho)=t_{-}$
is the unique solution in $[0,\rho)$, which proves the root dichotomy and
claim 1.

For claim 2, continuity follows from \eqref{eq:app-F-def} because the
denominator is positive. For the strict monotonicity, define
\begin{equation*}
\psi(\Omega,\rho,t)\triangleq\Omega(\rho-t)^{2}-\rho(1-t^{2}),
\end{equation*}
and fix a point in the interior of $\mathcal{A}$ with
$t=F_{\Omega}(\rho)\in[0,\rho)$. Substituting the root identity
$1-t^{2}=\Omega(\rho-t)^{2}/\rho$ into the partial derivatives of
$\psi$ and using $0\leq t<\rho\leq1$, tedious but straightforward
algebra gives $\psi_{\Omega}>0$, $\psi_{\rho}>0$, and $\psi_{t}<0$ at
this point. By the implicit function
theorem, $F$ is continuously differentiable in the interior of
$\mathcal{A}$ with
$\partial F/\partial\Omega=-\psi_{\Omega}/\psi_{t}>0$ and
$\partial F/\partial\rho=-\psi_{\rho}/\psi_{t}>0$. Consequently, for fixed
$\rho\in(0,1)$ the map $\Omega\mapsto F_{\Omega}(\rho)$ is continuous on
$[1/\rho,\infty)$ and strictly increasing on its interior, hence strictly
increasing on $[1/\rho,\infty)$; the same argument applies to
$\rho\mapsto F_{\Omega}(\rho)$ on $[1/\Omega,1]$ for fixed $\Omega>1$; and
on the remaining boundary segment $\rho=1$ the explicit formula
$F_{\Omega}(1)=(\Omega-1)/(\Omega+1)$ is strictly increasing in $\Omega$.
\end{proof}

\paragraph{Shooting orbits.}
For $\Omega\geq1$ define the shooting iterates $\Phi_{0}(\Omega)\triangleq1$ and $\Phi_{k+1}(\Omega)\triangleq F_{\Omega}\left(\Phi_{k}(\Omega)\right)$
defined for as long as $(\Omega,\Phi_{k}(\Omega))\in\mathcal{A}$. Taking
the positive square root in \eqref{eq:app-one-step} gives the step identity
\begin{equation}\label{eq:app-step-size}
\Phi_{k}(\Omega)-\Phi_{k+1}(\Omega)
=\sqrt{\frac{\Phi_{k}(\Omega)\left(1-\Phi_{k+1}(\Omega)^{2}\right)}{\Omega}}
\leq\frac{1}{\sqrt{\Omega}}
\end{equation}
whenever the left-hand side is defined; the final bound holds because the
iterates satisfy
$0\leq\Phi_{k+1}(\Omega)<\Phi_{k}(\Omega)\leq\Phi_{0}(\Omega)=1$
(Lemma~\ref{lem:app-one-step-map}), so
$\Phi_{k}(\Omega)\left(1-\Phi_{k+1}(\Omega)^{2}\right)\leq1$.

\begin{proof}[Proof of Lemma~\ref{lem:lemniscate-recurrence}]
Call a pair $(\Omega,\{\rho_{k}\}_{k=0}^{N+1})$ \emph{admissible for
horizon $N$} if it satisfies all requirements of the lemma:
$1=\rho_{0}>\rho_{1}>\cdots>\rho_{N+1}=0$ and
\eqref{eq:lemniscate-recurrence} for $0\leq k\leq N$.

\emph{Step 0: every admissible pair is a shooting orbit.}
Let $(\Omega,\{\rho_{k}\})$ be admissible for horizon $N$. Strict decrease
gives $\rho_{k}>\rho_{N+1}=0$ for $k\leq N$. The recurrence at $k=N$ reads
$\Omega\rho_{N}^{2}=\rho_{N}$, so $\rho_{N}=1/\Omega$, and the decrease of
the sequence gives $\Omega\rho_{k}\geq\Omega\rho_{N}=1$ for $k\leq N$;
hence $(\Omega,\rho_{k})\in\mathcal{A}$ for $0\leq k\leq N$. In particular
$1/\Omega=\rho_{N}\leq\rho_{0}=1$ forces $\Omega\geq1$, and for $N\geq1$
the strict inequality $\rho_{N}<\rho_{0}$ forces $\Omega>1$. Finally, for
$0\leq k\leq N$ the term $\rho_{k+1}$ solves \eqref{eq:app-one-step} with
$0\leq\rho_{k+1}<\rho_{k}$, so Lemma~\ref{lem:app-one-step-map} gives
$\rho_{k+1}=F_{\Omega}(\rho_{k})$. Therefore $\rho_{k}=\Phi_{k}(\Omega)$
for all $0\leq k\leq N+1$: the sequence is uniquely determined by
$\Omega$, which is the final claim of the lemma.

\emph{Step 1: existence, by induction on the horizon.}
We construct numbers $1=\Omega_{0}<\Omega_{1}<\Omega_{2}<\cdots$ such that
$\left(\Omega_{N},\{\Phi_{k}(\Omega_{N})\}_{k=0}^{N+1}\right)$ is
admissible for horizon $N$. For $N=0$ the requirements reduce to
$\Omega(\rho_{0}-\rho_{1})^{2}=\rho_{0}(1-\rho_{1}^{2})$ with $\rho_{0}=1$
and $\rho_{1}=0$, that is, $\Omega_{0}=1$; and indeed
$\Phi_{1}(1)=F_{1}(1)=0$.

For the induction step, assume
$\left(\Omega_{N},\{\Phi_{k}(\Omega_{N})\}\right)$ is admissible for
horizon $N$. We first record, by induction on $k\in\{1,\ldots,N+1\}$, that
the maps $\Omega\mapsto\Phi_{k}(\Omega)$ are well defined, continuous, and
strictly increasing on the ray $[\Omega_{N},\infty)$. For $k=1$,
$\Phi_{1}(\Omega)=F_{\Omega}(1)=(\Omega-1)/(\Omega+1)$ has all three
properties. For the step from $k$ to $k+1$ (with $k\leq N$), the
monotonicity of $F$ and admissibility at horizon $N$ give the domain bound $\Omega\,\Phi_{k}(\Omega)\geq\Omega_{N}\,\Phi_{k}(\Omega_{N})\geq\Omega_{N}\,\Phi_{N}(\Omega_{N})=\Omega_{N}/\Omega_{N}=1$,
so $(\Omega,\Phi_{k}(\Omega))\in\mathcal{A}$ and
$\Phi_{k+1}(\Omega)=F_{\Omega}(\Phi_{k}(\Omega))$ is well defined; it is
continuous by Lemma~\ref{lem:app-one-step-map} and the induction
hypothesis, and strictly increasing because, for
$\Omega_{N}\leq\Omega<\Omega'$, we have $\Phi_{k+1}(\Omega)=F_{\Omega}\left(\Phi_{k}(\Omega)\right)<F_{\Omega'}\left(\Phi_{k}(\Omega)\right)<F_{\Omega'}\left(\Phi_{k}(\Omega')\right)=\Phi_{k+1}(\Omega')$,
with the middle application being legitimate since
$\Omega'\,\Phi_{k}(\Omega)>\Omega\,\Phi_{k}(\Omega)\geq1$, and the second
inequality being strict by the induction hypothesis
$\Phi_{k}(\Omega)<\Phi_{k}(\Omega')$.

Now consider $h(\Omega)\triangleq\Omega\,\Phi_{N+1}(\Omega)$ on
$[\Omega_{N},\infty)$. It is continuous, and it is strictly increasing
because $\Phi_{N+1}$ is nonnegative and strictly increasing there: for
$\Omega_{N}\leq\Omega<\Omega'$,
$\Omega\,\Phi_{N+1}(\Omega)\leq\Omega\,\Phi_{N+1}(\Omega')
<\Omega'\,\Phi_{N+1}(\Omega')$. Moreover
$h(\Omega_{N})=\Omega_{N}\cdot\Phi_{N+1}(\Omega_{N})=0$, and telescoping
\eqref{eq:app-step-size} over $0\leq k\leq N$, all iterates being defined
on this ray, gives
$\Phi_{N+1}(\Omega)\geq1-(N+1)/\sqrt{\Omega}$, so
$h(\Omega)\geq\Omega-(N+1)\sqrt{\Omega}\to\infty$ as $\Omega\to\infty$. By
the intermediate value theorem there is a unique
$\Omega_{N+1}\in(\Omega_{N},\infty)$ with $h(\Omega_{N+1})=1$, that is, $\Phi_{N+1}(\Omega_{N+1})=1/\Omega_{N+1}$. The sequence $\rho_{k}\triangleq\Phi_{k}(\Omega_{N+1})$ for
$0\leq k\leq N+1$, extended by $\rho_{N+2}\triangleq0$, is then admissible
for horizon $N+1$: the recurrence holds for $0\leq k\leq N$ by
construction, and at $k=N+1$ it reads
$\Omega_{N+1}(\rho_{N+1}-0)^{2}=1/\Omega_{N+1}=\rho_{N+1}(1-0^{2})$; the
sequence is strictly decreasing because $F_{\Omega}(\rho)<\rho$ at every
step and $\rho_{N+1}=1/\Omega_{N+1}>0=\rho_{N+2}$; and it is nonnegative
because $\rho_{k}\geq\rho_{N+1}>0$ for $k\leq N+1$. This completes the
induction, so an admissible $\Omega$ exists for every horizon.

\emph{Step 2: uniqueness of $\Omega$.}
Suppose $(\Omega,\{\rho_{k}\})$ and $(\Omega',\{\rho'_{k}\})$ are both
admissible for horizon $N$ with $\Omega<\Omega'$. By Step 0, both sequences
are shooting orbits, every pair $(\Omega,\rho_{k})$ and
$(\Omega',\rho'_{k})$ with $k\leq N$ lies in $\mathcal{A}$, and
$\Omega,\Omega'\geq1$. For $N=0$, Step 0 gives
$\rho_{1}=(\Omega-1)/(\Omega+1)=0$ and
$\rho'_{1}=(\Omega'-1)/(\Omega'+1)=0$, forcing $\Omega=\Omega'=1$ and
contradicting $\Omega<\Omega'$; so assume $N\geq1$. We prove
$\rho_{k}<\rho'_{k}$ for $1\leq k\leq N+1$ by induction. For $k=1$,
\begin{equation*}
\rho_{1}=\frac{\Omega-1}{\Omega+1}<\frac{\Omega'-1}{\Omega'+1}=\rho'_{1}.
\end{equation*}
For the induction step, the strict monotonicity of $F$ in both arguments
(Lemma~\ref{lem:app-one-step-map}) gives $\rho_{k+1}=F_{\Omega}(\rho_{k})<F_{\Omega'}(\rho_{k})<F_{\Omega'}(\rho'_{k})=\rho'_{k+1}$, where the middle steps are valid because
$\Omega'\rho_{k}>\Omega\rho_{k}\geq1$, and the second inequality is
strict by the induction hypothesis $\rho_{k}<\rho'_{k}$. Iterating the
induction through
the final step $k=N$ yields $\rho_{N+1}<\rho'_{N+1}$, contradicting
$\rho_{N+1}=\rho'_{N+1}=0$. Hence the admissible
$\Omega$ is unique, and Step 0 shows the sequence is uniquely determined
by it.
\end{proof}

The same machinery yields the quantitative bounds on $\Omega_{N}$ quoted in
Theorem~\ref{thm:lemniscate-convergence} and
a complete bisection rule for computing $\Omega_{N}$.

\begin{cor}[Bounds on $\Omega_{N}$ and bisection correctness]\label{cor:app-omega-bounds}
Let $N\geq1$, and let $\Omega_{N}$ and $\{\rho_{k}\}_{k=0}^{N+1}$ be as in
Lemma~\ref{lem:lemniscate-recurrence}. Then:
\begin{enumerate}
\item $\dfrac{(N+1)^{2}}{\varpi^{2}}<\Omega_{N}<(N+1)^{2}$, where $\varpi$
is the lemniscate constant defined in
Theorem~\ref{thm:lemniscate-convergence}; moreover
$1=\Omega_{0}<\Omega_{1}<\Omega_{2}<\cdots$.
\item Fix $\Omega\geq1$ and generate the shooting iterates
$\Phi_{0}(\Omega),\Phi_{1}(\Omega),\ldots$, computing
$\Phi_{k+1}(\Omega)=F_{\Omega}\left(\Phi_{k}(\Omega)\right)$ whenever
$\Phi_{k}(\Omega)\geq1/\Omega$. If $\Omega>\Omega_{N}$, then
$\Phi_{k}(\Omega)>1/\Omega$ for all $0\leq k\leq N$ and
$\Phi_{N+1}(\Omega)>0$. If $\Omega<\Omega_{N}$, then there is an index
$k\leq N$ with $\Phi_{k}(\Omega)<1/\Omega$. If $\Omega=\Omega_{N}$, the
iterates realize the sequence of Lemma~\ref{lem:lemniscate-recurrence}.

Consequently, the following exact bisection is correct. Initialize
$[\Omega_{l},\Omega_{r}]=[1,(N+1)^{2}]$. Given the midpoint
$\Omega_{m}=(\Omega_{l}+\Omega_{r})/2$, generate the shooting iterates
until either some $k\leq N$ satisfies
$\Phi_{k}(\Omega_{m})<1/\Omega_{m}$, in which case set
$\Omega_{l}\gets\Omega_{m}$, or the iterates have been generated
through $\Phi_{N+1}(\Omega_{m})$. In the latter case, set
$\Omega_{r}\gets\Omega_{m}$ if $\Phi_{N+1}(\Omega_{m})>0$, and return
$\Omega_{m}=\Omega_{N}$ and the generated sequence if
$\Phi_{N+1}(\Omega_{m})=0$. At every nonterminal step the invariant
$\Omega_{l}<\Omega_{N}\leq\Omega_{r}$ is maintained, the bracket length
is halved, and the midpoints converge to $\Omega_{N}$.

\end{enumerate}
\end{cor}

\begin{proof}
\emph{Part 1.} At $\Omega=\Omega_{N}$ the step identity
\eqref{eq:app-step-size} holds for $0\leq k\leq N$, and telescoping gives
\begin{equation*}
1=\rho_{0}-\rho_{N+1}
=\sum_{k=0}^{N}\sqrt{\frac{\rho_{k}\left(1-\rho_{k+1}^{2}\right)}{\Omega_{N}}}
\leq\frac{N+1}{\sqrt{\Omega_{N}}},
\end{equation*}
and the term $k=1$, present since $N\geq1$, is strictly smaller than
$1/\sqrt{\Omega_{N}}$ because $\rho_{1}<1$; hence $\sqrt{\Omega_{N}}<N+1$. For the lower bound, dividing
the step identity by $\sqrt{\rho_{k}(1-\rho_{k+1}^{2})}$ and summing gives
\begin{equation*}
\frac{N+1}{\sqrt{\Omega_{N}}}
=\sum_{k=0}^{N}\frac{\rho_{k}-\rho_{k+1}}{\sqrt{\rho_{k}\left(1-\rho_{k+1}^{2}\right)}}
=\sum_{k=0}^{N}\int_{\rho_{k+1}}^{\rho_{k}}
\frac{dt}{\sqrt{\rho_{k}\left(1-\rho_{k+1}^{2}\right)}}
<\sum_{k=0}^{N}\int_{\rho_{k+1}}^{\rho_{k}}\frac{dt}{\sqrt{t\left(1-t^{2}\right)}}
=\int_{0}^{1}\frac{dt}{\sqrt{t\left(1-t^{2}\right)}}=\varpi,
\end{equation*}
where the strict inequality holds because on the interior of each
interval the integrand on the right strictly dominates the constant
integrand on the left, and the substitution $t=s^{2}$ identifies the
comparison integral with the form
$\varpi=2\int_{0}^{1}ds/\sqrt{1-s^{4}}$ of
Section~\ref{sec:lemniscate-constant}. Hence $\Omega_{N}>(N+1)^{2}/\varpi^{2}$. Strict monotonicity of
$N\mapsto\Omega_{N}$ was established in Step 1 of the proof of
Lemma~\ref{lem:lemniscate-recurrence} and is inherited by the unique
admissible values.

\emph{Part 2.} If $\Omega>\Omega_{N}$, the sub-induction in Step 1 (anchored
at $\Omega_{N}$) shows that $\Phi_{k}(\Omega)$ is defined and strictly
increasing in $\Omega$ for $1\leq k\leq N+1$; therefore, for
$0\leq k\leq N$, using that the iterates decrease in the index
($F_{\Omega}(\rho)<\rho$),
\begin{equation*}
\Phi_{k}(\Omega)\geq\Phi_{N}(\Omega)>\Phi_{N}(\Omega_{N})
=\frac{1}{\Omega_{N}}>\frac{1}{\Omega},
\end{equation*}
and $\Phi_{N+1}(\Omega)>\Phi_{N+1}(\Omega_{N})=0$. If
$1\leq\Omega<\Omega_{N}$, suppose toward a contradiction that
$\Phi_{k}(\Omega)\geq1/\Omega$ for all $0\leq k\leq N$; then all iterates
up to $\Phi_{N}(\Omega)$ are defined. The comparison induction of Step 2
does not use admissibility of the lower orbit at the terminal index: it
only needs both shooting orbits to be defined up to the index being
compared, the domain of each step being guaranteed here by the
contradiction hypothesis and by
$\Omega_{N}\,\Phi_{k}(\Omega)>\Omega\,\Phi_{k}(\Omega)\geq1$. The same
induction therefore gives
$\Phi_{N}(\Omega)<\Phi_{N}(\Omega_{N})=1/\Omega_{N}<1/\Omega$, a
contradiction. The case $\Omega=\Omega_{N}$ is
Lemma~\ref{lem:lemniscate-recurrence} itself.

For the bisection, Part 1 gives
$\Omega_{l}=1<\Omega_{N}<(N+1)^{2}=\Omega_{r}$ at initialization.
The sign characterization just proved says that the left update occurs
exactly when $\Omega_{m}<\Omega_{N}$, the right update occurs exactly
when $\Omega_{m}>\Omega_{N}$, and the remaining exact case is
$\Omega_{m}=\Omega_{N}$, where
Lemma~\ref{lem:lemniscate-recurrence} gives
$\Phi_{N+1}(\Omega_{m})=0$. Hence each nonterminal update preserves
$\Omega_{l}<\Omega_{N}\leq\Omega_{r}$ while halving the bracket, and the
midpoints converge to $\Omega_{N}$.

\end{proof}

The comparison with the integral in Part 1 in fact identifies the limit
of $\Omega_{N}/(N+1)^{2}$: the lower bound of
Corollary~\ref{cor:app-omega-bounds} is asymptotically tight.

\begin{lem}\label{lem:app-omega-limit}
Let $\Omega_{N}$ be as in Lemma~\ref{lem:lemniscate-recurrence}. Then
\begin{equation*}
\lim_{N\to\infty}\frac{\Omega_{N}}{(N+1)^{2}}=\frac{1}{\varpi^{2}}.
\end{equation*}
\end{lem}

\begin{proof}
Write $\{\rho_{k}\}_{k=0}^{N+1}$ for the sequence of
Lemma~\ref{lem:lemniscate-recurrence} at $\Omega_{N}$, put
$\ell_{k}\triangleq\rho_{k}-\rho_{k+1}>0$, and set
\begin{equation*}
g(x)\triangleq\frac{1}{\sqrt{x\left(1-x^{2}\right)}},\quad 0<x<1,
\end{equation*}
so that the improper integral $\int_{0}^{1}g(x)\,dx=\varpi$, as recorded in the proof of
Corollary~\ref{cor:app-omega-bounds}. As in Part 1 of that proof,
dividing the step identity \eqref{eq:app-step-size} at $\Omega=\Omega_{N}$
by $\sqrt{\rho_{k}(1-\rho_{k+1}^{2})}$ and summing over $0\leq k\leq N$
gives
\begin{equation}\label{eq:app-omega-limit-sum}
\frac{N+1}{\sqrt{\Omega_{N}}}
=\sum_{k=0}^{N}\frac{\ell_{k}}{\sqrt{\rho_{k}\left(1-\rho_{k+1}^{2}\right)}}.
\end{equation}
We claim that the right-hand side converges to $\varpi$; the lemma
follows, since then $\sqrt{\Omega_{N}}/(N+1)\to1/\varpi$.

The intervals $[\rho_{k+1},\rho_{k}]$, $0\leq k\leq N$, partition
$[0,1]$, and by the step identity and the lower bound of
Corollary~\ref{cor:app-omega-bounds} their lengths satisfy $\delta_{N}\triangleq\max_{0\leq k\leq N}\ell_{k}\leq1/\sqrt{\Omega_{N}}<\varpi/(N+1),$
so the mesh $\delta_{N}$ tends to zero. For
$x\in[\rho_{k+1},\rho_{k}]\cap(0,1)$ we have $x\leq\rho_{k}$ and
$1-x^{2}\leq1-\rho_{k+1}^{2}$, hence
$g(x)\geq1/\sqrt{\rho_{k}(1-\rho_{k+1}^{2})}$, so the right-hand side of
\eqref{eq:app-omega-limit-sum} is a lower Riemann-type sum for
$\int_{0}^{1}g(x)\,dx=\varpi$.

In particular, each summand is at most
$\int_{\rho_{k+1}}^{\rho_{k}}g(x)\,dx$, so the sum is at most $\varpi$
for every $N$. For the matching lower bound, fix
$\varepsilon\in(0,1/4]$. On $(0,\varepsilon]$ we have
$1-x^{2}\geq3/4$, and on $[1-\varepsilon,1)$ we have $x\geq1/2$ and
$1-x^{2}\geq1-x$, so the endpoint tails satisfy
\begin{equation*}
\int_{0}^{\varepsilon}g(x)\,dx\leq\frac{4}{\sqrt{3}}\sqrt{\varepsilon}
\qquad\text{and}\qquad
\int_{1-\varepsilon}^{1}g(x)\,dx\leq2\sqrt{2}\,\sqrt{\varepsilon}.
\end{equation*}
Once $\delta_{N}<\varepsilon/2$, every interval
$[\rho_{k+1},\rho_{k}]$ that meets $[\varepsilon,1-\varepsilon]$ lies
in $[\varepsilon/2,1-\varepsilon/2]$, where the two-variable kernel
$(x,y)\mapsto1/\sqrt{x(1-y^{2})}$ is uniformly continuous, with
modulus of continuity $\omega_{\varepsilon}$; since $g(x)$ is this
kernel at $(x,x)$, we get
$1/\sqrt{\rho_{k}(1-\rho_{k+1}^{2})}\geq g(x)-\omega_{\varepsilon}(\delta_{N})$
for every $x\in[\rho_{k+1},\rho_{k}]$. Discarding the remaining
(positive) summands and summing over the intervals that meet
$[\varepsilon,1-\varepsilon]$, whose union covers
$[\varepsilon,1-\varepsilon]$ and has total length at most $1$,
\begin{equation*}
\sum_{k=0}^{N}\frac{\ell_{k}}{\sqrt{\rho_{k}\left(1-\rho_{k+1}^{2}\right)}}
\geq\int_{\varepsilon}^{1-\varepsilon}g(x)\,dx-\omega_{\varepsilon}(\delta_{N})
\geq\varpi-\left(\frac{4}{\sqrt{3}}+2\sqrt{2}\right)\sqrt{\varepsilon}-\omega_{\varepsilon}(\delta_{N}).
\end{equation*}
Letting $N\to\infty$ and then $\varepsilon\downarrow0$ shows that the
sum converges to $\varpi$.

Hence $(N+1)/\sqrt{\Omega_{N}}\to\varpi$.
\end{proof}

\subsection{Equivalence of the two forms of \eqref{eq:lemniscate-acceleration}}
\label{sec:lemniscate-form-equivalence}
Throughout this subsection, set $\phi_{k}\triangleq(1+\rho^{2}_{k})/(2\rho_{k})$
for $0\leq k\leq N$, $\Delta_{k}\triangleq\phi_{k+1}-\phi_{k}$ for
$0\leq k\leq N-1$, and $\psi_{k}\triangleq(1+\rho^{2}_{k})/(1-\rho^{2}_{k})$
for $1\leq k\leq N+1$ so that the update of $x_{k+1}$ in (\ref{eq:lemniscate-acceleration})
reads $x_{k+1}=x^{+}_{k}+\left(\psi_{k+1}-\psi_{k+2}\right)z_{k+1}$
and the coefficients of (\ref{eq:lemniscate-momentum-form}) read
$\alpha^{\text{(init)}}=\Omega_{N}\Delta_{0}\Delta_{N-1}$, $\alpha^{\text{(mom)}}_{k}=\Delta_{N-k-1}/\Delta_{N-k}$,
and $\alpha^{\text{(corr)}}_{k}=\Delta_{N-k-1}\left(\Omega_{N}\Delta_{k}-\frac{1}{\Delta_{N-k}}\right)$.
Since $1=\rho_{0}>\rho_{1}>\cdots>\rho_{N}>0$ and the map $\rho\mapsto(1+\rho^{2})/(2\rho)$
is strictly decreasing on $(0,1]$, we have $\Delta_{k}>0$ for $0\leq k\leq N-1$.

\begin{lem} \label{lem:lemniscate-form-equivalence} Fix $N\geq1$,
	and let $\Omega_{N}$ and $\rho_{0},\dots,\rho_{N+1}$ be as in Lemma~\ref{lem:lemniscate-recurrence}.
	From any starting point $x_{0}\in\mathbb{R}^{d}$, the iterates $x_{0},x_{1},\dots,x_{N}$
	of (\ref{eq:lemniscate-acceleration}) coincide with those of the
	momentum form (\ref{eq:lemniscate-momentum-form}). \end{lem}

\begin{proof} First, we claim that
	\begin{equation}
		\psi_{k}=\phi_{N+1-k},\quad1\leq k\leq N+1.\label{eq:lemniscate-psi-phi-mirror}
	\end{equation}
	Indeed, for $0<\rho<1$, the map $T(\rho)=(1-\rho)/(1+\rho)$ satisfies
	\[
	\frac{1+T(\rho)^{2}}{2T(\rho)}=\frac{(1+\rho)^{2}+(1-\rho)^{2}}{2(1+\rho)(1-\rho)}=\frac{1+\rho^{2}}{1-\rho^{2}},
	\]
	and combining this identity at $\rho=\rho_{k}$, which lies in $(0,1)$
	for $1\leq k\leq N$, with the relation $\rho_{N+1-k}=T(\rho_{k})$
	recorded after Lemma~\ref{lem:lemniscate-recurrence} in Section~\ref{sec:lemniscate}
	gives $\phi_{N+1-k}=\psi_{k}$ for $1\leq k\leq N$. At the endpoint
	$k=N+1$, where $\rho_{N+1}=0$ and $\rho_{0}=1$, direct evaluation
	gives $\psi_{N+1}=1=\phi_{0}$, which completes (\ref{eq:lemniscate-psi-phi-mirror}).
	By (\ref{eq:lemniscate-psi-phi-mirror}), $\psi_{k+1}-\psi_{k+2}=\phi_{N-k}-\phi_{N-k-1}=\Delta_{N-k-1}$
	for $0\leq k\leq N-1$, so the $x$-update of (\ref{eq:lemniscate-acceleration})
	is
	\begin{equation}
		x_{k+1}=x^{+}_{k}+\Delta_{N-k-1}\,z_{k+1},\quad0\leq k\leq N-1.\label{eq:lemniscate-x-update-compact}
	\end{equation}

	We now argue by induction that the two forms generate the same iterates.
	They start from the same $x_{0}$, and whenever the iterates agree
	up to index $k$, both forms evaluate the gradient at the same points,
	so the vectors $x^{+}_{j}$ for $0\leq j\leq k$ agree as well.

	At $k=0$: since $z_{0}=0$ and $-\nabla f(x_{0})/L=x^{+}_{0}-x_{0}$,
	the $z$-update of (\ref{eq:lemniscate-acceleration}) gives $z_{1}=\Omega_{N}\Delta_{0}\left(x^{+}_{0}-x_{0}\right)$,
	and (\ref{eq:lemniscate-x-update-compact}) yields $x_{1}=x^{+}_{0}+\Omega_{N}\Delta_{0}\Delta_{N-1}\left(x^{+}_{0}-x_{0}\right)=x^{+}_{0}+\alpha^{\text{(init)}}\left(x^{+}_{0}-x_{0}\right)$, the
	first update of (\ref{eq:lemniscate-momentum-form}).

	For $1\leq k\leq N-1$: solving (\ref{eq:lemniscate-x-update-compact})
	at index $k-1$ for the auxiliary variable gives $z_{k}=\left(x_{k}-x^{+}_{k-1}\right)/\Delta_{N-k}$,
	and the $z$-update gives $z_{k+1}=z_{k}+\Omega_{N}\Delta_{k}\left(x^{+}_{k}-x_{k}\right)$.
	Substituting both into (\ref{eq:lemniscate-x-update-compact}) and
	writing $x_{k}-x^{+}_{k-1}=\left(x^{+}_{k}-x^{+}_{k-1}\right)-\left(x^{+}_{k}-x_{k}\right)$,
	we get
	\[
	x_{k+1}=x^{+}_{k}+\frac{\Delta_{N-k-1}}{\Delta_{N-k}}\left(x^{+}_{k}-x^{+}_{k-1}\right)+\Delta_{N-k-1}\left(\Omega_{N}\Delta_{k}-\frac{1}{\Delta_{N-k}}\right)\left(x^{+}_{k}-x_{k}\right),
	\]
	which is the $k$-th update of (\ref{eq:lemniscate-momentum-form})
	with the coefficients $\alpha^{\text{(mom)}}_{k}$ and $\alpha^{\text{(corr)}}_{k}$
	 above. By induction, the two forms produce the same iterates
	$x_{0},x_{1},\dots,x_{N}$. \end{proof}

\subsection{Computer-assisted algorithm design by PEP}
\label{sec:lemniscate-convergence-proof}\label{sec:lemniscate-pep-design}
\paragraph{PEP formulation of the problem.}
This subsection presents, lightly edited for
presentation, the Stage-1 derivation that \textsf{bnb-pep-skill} wrote
for the design problem behind \eqref{eq:lemniscate-acceleration}.
The skill instantiates the Stage 1 design problem of
Section~\ref{subsec:autoopt-stage-numerical} for
problem~\eqref{eq:main-opt}, in the notation of
Section~\ref{subsec:notation}: the function class is
$\mathcal{F}=\mathcal{F}_{0,L}$, the performance measure is
$\mathcal{E}=\norm{\nabla f(x_{N})}^{2}$, the initial condition is
$\mathcal{C}=\norm{x_{0}-x_{\star}}^{2}-R^{2}\leq0$ for a given $R>0$, and
the method class is the class $\mathcal{M}_{N}$ of fixed-step first-order
methods (FSFOMs) with $N$ steps. Applied to $f\in\mathcal{F}_{0,L}$ from a
starting point $x_{0}\in\mathbb{R}^{d}$, an FSFOM produces its iterates
with
\begin{equation*}
x_{i}=x_{i-1}-\frac{1}{L}\sum^{i-1}_{j=0}h_{i,j}\nabla f(x_{j})
\end{equation*}
for $1\leq i\leq N$, where, relative to the generic form
displayed in Section~\ref{subsec:autoopt-stage-numerical}, the stepsizes
$h_{i,j}$ are normalized by $L$. We can equivalently express the FSFOM
as
\begin{equation*}
x_{i}=x_{0}-\frac{1}{L}\sum^{i-1}_{j=0}\alpha_{i,j}\nabla f(x_{j}),
\end{equation*}
where $\alpha_{i,j}\triangleq\sum_{k=j+1}^{i}h_{k,j}$ for
$0\leq j<i\leq N$; this triangular relationship between $h$ and
$\alpha$ is invertible, so we can compute $h$ given $\alpha$ and vice
versa. The stepsizes $h$
or $\alpha$ may depend on $\mathcal{F}_{0,L}$ and the value of
$N$, but are otherwise predetermined.

Instantiated with these choices, the worst-case
performance $\mathcal{R}(M)$ of a method $M\in\mathcal{M}_{N}$ in
\eqref{eq:autoopt-outer-design} becomes
\begin{align}
\mathcal{R}\left(M\right) & =\left(\begin{array}{ll}
\textrm{maximize}\quad\norm{\nabla f(x_{N})}^{2}\\
\textrm{subject to}\\
f\in\mathcal{F}_{0,L},\\
x_{\star}\textrm{ is an optimal solution to }\eqref{eq:main-opt}\text{ satisfying }\nabla f(x_{\star})=0,\\
\{x_{i}\}_{1\leq i\leq N}\textrm{ is generated by \ensuremath{M} with initial point \ensuremath{x_{0}}},\\
\norm{x_{0}-x_{\star}}^{2}\leq R^{2},\\
x_{\star}=0,\,f(x_{\star})=0
\end{array}\right),\tag{\ensuremath{\mathcal{O}^{\textrm{inner}}}}\label{eq:worst-case-pfm}
\end{align}
where $f$, $x_{0},\ldots,x_{N}$, and $x_{\star}$ are the decision
variables. We set $x_{\star}=0$ and $f(x_{\star})=0$ without any
loss of generality because the function class $\mathcal{F}_{0,L}$
is closed and invariant under shifting variables and function values.

The optimal FSFOM $M^{\star}_{N}\in\mathcal{M}_{N}$ for
this setup solves the outer design problem
$\mathcal{R}^{\star}(\mathcal{M}_{N})=\min_{M\in\mathcal{M}_{N}}\mathcal{R}(M)$
of \eqref{eq:autoopt-outer-design}.

Following the BnB-PEP reformulation outlined in
Section~\ref{subsec:autoopt-stage-numerical}, the
derivation reduces this design problem
to a finite-dimensional nonconvex QCQP in two steps: it
first formulates the
inner problem \eqref{eq:worst-case-pfm} as a finite-dimensional convex SDP, and then
optimizes the stepsizes against its dual \cite{dasgupta2022BnBPEP}.

As the first step, the derivation formulates the inner problem
\eqref{eq:worst-case-pfm} as a finite-dimensional convex SDP using the
smooth-convex interpolation result of
\cite[Corollary 1]{taylor2017smooth}. This result exactly characterizes the
sampled function values and gradients that arise from a function in
$\mathcal F_{0,L}$.

In other words, we impose the smooth-convex interpolation inequalities at the
sampled points $x_i$, $i\in I_N^\star$, using
\eqref{eq:F_0L_formula}, with the notation
$g_{i}\triangleq\nabla f(x_{i})$ and $f_{i}\triangleq f(x_{i})$. This
finite representation is lossless by
\cite[Corollary 1]{taylor2017smooth}. Thus, the derivation
casts \eqref{eq:worst-case-pfm}
as the following finite-dimensional but nonconvex problem:

\begin{align}
\mathcal{R}(M)=\left(\begin{array}{ll}
\textrm{maximize} & \norm{g_{N}}^{2}\\
\textrm{subject to} & f_{i}\geq f_{j}+\langle g_{j}, x_{i}-x_{j}\rangle+\frac{1}{2L}\norm{g_{i}-g_{j}}^{2},\quad i,j\in I^{\star}_{N},\\
 & g_{\star}=0,\\
 & x_{i}=x_{0}-\frac{1}{L}\sum^{i-1}_{j=0}\alpha_{i,j}g_{j},\quad 1\leq i\leq N,\\
 & \norm{x_{0}-x_{\star}}^{2}\leq R^{2},\\
 & x_{\star}=0,\;f_{\star}=0
\end{array}\right),\label{eq:worst-case-pfm-fmuL-1}
\end{align}
where the decision variables are $\{x_{i},g_{i},f_{i}\}_{i\in I^{\star}_{N}}\subseteq\mathbb{R}^{d}\times\mathbb{R}^{d}\times\mathbb{R}$.

Next, \eqref{eq:worst-case-pfm-fmuL-1} is formulated as a convex SDP.
Let $P\triangleq[x_{0}\mid g_{0}\mid g_{1}\mid\ldots\mid g_{N}]\in\mathbb{R}^{d\times(N+2)},$
$G\triangleq P^{\top}P\in\mathbb{S}^{N+2}_{+},$ and $F\triangleq[f_{0}\mid f_{1}\mid\ldots\mid f_{N}]\in\mathbb{R}^{1\times(N+1)}.$
Note that $\mathop{\textbf{rank}}G\leq d$. Define the following notation
for selecting columns and elements of $P$ and $F$: 
\begin{equation*}
\begin{alignedat}{1} & \mathbf{g}_{\star}=0\in\mathbb{R}^{N+2},\;\mathbf{g}_{i}=e_{i+2}\in\mathbb{R}^{N+2},\quad 0\leq i\leq N,\\
 & \mathbf{x}_{0}=e_{1}\in\mathbb{R}^{N+2},\;\mathbf{x}_{\star}=0\in\mathbb{R}^{N+2},\;\mathbf{x}_{i}=\mathbf{x}_{0}-\frac{1}{L}\sum^{i-1}_{j=0}\alpha_{i,j}\mathbf{g}_{j}\in\mathbb{R}^{N+2},\quad 1\leq i\leq N,\\
 & \mathbf{f}_{\star}=0\in\mathbb{R}^{N+1},\;\mathbf{f}_{i}=e_{i+1}\in\mathbb{R}^{N+1},\quad 0\leq i\leq N.
\end{alignedat}
\end{equation*}
This notation is defined so that $x_{i}=P\mathbf{x}_{i},\;g_{i}=P\mathbf{g}_{i},\;f_{i}=F\mathbf{f}_{i}$
for $i\in I^{\star}_{N}$. Note that $\mathbf{x}_{i}$ depends on
$\{\alpha_{i,j}\}_{0\leq j\leq i-1}$ linearly for $1\leq i\leq N$.
For $i,j\in I^{\star}_{N}$,
now define
\begin{equation}
\begin{alignedat}{1} & A_{i,j}(\alpha)=\mathbf{g}_{j}\odot(\mathbf{x}_{i}-\mathbf{x}_{j})\in\mathbb{S}^{N+2},\\
 & B_{i,j}(\alpha)=(\mathbf{x}_{i}-\mathbf{x}_{j})\odot(\mathbf{x}_{i}-\mathbf{x}_{j})\in\mathbb{S}^{N+2}_{+},\\
 & C_{i,j}=(\mathbf{g}_{i}-\mathbf{g}_{j})\odot(\mathbf{g}_{i}-\mathbf{g}_{j})\in\mathbb{S}^{N+2}_{+}.
\end{alignedat}
\label{eq:ABCa-mat-vec}
\end{equation}

Note that $A_{i,j}(\alpha)$ is affine and $B_{i,j}(\alpha)$
is quadratic as functions of $\{\alpha_{i,j}\}_{0\leq j<i\leq N}$. Under this notation, $\left\langle g_{j},x_{i}-x_{j}\right\rangle =\mathbf{tr}GA_{i,j}(\alpha)$,  $\norm{x_{0}-x_{\star}}^{2}=\mathbf{tr}GB_{0,\star},$
and $\norm{g_{i}-g_{j}}^{2}=\mathbf{tr}GC_{i,j}$ for $i,j\in I^{\star}_{N}$. In these terms, \eqref{eq:worst-case-pfm-fmuL-1} becomes a
maximization over $F$ and $G=P^{\top}P$, at the price of the rank
constraint $\mathop{\textbf{rank}}G\leq d$: the equivalence relies on the
fact that any $G\in\mathbb{S}^{N+2}_{+}$ with
$\mathop{\textbf{rank}}G\leq d$ factors as $G=P^{\top}P$ with
$P\in\mathbb{R}^{d\times(N+2)}$; see
\cite[$\mathsection$3.2]{taylor2017smooth}. The rank constraint is removed
with the following large-scale assumption.

\begin{assumption} \label{large-scale-assumption} We have $d\geq N+2$.
\end{assumption}

Under this assumption, the constraint $\mathop{\textbf{rank}}G\leq d$
becomes vacuous, since $G\in\mathbb{S}^{N+2}_{+}$. The \textsf{bnb-pep-skill}
derivation drops the rank
constraint and formulates \eqref{eq:worst-case-pfm} as a convex SDP
\begin{align}
\mathcal{R}(M)= & \left(\begin{array}{l}
\textrm{maximize}\quad\mathbf{tr}G\left(C_{N,\star}\right)\\
\textrm{subject to}\\
F(\mathbf{f}_{j}-\mathbf{f}_{i})+\mathbf{tr}G\Big[A_{i,j}(\alpha)+\frac{1}{2L}C_{i,j}\Big]\leq0,\quad i,j\in I^{\star}_{N}:i\neq j,\quad\mathbin{{\color{annotgray}\rhd}}{\color{annotgray}\,\textsf{dual var.}\,\lambda_{i,j}\geq0}\\
-G\preceq0,\quad\mathbin{{\color{annotgray}\rhd}}{\color{annotgray}\,\textsf{dual var.}\,\ensuremath{Z\succeq0}}\\
\mathbf{tr}GB_{0,\star}\leq R^{2},{\color{annotgray}\quad\rhd\,\textsf{dual var.}\,\texttt{\ensuremath{\nu\geq0}}}
\end{array}\right),\label{eq:worst-case-primal}
\end{align}
where $F\in\mathbb{R}^{N+1},$ $G\in\mathbb{S}^{N+2}$ are the decision
variables. We denote the corresponding dual variables on the right-hand side of the constraints for later use. Note that \eqref{eq:worst-case-primal}
is free from problem dimension $d$.

Next the \textsf{bnb-pep-skill}
derivation uses convex duality to formulate the SDP \eqref{eq:worst-case-primal}
in the previous step, originally a maximization problem, as a minimization
problem. Taking the dual of \eqref{eq:worst-case-primal} gives
\begin{align}
\mathcal{\overline{R}}(M)= & \left(\begin{array}{l}
\textrm{minimize}\quad\nu R^{2}\\
\textrm{subject to}\\
\sum_{i,j\in I^{\star}_{N}:i\neq j}\lambda_{i,j}(\mathbf{f}_{j}-\mathbf{f}_{i})=0,\\
\nu B_{0,\star}-C_{N,\star}+\sum_{i,j\in I^{\star}_{N}:i\neq j}\lambda_{i,j}\left[A_{i,j}(\alpha)+\frac{1}{2L}C_{i,j}\right]=Z,\\
Z\succeq0,\\
\nu\geq0,\;\lambda_{i,j}\geq0,\quad i,j\in I^{\star}_{N}:i\neq j
\end{array}\right),\label{eq:worst-case-pfm-dual-1}
\end{align}
where $\nu\in\mathbb{R}$, $\lambda=\{\lambda_{i,j}\}_{i,j\in I^{\star}_{N}:i\neq j}$,
and $Z\in\mathbb{S}^{N+2}_{+}$ are the decision variables. We call
$\lambda,\nu,$ and $Z$ the \emph{inner-dual variables}. By weak
duality of convex SDPs, we have $\mathcal{R}(M)\le\overline{\mathcal{R}}(M).$

In convex SDPs, strong duality holds often but not always. For the
sake of simplicity, we will assume strong duality holds in our setup.
\begin{assumption}\label{strong-duality-assumption} Strong duality
holds between \eqref{eq:worst-case-primal} and \eqref{eq:worst-case-pfm-dual-1},
i.e., $\mathcal{R}(M)=\overline{\mathcal{R}}(M)$.
\end{assumption}

As the second step, with the inner problem
\eqref{eq:worst-case-pfm} replaced by its dual minimization
\eqref{eq:worst-case-pfm-dual-1}, the outer design problem becomes a joint
minimization over the stepsizes $\{\alpha_{i,j}\}_{0\leq j<i\leq N}$ and
the inner-dual variables. This joint problem is nonconvex, and, following
\cite{dasgupta2022BnBPEP}, the derivation formulates it as a QCQP by replacing the
semidefinite constraint $Z\succeq0$ with the Cholesky factorization
$Z=VV^{\top}$, where $V$ is lower triangular with nonnegative diagonals
\cite[Corollary 7.2.9]{horn2012matrix}. The problem of finding the optimal
FSFOM then takes the QCQP form
\begin{align}
\mathcal{R}^{\star}(\mathcal{M}_{N})= & \left(\begin{array}{l}
\textrm{minimize}\quad\nu R^{2}\\
\textrm{subject to}\\
\sum_{i,j\in I^{\star}_{N}:i\neq j}\lambda_{i,j}(\mathbf{f}_{j}-\mathbf{f}_{i})=0,\\
\nu B_{0,\star}-C_{N,\star}+\sum_{i,j\in I^{\star}_{N}:i\neq j}\lambda_{i,j}\left[A_{i,j}(\alpha)+\frac{1}{2L}C_{i,j}\right]=Z,\\
V\text{ is lower triangular with nonnegative diagonals},\\
Z=VV^{\top},\\
\nu\geq0,\;\lambda_{i,j}\geq0,\quad i,j\in I^{\star}_{N}:i\neq j
\end{array}\right),\label{eq:BnB-PEP-Preli}
\end{align}
where $\lambda=\{\lambda_{i,j}\}_{i,j\in I^{\star}_{N}:i\neq j}$,
$\nu$, $Z$, $V$, $\{\alpha_{i,j}\}_{0\leq j<i\leq N}$ are
the decision variables.

Upon user approval of the derivation above,
\textsf{bnb-pep-skill} solves \eqref{eq:BnB-PEP-Preli} numerically;
the resultant optimal $\alpha^{\star}$ gives the optimal stepsizes
$h^{\star}$ for any large-scale setup $d\geq N+2$. By the scale
invariance discussed in \cite[$\mathsection$3.5]{taylor2017smooth}, it
suffices to solve the QCQP for $L=1$ and $R=1$: for any other $L>0$ and
$R>0$ the optimal stepsizes are unchanged, and the optimal worst-case
value scales by the factor $L^{2}R^{2}$.

In Stage 2 of the pipeline, the
\textsf{frontier-llm-consult} skill fitted symbolic expressions to the
numerical solutions of \eqref{eq:BnB-PEP-Preli}, yielding
an analytic parametrization of the numerically optimal
stepsizes in the design variables $\alpha$, equivalently $h$, together
with an analytic feasible point of
the inner dual \eqref{eq:worst-case-pfm-dual-1}, whose feasibility
already certifies the worst-case bound by weak duality.
From this stepsize parametrization, the consultation
then recast the method into the memory-efficient form
\eqref{eq:lemniscate-acceleration} presented in
Section~\ref{sec:lemniscate}. In continued
consultation, Stage 2 converted the multipliers of this feasible point
into the Lyapunov proof of Theorem~\ref{thm:lemniscate-convergence}
given in Section~\ref{sec:lemniscate-lyapunov}, the form in which this
paper presents the convergence analysis; the Lean coverage of the
presented results is described in Section~\ref{sec:lemniscate-lean}.

\subsection{Details for continuous-time analysis of lemniscate ODE}
\label{sec:app-continuous-time}
In this section, we develop a leading-order continuous-time model of the
lemniscate acceleration algorithm and prove Theorem~\ref{thm:continuous-time-lyap}. For an $L_f$-smooth objective with
$L_f>0$, we apply the method with the algorithmic smoothness upper bound
$\widehat L_N=N^{2}/T^{2}$ and use the iterate grid $t=kT/N$. At fixed
$L_f$ and $T$, the condition $L_f\leq\widehat L_N$ holds for all sufficiently
large $N$. We retain only the leading-order terms as $N\rightarrow\infty$.

Recall the continuous-time coefficients $\rho(t)$ and $\sigma(t)$
are given by 
\begin{equation}
\begin{aligned}\rho(t) & =\mathrm{cl}^{2}\left(\frac{\varpi_{T}}{2}t\right),\\
\sigma(t) & =\sqrt{\frac{1-\rho(t)^{2}}{\rho(t)}},
\end{aligned}
\label{eq:def-sigma}
\end{equation}
where $\varpi_{T}\triangleq\varpi/T$ and $\mathrm{cl}(x)$ is the
lemniscate cosine function in Section~\ref{sec:lemniscate}. We prove
some identities of $\rho(t)$ and $\sigma(t)$ used in the derivation
of the ODE and the Lyapunov function.

We first collect the properties of the lemniscate elliptic functions
that these derivations use. Recall from
Section~\ref{sec:lemniscate-constant} that
$\operatorname{arcsl}(u)=\int_{0}^{u}dx/\sqrt{1-x^{4}}$, that
$\operatorname{sl}$ is the inverse of $\operatorname{arcsl}$, and that
$\operatorname{cl}(x)=\operatorname{sl}(\varpi/2-x)$.

\begin{lem}[Lemniscatic calculus]\label{lem:lemniscatic-calculus}
The following hold.
\begin{enumerate}
\item $\operatorname{sl}$ is a continuous strictly increasing bijection
of $[0,\varpi/2]$ onto $[0,1]$ with $\operatorname{sl}(0)=0$ and
$\operatorname{sl}(\varpi/2)=1$; consequently $\operatorname{cl}$ is a
continuous strictly decreasing bijection of $[0,\varpi/2]$ onto
$[0,1]$, and $\operatorname{sl}(x),\operatorname{cl}(x)\in(0,1)$ for
$x\in(0,\varpi/2)$.
\item $\operatorname{sl}'(x)=\sqrt{1-\operatorname{sl}^{4}(x)}$ and
$\operatorname{cl}'(x)=-\sqrt{1-\operatorname{cl}^{4}(x)}$ for
$x\in(0,\varpi/2)$.
\item $\operatorname{cl}(\varpi/2-x)=\operatorname{sl}(x)$ for
$x\in[0,\varpi/2]$.
\item $\operatorname{cl}^{2}(x)=\dfrac{1-\operatorname{sl}^{2}(x)}{1+\operatorname{sl}^{2}(x)}$
for $x\in[0,\varpi/2]$; equivalently,
$\operatorname{cl}^{2}(x)+\operatorname{sl}^{2}(x)+\operatorname{cl}^{2}(x)\operatorname{sl}^{2}(x)=1$.
\item $\operatorname{sl}(x)=x+o(x)$ as $x\downarrow0$.
\item For $T>0$, the coefficient (\ref{eq:def-sigma}) satisfies $0<\rho(t)<1$
for $t\in(0,T)$. Its symmetry identity $\rho(T-t)=\left(1-\rho(t)\right)/\left(1+\rho(t)\right)$
holds for $t\in[0,T]$.
\end{enumerate}
\end{lem}

\begin{proof}
\emph{Part 1.} On $[0,1)$ the integrand $1/\sqrt{1-x^{4}}$ is positive,
and near $x=1$ it is bounded by $1/\sqrt{1-x}$, so
$\operatorname{arcsl}$ is a continuous strictly increasing bijection of
$[0,1]$ onto $[0,\varpi/2]$; here $\operatorname{arcsl}(1)=\varpi/2$ is
the definition of $\varpi$. Its inverse $\operatorname{sl}$ inherits
these properties, and $\operatorname{cl}(x)=\operatorname{sl}(\varpi/2-x)$
inherits the rest.

\emph{Part 2.} For $x\in(0,\varpi/2)$ the inverse function theorem
gives
$\operatorname{sl}'(x)=1/\operatorname{arcsl}'(\operatorname{sl}(x))=\sqrt{1-\operatorname{sl}^{4}(x)}$,
and the chain rule gives
$\operatorname{cl}'(x)=-\operatorname{sl}'(\varpi/2-x)=-\sqrt{1-\operatorname{cl}^{4}(x)}$.

\emph{Part 3.} Substitute $\varpi/2-x$ for $x$ in the definition
$\operatorname{cl}(x)=\operatorname{sl}(\varpi/2-x)$.

\emph{Part 4.} For $u\in[0,1]$ set
$v(u)\triangleq\sqrt{(1-u^{2})/(1+u^{2})}$, so that
$v(u)(1+u^{2})=\sqrt{1-u^{4}}$ and
$\sqrt{1-v(u)^{4}}=2u/(1+u^{2})$. For $u\in(0,1)$,
\begin{equation*}
v'(u)=\frac{-2u}{v(u)\left(1+u^{2}\right)^{2}},
\qquad\text{so}\qquad
\frac{v'(u)}{\sqrt{1-v(u)^{4}}}
=\frac{-1}{v(u)\left(1+u^{2}\right)}
=\frac{-1}{\sqrt{1-u^{4}}}.
\end{equation*}
Hence $\frac{d}{du}\left[\operatorname{arcsl}(u)+\operatorname{arcsl}(v(u))\right]=0$
on $(0,1)$, and by continuity and the value
$\operatorname{arcsl}(0)+\operatorname{arcsl}(1)=\varpi/2$ at $u=0$,
\begin{equation*}
\operatorname{arcsl}(u)+\operatorname{arcsl}\left(v(u)\right)=\frac{\varpi}{2},
\qquad u\in[0,1].
\end{equation*}
Taking $u=\operatorname{sl}(x)$ for $x\in[0,\varpi/2]$ gives
$\operatorname{arcsl}(v(\operatorname{sl}(x)))=\varpi/2-x$, that is,
$v(\operatorname{sl}(x))=\operatorname{sl}(\varpi/2-x)=\operatorname{cl}(x)$
by part 3. Squaring yields the first identity, and clearing
denominators yields the second.

\emph{Part 5.} By part 1, $\operatorname{sl}(0)=0$, and by the mean
value theorem
$\operatorname{sl}(x)=\operatorname{sl}'(\xi_{x})\,x$ for some
$\xi_{x}\in(0,x)$, where
$\operatorname{sl}'(\xi_{x})=\sqrt{1-\operatorname{sl}^{4}(\xi_{x})}\to1$
as $x\downarrow0$ by part 1 and continuity.

\emph{Part 6.} For $t\in(0,T)$ we have
$\varpi_{T}t/2\in(0,\varpi/2)$, so
$\rho(t)=\operatorname{cl}^{2}(\varpi_{T}t/2)\in(0,1)$ by part 1.
For the symmetry, parts 3 and 4 give
\begin{equation*}
\rho(T-t)
=\operatorname{cl}^{2}\left(\frac{\varpi}{2}-\frac{\varpi_{T}t}{2}\right)
=\operatorname{sl}^{2}\left(\frac{\varpi_{T}t}{2}\right)
=\frac{1-\operatorname{cl}^{2}(\varpi_{T}t/2)}{1+\operatorname{cl}^{2}(\varpi_{T}t/2)}
=\frac{1-\rho(t)}{1+\rho(t)}. \qedhere
\end{equation*}
\end{proof}

\begin{lem}\label{lem:sigma-identities} The following identities
hold for $t\in(0,T)$: 
\begin{align*}
 & \dot{\rho}(t)=-\varpi_{T}\sqrt{\rho(t)(1-\rho(t)^{2})}=-\varpi_{T}\sigma(t)\rho(t),\\
 & \dot{\sigma}(t)=\frac{\varpi_{T}}{2}\left(\rho(t)+\frac{1}{\rho(t)}\right),\\
 & \Ddot{\sigma}(t)=\frac{\varpi^{2}_{T}}{2}\sigma(t)^{3},\\
 & \sigma(t)\,\sigma(T-t)=2.
\end{align*}
\end{lem} \begin{proof} Throughout the proof,
$0<\rho(t)<1$ and hence $\sigma(t)>0$ on $(0,T)$, by part~6 of
Lemma~\ref{lem:lemniscatic-calculus}, so every division below is
legitimate. By the derivative $\mathrm{cl}'(x)=-\sqrt{1-\mathrm{cl}^{4}(x)}$
(part~2 of Lemma~\ref{lem:lemniscatic-calculus},
with $\mathrm{cl}\geq0$ from part~1)
and (\ref{eq:def-sigma}), $\rho(t)$ satisfies the first identity:
$\dot{\rho}(t)=-\varpi_{T}\sqrt{\rho(t)(1-\rho(t)^{2})}=-\varpi_{T}\sigma(t)\rho(t).$
Differentiating $\sigma(t)^{2}=1/\rho(t)-\rho(t)$ and using the first
identity gives 
\begin{align*}
2\sigma(t)\dot{\sigma}(t) & =-\left(1+\frac{1}{\rho(t)^{2}}\right)\dot{\rho}(t)=\varpi_{T}\left(1+\frac{1}{\rho(t)^{2}}\right)\rho(t)\sigma(t),
\end{align*}
and dividing both sides by $2\sigma(t)$ gives the second identity.
Differentiating the second identity and using the first identity gives
the third identity: 
\begin{align*}
	\Ddot{\sigma}(t) & =\frac{\varpi_{T}}{2}\left(1-\frac{1}{\rho(t)^{2}}\right)\dot{\rho}(t)=\frac{\varpi^{2}_{T}}{2}\frac{(1-\rho(t)^{2})^{3/2}}{\rho(t)^{3/2}}=\frac{\varpi^{2}_{T}}{2}\sigma(t)^{3}.
\end{align*}

For the last identity, we use the symmetry
$\rho(T-t)=\left(1-\rho(t)\right)/\left(1+\rho(t)\right)$ from part~6
of Lemma~\ref{lem:lemniscatic-calculus}.
Thus
\begin{align*}
\sigma(T-t)^{2} & =\frac{1-\rho(T-t)^{2}}{\rho(T-t)}=\frac{4\rho(t)}{1-\rho(t)^{2}}=\frac{4}{\sigma(t)^{2}},
\end{align*}
and since $\sigma(t)>0$, this proves the last identity. \end{proof}

\paragraph{The continuous-time ODE.}

We sketch a derivation of the ODE by identifying the leading-order
behavior of~\eqref{eq:lemniscate-acceleration} as $N\rightarrow\infty$ with
$\widehat L_N=N^{2}/T^{2}$. On the iterate grid $t_k^x=k\,T/N$, we identify
$X(t_k^x)\approx x_{k}$ and
$Z(t_k^x)\approx \bigl(N/(\sqrt{\Omega_{N}}\,T)\bigr)\,z_{k}$. The
coefficient $\rho_k$ is naturally placed at
$t_k^\rho=k\,T/(N+1)$; for $0\leq k\leq N$,
$|t_k^x-t_k^\rho|\leq T/(N+1)=O(1/N)$, so we identify
$\rho(t_k^x)\approx\rho_k$ at leading order. The
updates of~(\ref{eq:lemniscate-acceleration}) are

\begin{align*}
  & z_{k+1}=z_{k}-\Omega_{N}\Bigl(\dfrac{1+\rho^{2}_{k+1}}{2\rho_{k+1}}-\dfrac{1+\rho^{2}_{k}}{2\rho_{k}}\Bigr)\frac{1}{\widehat L_N}\nabla f(x_{k}),\\
 & x^{+}_{k}=x_{k}-\frac{1}{\widehat L_N}\nabla f(x_{k}),\\
 & x_{k+1}=x^{+}_{k}+\Bigl(\dfrac{1+\rho^{2}_{k+1}}{1-\rho^{2}_{k+1}}-\dfrac{1+\rho^{2}_{k+2}}{1-\rho^{2}_{k+2}}\Bigr)z_{k+1}.
\end{align*}

The $z$-update coefficient, after division by $\sqrt{\Omega_{N}}$ to account for the rescaling of
$Z$, has the following leading-order form:

\[
\sqrt{\Omega_{N}}\left.\frac{1+x^{2}}{2x}\right|^{\rho_{k+1}}_{\rho_{k}}\approx\sqrt{\Omega_{N}}(\rho_{k}-\rho_{k+1})\frac{1-\rho^{2}_{k}}{2\rho^{2}_{k}}\approx\sqrt{\rho(1-\rho^{2})}\frac{1-\rho^{2}}{2\rho^{2}}=\frac{1}{2}\left(\frac{1-\rho^{2}}{\rho}\right)^{3/2}=\frac{1}{2}\sigma^{3},
\]
which follows from the mean value theorem, the continuity of $\rho(t)$, and the recurrence~\eqref{eq:lemniscate-recurrence}.
For the $x$-update, define $\phi_k=(1+\rho_k^2)/(2\rho_k)$. The symmetry
$\rho_{N+1-k}=(1-\rho_k)/(1+\rho_k)$ from
Section~\ref{sec:lemniscate} rewrites the coefficient in the $x$-update as
$\phi_{N-k}-\phi_{N-k-1}$. Applying the preceding calculation at $T-t$ gives
$\sqrt{\Omega_N}(\phi_{N-k}-\phi_{N-k-1}) \approx \sigma(T-t)^{3}/2$.
Since $1/\widehat L_N=T^2/N^2$ and $\Delta t=T/N$, the updates give
$\Delta Z\approx-(\sigma^3/2)\nabla f(X)\,\Delta t$ and
$\Delta X\approx(\sigma(T-t)^3/2)Z\,\Delta t$ at leading order. The
direct gradient step $x_{k}^{+}-x_{k}$ carries the smaller factor
$1/\widehat L_N=O(\Delta t^2)$ and hence does not contribute at this order.

This development suggests the coupled ODE for $X(t)$ and $Z(t)$ on $(0,T)$:
\begin{align*}
 & \dot{Z}(t)=-\frac{\sigma(t)^{3}}{2}\nabla f(X(t)),\\
 & \dot{X}(t)=\frac{\sigma(T-t)^{3}}{2}Z(t)=\frac{4}{\sigma(t)^{3}}Z(t).
\end{align*}
Differentiating $Z(t)=\sigma(t)^{3}\dot{X}(t)/4$ and using $\dot{Z}(t)=-\sigma(t)^{3}\nabla f(X(t))/2$
gives the following second-order ODE for $X(t)$:
\begin{equation}
\Ddot{X}(t)+\gamma(t)\dot{X}(t)+2\nabla f\left(X(t)\right)=0,\quad0<t<T,\label{eq:ODE}
\end{equation}
where $\gamma(t)\triangleq3\dot{\sigma}(t)/\sigma(t)$ is the coefficient
of the friction term.

\begin{lem}\label{lem:ODE-expanded} The ODE in (\ref{eq:ODE}) can
be written as 
\[
\Ddot{X}(t)+\frac{3\varpi}{2T}\left[\frac{\mathrm{cl}\!\left(\frac{\varpi t}{2T}\right)}{\mathrm{sl}\!\left(\frac{\varpi t}{2T}\right)}+\frac{\mathrm{sl}\!\left(\frac{\varpi t}{2T}\right)}{\mathrm{cl}\!\left(\frac{\varpi t}{2T}\right)}\right]\dot{X}(t)+2\nabla f\left(X(t)\right)=0,\quad0<t<T.
\]
\end{lem} \begin{proof} We only need to calculate the friction term
$\gamma(t)=3\dot{\sigma}(t)/\sigma(t)$ in (\ref{eq:ODE}). By the
second identity of Lemma~\ref{lem:sigma-identities}, we have 
\[
\frac{\dot{\sigma}}{\sigma}=\frac{\varpi_{T}}{2}\frac{\rho+\rho^{-1}}{\sqrt{(1-\rho^{2})/\rho}}=\frac{\varpi_{T}}{2}\frac{1+\rho^{2}}{\sqrt{\rho(1-\rho^{2})}}.
\]
Substituting $\rho(t)=\mathrm{cl}^{2}(\varpi_{T}t/2)$ gives 
\begin{align*}
\gamma(t) & =\frac{3\varpi_{T}}{2}\frac{1+\mathrm{cl}^{4}(\varpi_{T}t/2)}{\mathrm{cl}(\varpi_{T}t/2)\sqrt{1-\mathrm{cl}^{4}(\varpi_{T}t/2)}}\\
 & =\frac{3\varpi_{T}}{2}\frac{1+\mathrm{cl}^{4}(\varpi_{T}t/2)}{\mathrm{cl}(\varpi_{T}t/2)\,\mathrm{sl}(\varpi_{T}t/2)\left(1+\mathrm{cl}^{2}(\varpi_{T}t/2)\right)}\\
 & =\frac{3\varpi_{T}}{2}\frac{\mathrm{cl}^{2}(\varpi_{T}t/2)+\mathrm{sl}^{2}(\varpi_{T}t/2)}{\mathrm{cl}(\varpi_{T}t/2)\,\mathrm{sl}(\varpi_{T}t/2)}\\
 & =\frac{3\varpi_{T}}{2}\left[\frac{\mathrm{cl}(\varpi_{T}t/2)}{\mathrm{sl}(\varpi_{T}t/2)}+\frac{\mathrm{sl}(\varpi_{T}t/2)}{\mathrm{cl}(\varpi_{T}t/2)}\right].
\end{align*}
where we use $\mathrm{cl}^{2}(x)+\mathrm{sl}^{2}(x)+\mathrm{cl}^{2}(x)\mathrm{sl}^{2}(x)=1$
(part~4 of Lemma~\ref{lem:lemniscatic-calculus}, which
also gives $\sqrt{1-\mathrm{cl}^{4}}=\mathrm{sl}\,(1+\mathrm{cl}^{2})$
for the second equality)
to simplify the form. \end{proof}

\begin{proof}[Proof of Theorem~\ref{thm:continuous-time-lyap}]
Assume that $X\in C^{2}([0,T];\mathbb{R}^{d})$. Denote $x_{T}\triangleq X(T)$
and $\varpi_{T}\triangleq\varpi/T$, and set $Z(t)\triangleq\sigma(t)^{3}\dot{X}(t)/4$
for $t\in(0,T)$. Then the second-order ODE and the definition of
$\gamma$ imply $\dot{Z}(t)=-\sigma(t)^{3}\nabla f(X(t))/2$. The
Lyapunov function $\mathcal{V}(t)$ is defined for $t\in(0,T)$ as
\begin{equation}
\begin{aligned}\mathcal{V}(t) & \triangleq\underbrace{\frac{1-\rho^{2}}{2\rho}\left(f(X(t))-f(x_{\star})\right)-\frac{(1-\rho)^{2}}{2\rho}\left(f(x_{T})-f(x_{\star})\right)}_{A(t)}\\
 & \quad+\underbrace{\frac{\varpi^{2}_{T}}{2}\left(\norm{X(t)-x_{\star}+\frac{1+\rho^{2}}{1-\rho^{2}}\frac{Z(t)}{\varpi_{T}}}^{2}-\norm{x_{T}-x_{\star}+\frac{Z(t)}{\varpi_{T}}}^{2}\right)}_{B(t)}.
\end{aligned}
\label{eq:lyapunov}
\end{equation}
Here and below, $\rho$, $\sigma$, $X$, $Z$, and their derivatives
are evaluated at $t$ when their arguments are omitted. We calculate
the derivative of function-value terms $A(t)$ and squared-norm terms
$B(t)$ separately. The function-value terms can be simplified as
follows: 
\begin{align*}
A(t) & =\frac{1-\rho^{2}}{2\rho}\left(f(X)-f(x_{\star})\right)-\frac{(1-\rho)^{2}}{2\rho}\left(f(x_{T})-f(x_{\star})\right)\\
 & =(1-\rho)\left(f(X)-f(x_{\star})\right)+\frac{(1-\rho)^{2}}{2\rho}\left(f(X)-f(x_{T})\right).
\end{align*}
Using the identity $\dot{\rho}=-\varpi_{T}\sigma\rho$ of Lemma~\ref{lem:sigma-identities}
and $\sigma^{2}=\rho^{-1}-\rho$, we have 
\[
\frac{d}{dt}\left(\frac{(1-\rho)^{2}}{2\rho}\right)=\frac{\varpi_{T}\sigma^{3}}{2},
\]
so the derivative of the function-value terms is 
\begin{align*}
\dot{A}(t) & =\varpi_{T}\sigma\rho\left(f(X)-f(x_{\star})\right)+\frac{\varpi_{T}\sigma^{3}}{2}\left(f(X)-f(x_{T})\right)+\frac{2}{\sigma}\left\langle Z,\nabla f(X)\right\rangle .
\end{align*}

For the squared-norm terms $B(t)$, the identities of Lemma~\ref{lem:sigma-identities}
give 
\[
\frac{d}{dt}\left(\frac{1+\rho^{2}}{1-\rho^{2}}\right)=-\frac{4\varpi_{T}}{\sigma^{3}}.
\]
Consequently, 
\[
\frac{d}{dt}\left(X-x_{\star}+\frac{1+\rho^{2}}{1-\rho^{2}}\frac{Z}{\varpi_{T}}\right)=\frac{1+\rho^{2}}{1-\rho^{2}}\frac{\dot{Z}}{\varpi_{T}},
\]
where $\dot{X}=4Z/\sigma^{3}$ cancels $-4Z/\sigma^{3}$. Differentiating
the squared-norm terms $B(t)$, substituting $\dot{Z}=-\sigma^{3}\nabla f(X)/2$,
and using $\sigma^{2}=\rho^{-1}-\rho$ gives 
\[
\dot{B}(t)=-\varpi_{T}\sigma\rho\left\langle X-x_{\star},\nabla f(X)\right\rangle -\frac{\varpi_{T}\sigma^{3}}{2}\left\langle X-x_{T},\nabla f(X)\right\rangle -\frac{2}{\sigma}\left\langle Z,\nabla f(X)\right\rangle .
\]
Adding the two derivatives $\dot{A}(t)$ and $\dot{B}(t)$, the terms
involving $\langle Z,\nabla f(X)\rangle$ cancel, and we obtain $\dot{\mathcal{V}}(t)=-\varpi_{T}\sigma\rho D_{f}(x_{\star},X)-(\varpi_{T}\sigma^{3}/2)D_{f}(x_{T},X)\leq0,$
where $D_{f}(y,x)\triangleq f(y)-f(x)-\langle y-x,\nabla f(x)\rangle$,
and the inequality follows from convexity of $f$, in the form $D_{f}(y,x)\geq0$.
Thus $\mathcal{V}(t)$ is non-increasing on $(0,T)$.

For limits of the endpoints $t\downarrow0$ and $t\uparrow T$, we
substitute $Z=\sigma^{3}\dot{X}/4$ and expand the two squared norms
in (\ref{eq:lyapunov}) to expose the cancellation between
potentially singular terms as $t\uparrow T$.
The expanded form of (\ref{eq:lyapunov}) is: 
\begin{equation}
\begin{aligned}\mathcal{V}(t)={} & (1-\rho)\left(f(X)-f(x_{\star})\right)+\frac{(1-\rho)^{2}}{2\rho}\left(f(X)-f(x_{T})\right)\\
 & +\frac{\varpi^{2}_{T}}{2}\left(\norm{X-x_{\star}}^{2}-\norm{x_{T}-x_{\star}}^{2}\right)+\frac{\sigma^{2}}{8}\norm{\dot{X}}^{2}\\
 & +\frac{\varpi_{T}\sigma\rho}{2}\left\langle X-x_{\star},\dot{X}\right\rangle +\frac{\varpi_{T}\sigma^{3}}{4}\left\langle X-x_{T},\dot{X}\right\rangle .
\end{aligned}
\label{eq:lyapunov-expanded}
\end{equation}

For the limit $t\downarrow0$, $\rho(t)\to1$, $\sigma(t)\to0$, $X(t)\to x_{0}$,
and $\dot{X}(t)\to0$, so $\mathcal{V}(t)$ has the limit: 
\begin{equation}
\lim_{t\downarrow0}\mathcal{V}(t)=\frac{\varpi^{2}_{T}}{2}\left(\norm{x_{0}-x_{\star}}^{2}-\norm{x_{T}-x_{\star}}^{2}\right).\label{eq:lyapunov-initial-limit}
\end{equation}

For the limit $t\uparrow T$, denote $\delta=T-t$ and $g_{T}\triangleq\nabla f(x_{T})$.
The leading order terms of the coefficients $\rho$, $\sigma$, and
$\gamma$ as $\delta\downarrow0$ are 
\begin{equation}
\begin{aligned}\rho(T-\delta) & =\frac{\varpi^{2}_{T}}{4}\delta^{2}+o(\delta^{2}),\\
\sigma(T-\delta) & =\frac{2}{\varpi_{T}\delta}+o(\delta^{-1}),\\
\gamma(T-\delta) & =\frac{3}{\delta}+o(\delta^{-1}).
\end{aligned}
\label{eq:continuous-terminal-coefficients}
\end{equation}
Indeed, the relation $\mathrm{cl}\left(\varpi_{T}(T-\delta)/2\right)=\mathrm{sl}(\varpi_{T}\delta/2)$
and $\mathrm{sl}(x)=x+o(x)$ (parts~3 and~5 of
Lemma~\ref{lem:lemniscatic-calculus}) gives the expression for $\rho$; the
other two follow from the definitions of $\sigma$ and $\gamma$.
Moreover, differentiating the identity
$\sigma(t)\,\sigma(T-t)=2$ of Lemma~\ref{lem:sigma-identities} gives
$\dot{\sigma}(t)/\sigma(t)=\dot{\sigma}(T-t)/\sigma(T-t)$, so
$\gamma(t)=\gamma(T-t)$; hence $\gamma(t)=3/t+o(t^{-1})$ as
$t\downarrow0$ as well.

We next obtain the corresponding expansion of the trajectory. Since
$X\in C^{2}([0,T])$, the right-hand side of $\gamma(t)\dot{X}(t)=-\Ddot X(t)-2\nabla f(X(t))$
is bounded as $t\uparrow T$. Because $\gamma(t)\to\infty$, this
forces $\dot{X}(T)=0$. Taylor expansion of $\dot{X}(t)$ gives $\dot{X}(T-\delta)=-\Ddot X(T)\delta+o(\delta)$.
Hence the left-hand side above converges to $-3\Ddot X(T)$, whereas
its right-hand side converges to $-\Ddot X(T)-2g_{T}$. Thus $\Ddot X(T)=g_{T}$,
and further Taylor expansion gives
\begin{equation}
\begin{aligned} & \dot{X}(T-\delta)=-g_{T}\delta+o(\delta),\\
 & X(T-\delta)=x_{T}+\frac{1}{2}g_{T}\delta^{2}+o(\delta^{2}),\\
 & f\left(X(T-\delta)\right)-f(x_{T})=\frac{1}{2}\norm{g_{T}}^{2}\delta^{2}+o(\delta^{2}).
\end{aligned}
\label{eq:continuous-terminal-trajectory}
\end{equation}
Using (\ref{eq:continuous-terminal-coefficients}) and (\ref{eq:continuous-terminal-trajectory}),
the limit of $\mathcal{V}(t)$ in (\ref{eq:lyapunov-expanded}) is
calculated term by term: 
\begin{equation}
\begin{aligned}\lim_{t\uparrow T}\mathcal{V}(t)={} & f(x_{T})-f(x_{\star})+\frac{\norm{g_{T}}^{2}}{\varpi^{2}_{T}}+0+\frac{\norm{g_{T}}^{2}}{2\varpi^{2}_{T}}+0-\frac{\norm{g_{T}}^{2}}{\varpi^{2}_{T}}\\
={} & \frac{1}{2\varpi^{2}_{T}}\norm{\nabla f(x_{T})}^{2}+f(x_{T})-f(x_{\star}).
\end{aligned}
\label{eq:lyapunov-terminal-limit}
\end{equation}

For any $0<s<t<T$, monotonicity gives $\mathcal{V}(s)\geq\mathcal{V}(t)$.
Letting $s\downarrow0$ and then $t\uparrow T$ in this inequality,
(\ref{eq:lyapunov-initial-limit}) and (\ref{eq:lyapunov-terminal-limit})
gives the stronger estimate 
\[
\norm{\nabla f(x_{T})}^{2}+2\varpi^{2}_{T}\left(f(x_{T})-f(x_{\star})\right)+\varpi^{4}_{T}\norm{x_{T}-x_{\star}}^{2}\leq\varpi^{4}_{T}\norm{x_{0}-x_{\star}}^{2}.
\]
Dropping nonnegative terms and substituting $\varpi_{T}=\varpi/T$
proves both stated bounds. \end{proof}

\section{Deferred proofs and details for ITEM-f}
\label{sec:appendix-itemf}

\subsection{Existence and uniqueness of the ITEM-f construction}
\label{sec:app-itemf-existence}
We first derive the target angle and describe the shooting computation of
the ITEM-f coefficients, then prove that it produces the unique construction
of Lemma~\ref{lem:itemf-construction}.
Throughout, $N\geq1$ and $q\in(0,1)$ are fixed, $\Upsilon>1$ denotes a
candidate value of the constant, and we write
\begin{equation*}
 C(\Upsilon)\triangleq(\Upsilon,0),\quad
 R(\Upsilon)\triangleq\sqrt{\Upsilon^{2}-1},\quad
 p(\Upsilon)\triangleq\frac{q\,R(\Upsilon)}{\Upsilon}
 =q\sqrt{1-\Upsilon^{-2}},
\end{equation*}
abbreviated to $C$, $R$, and $p$ when the argument is clear. The map
$\Upsilon\mapsto1-\Upsilon^{-2}$ is continuous and strictly increasing on
$(1,\infty)$; hence so is $p$, with $0<p(\Upsilon)<q$,
$p(\Upsilon)\to0$ as $\Upsilon\downarrow1$, and
$p(\Upsilon)\to q$ as $\Upsilon\to\infty$.

\paragraph{Target angle.}
For a prospective admissible configuration at a candidate $\Upsilon>1$, write
$P_1-C=R(\cos\theta_1,\sin\theta_1)$ with $\theta_1\in(0,\pi)$. The endpoint formula gives
$P_{0}-C
=\left(-R^{2}/\Upsilon,\;\sqrt{1-q}\,R/\Upsilon\right)$, so the
inner-product condition at $k=0$ evaluates on both sides; equating the
two evaluations and simplifying (tedious but straightforward algebra)
gives
\begin{equation}\label{eq:app-itemf-shooting}
\sin\theta_{1}=\sqrt{1-q}\left(\Upsilon+R\cos\theta_{1}\right)
=\sqrt{1-q}\,a_{1}.
\end{equation}
Since $b_1=R\sin\theta_1$, this says that $P_1$ lies on the line
$b=\sqrt{1-q}\,R\,a$ through the origin and $P_0$. Substitution into
$(a-\Upsilon)^{2}+b^{2}=R^{2}$, using $\Upsilon^2-R^2=1$, gives the quadratic
$\left(1+(1-q)R^{2}\right)a^{2}-2\Upsilon a+1=0$, whose roots are
$a_{\pm}=\left(\Upsilon\pm\sqrt{q}\,R\right)/\left(1+(1-q)R^{2}\right)$.
The quadratic is negative at $a_0=1/\Upsilon$, so
$a_-<a_0<a_+$. Hence the ordering constraint selects the point
$(a_+,\sqrt{1-q}\,R(\Upsilon)a_+)$.
Its angle about $C(\Upsilon)$ is the \emph{target angle}
$\vartheta(\Upsilon)\in(0,\pi)$, where
\begin{equation}\label{eq:app-itemf-target}
\cos\vartheta(\Upsilon)
\triangleq\frac{a_{+}-\Upsilon}{R}
=\frac{\sqrt{q}-(1-q)R\,\Upsilon}{1+(1-q)R^{2}}.
\end{equation}
We also define $\vartheta(\Upsilon)$ for any $\Upsilon>1$ by this formula, even if the corresponding construction does not exist.

\paragraph{Shooting computation.}
For each candidate $\Upsilon>1$, define its backward shooting orbit by
\begin{equation}
\begin{aligned}
 &\theta_N(\Upsilon) \triangleq\arccos(-\sqrt q),\\
&  \theta_k(\Upsilon) \triangleq\theta_{k+1}(\Upsilon)
 +\arccos\!\left(1-q-p(\Upsilon)\cos\theta_{k+1}(\Upsilon)\right),
 \quad k=N-1,\dots,1.
\end{aligned}
\label{eq:itemf-angle-recurrence}
\end{equation}
Every admissible construction obeys this recurrence. Indeed, if a candidate
$\Upsilon$ admits the required points, then every intermediate
point has a unique representation
$P_k-C(\Upsilon)=R(\Upsilon)(\cos\theta_k,\sin\theta_k)$ with
$\theta_k\in(0,\pi)$. The ordering of the abscissas gives
$\theta_1>\cdots>\theta_N$, while the inner-product conditions give $\cos(\theta_k-\theta_{k+1})
 =1-q-p(\Upsilon)\cos\theta_{k+1}$ for $1\leq k \leq N-1$. At $k=N$, the same condition gives
$\sin\theta_N=\sqrt{1-q}$, and $a_N<a_{N+1}$ selects
$\cos\theta_N=-\sqrt q$. Therefore admissible constructions satisfy \eqref{eq:itemf-angle-recurrence}. Define the shooting residual by $\Xi_N(\Upsilon)\triangleq
 \theta_1(\Upsilon)-\vartheta(\Upsilon)$ for $\Upsilon>1$. Lemma~\ref{lem:app-itemf-shooting-equivalence} below
shows that the candidate closes exactly when $\Xi_N(\Upsilon)=0$.

The numerical value of the root $\Upsilon_N$ can be computed by bisection. Initialize
$\Upsilon_l=1$ as an unevaluated sentinel, since $\Xi_N$ is defined only on $(1,\infty)$, and set $\Upsilon_r=2$. Double $\Upsilon_r$ until
$\Xi_N(\Upsilon_r)<0$.
Then bisect the interval $[\Upsilon_l,\Upsilon_r]$ by computing the midpoint
$\Upsilon_m=(\Upsilon_l+\Upsilon_r)/2$, replace $\Upsilon_l$ by
$\Upsilon_m$ when $\Xi_N(\Upsilon_m)>0$, and replace $\Upsilon_r$ by
$\Upsilon_m$ when $\Xi_N(\Upsilon_m)<0$. For a prescribed tolerance
$\epsilon>0$, return $\Upsilon_m$ if $\Xi_N(\Upsilon_m)=0$ or
$|\Upsilon_r-\Upsilon_l|<\epsilon$.
Corollary~\ref{cor:app-upsilon-monotone} proves that the doubling terminates and the bracket always contains
$\Upsilon_N$. If the bisection is continued indefinitely, its midpoints converge to $\Upsilon_N$; if it stops with bracket width below $\epsilon$, the returned value is within $\epsilon$ of $\Upsilon_N$.

\begin{lem}[One-step map]\label{lem:app-itemf-one-step}
Fix $q\in(0,1)$ and $p\in(0,q)$, and define $F_{p}(\theta)\triangleq\theta+\arccos\left(1-q-p\cos\theta\right)$. Then:
\begin{enumerate}
\item $F_{p}$ is well defined on $\mathbb{R}$: the argument of the
$\arccos$ lies in $(1-2q,1)\subset(-1,1)$, so
$F_{p}(\theta)-\theta\in(0,\pi)$;
\item $F_{p}$ is strictly increasing in $\theta$;
\item for fixed $\theta$ with $\cos\theta<0$, the map $p\mapsto
F_{p}(\theta)$ is strictly decreasing.
\end{enumerate}
\end{lem}

\begin{proof}
Write $u\triangleq1-q-p\cos\theta$. Claim 1 follows from
$\left|p\cos\theta\right|\leq p<q$ leading to $-1<1-2q<1-q-p\leq u\leq1-q+p<1$.

For claim 2, differentiating in $\theta$ gives $F_{p}'(\theta)=1-(p\sin\theta/\sqrt{1-u^{2}})$, so it suffices to prove $p^{2}\sin^{2}\theta<1-u^{2}$ expanding
$u=1-q-p\cos\theta$, tedious but straightforward algebra gives
$p^{2}\sin^{2}\theta+u^{2}\leq\left(p+1-q\right)^{2}<1$, the last
inequality because $0<p+1-q<1$ by $p<q$. This
proves claim 2. For claim 3, differentiating in $p$ gives for $\cos\theta<0$, $\partial F_p(\theta)/\partial p = \cos\theta/\sqrt{1-u^2} < 0$.
\end{proof}

By Lemma~\ref{lem:app-itemf-one-step}, the orbit in
\eqref{eq:itemf-angle-recurrence} which is written as $\theta_k(\Upsilon)=F_{p(\Upsilon)}\left(\theta_{k+1}(\Upsilon)\right)$ is defined for every
$\Upsilon>1$, each $\theta_{k}(\Upsilon)$ is a
continuous function of $\Upsilon$, and
$\theta_{1}(\Upsilon)>\theta_{2}(\Upsilon)>\cdots>\theta_{N}(\Upsilon)
=\arccos\left(-\sqrt{q}\right)>\pi/2$.

\begin{lem}[Orbit comparison]\label{lem:app-itemf-comparison}
Let $N\geq2$ and $1<\Upsilon<\Upsilon'$. If $\theta_{1}(\Upsilon)\leq\pi$,
then $\theta_{k}(\Upsilon')<\theta_{k}(\Upsilon)$ for $1\leq k\leq N-1$.
\end{lem}

\begin{proof}
The hypothesis gives $\theta_{k}(\Upsilon)\in(\pi/2,\pi]$, hence
$\cos\theta_{k}(\Upsilon)<0$, for every $1\leq k\leq N$. We induct
downward on $k$. For the base case $k=N-1$, since
$\cos\theta_{N}=-\sqrt{q}<0$ and $p(\Upsilon')>p(\Upsilon)$, claim~3 of
Lemma~\ref{lem:app-itemf-one-step} gives $\theta_{N-1}(\Upsilon')=F_{p(\Upsilon')}\left(\theta_{N}\right)
<F_{p(\Upsilon)}\left(\theta_{N}\right)=\theta_{N-1}(\Upsilon)$.
For the step from $k+1$ to $k\leq N-2$, claims 2 and 3 of
Lemma~\ref{lem:app-itemf-one-step} give $\theta_{k}(\Upsilon')
=F_{p(\Upsilon')}\left(\theta_{k+1}(\Upsilon')\right)
<F_{p(\Upsilon')}\left(\theta_{k+1}(\Upsilon)\right)
<F_{p(\Upsilon)}\left(\theta_{k+1}(\Upsilon)\right)
=\theta_{k}(\Upsilon)$, where the first inequality uses the induction hypothesis and the second
uses $\cos\theta_{k+1}(\Upsilon)<0$.
\end{proof}

\begin{lem}[Target angle]\label{lem:app-itemf-target-angle}
The target angle \eqref{eq:app-itemf-target} is well defined,
continuous, and strictly increasing on $(1,\infty)$, with
$\vartheta(\Upsilon)\in(0,\pi)$ for every $\Upsilon>1$,
$\vartheta(\Upsilon)\to\arccos\left(\sqrt{q}\right)$ as
$\Upsilon\downarrow1$, and $\vartheta(\Upsilon)\to\pi$ as
$\Upsilon\to\infty$.
\end{lem}

\begin{proof}
Write $c(\Upsilon)$ for the fraction in \eqref{eq:app-itemf-target}, and
note that its denominator is $1+(1-q)R^{2}=q+(1-q)\Upsilon^{2}$. Then
$c(\Upsilon)<1$, since the numerator is at most
$\sqrt{q}<1<q+(1-q)\Upsilon^{2}$. Moreover, using
$\Upsilon\left(\Upsilon-R\right)=\Upsilon/(\Upsilon+R)$ (from
$\Upsilon^{2}-R^{2}=1$), tedious but straightforward algebra writes
$c(\Upsilon)+1$ as a ratio of positive quantities, so $c(\Upsilon)>-1$:
hence $\vartheta(\Upsilon)=\arccos c(\Upsilon)$ is well defined in
$(0,\pi)$, and it is continuous. Differentiating the quotient and simplifying (tedious
but straightforward algebra) gives, for $\Upsilon>1$,
\[
c'(\Upsilon)
=-\frac{(1-q)\left((1+q)\Upsilon^{2}-q+2\sqrt{q}\,\Upsilon R\right)}{R\left(q+(1-q)\Upsilon^{2}\right)^{2}},
\]
whose bracketed factor equals
$(1+q)\left(\Upsilon^{2}-1\right)+2\sqrt{q}\,\Upsilon R+1>1$. Every
remaining factor is positive, so $c'(\Upsilon)<0$ and $c$ is strictly
decreasing on $(1,\infty)$, and composing with the strictly decreasing
$\arccos$ shows that $\vartheta$ is strictly increasing. The limits
follow by substitution: as $\Upsilon\downarrow1$, $R\to0$ and
$c\to\sqrt{q}$; as $\Upsilon\to\infty$, dividing the numerator and
denominator of $c$ by $(1-q)\Upsilon^{2}$ gives $c\to-1$.
\end{proof}

For $\Upsilon>1$, call a collection of points an \emph{admissible
configuration at $\Upsilon$} if it satisfies the ordering, positivity,
and conditions of Lemma~\ref{lem:itemf-construction} with
$\Upsilon_{N}$ replaced by $\Upsilon$ (and hence with
$C=C(\Upsilon)$ and $R=R(\Upsilon)$).

\begin{lem}[Shooting equivalence]
\label{lem:app-itemf-shooting-equivalence}
An admissible configuration at $\Upsilon>1$ exists if and only if
$\Xi_{N}(\Upsilon)=0$. When it exists, it is unique.
\end{lem}

\begin{proof}
Suppose an admissible configuration exists at $\Upsilon$. By the
angle parametrization, its angles satisfy
$\theta_{N}=\arccos\left(-\sqrt{q}\right)$ and the backward recurrence
\eqref{eq:itemf-angle-recurrence}, so they coincide with the orbit
$\theta_{k}(\Upsilon)$; in particular $\theta_{1}(\Upsilon)<\pi$.
Moreover, the shooting equation \eqref{eq:app-itemf-shooting} together
with $a_{1}>a_{0}$ places $P_{1}$ at the farther intersection, so
$\cos\theta_{1}(\Upsilon)=\cos\vartheta(\Upsilon)$; both angles lie in
$(0,\pi)$, hence $\theta_{1}(\Upsilon)=\vartheta(\Upsilon)$. Every
coordinate is forced ($P_{0}$ and $P_{N+1}$ by their formulas, $P_{k}$ by
$\theta_{k}(\Upsilon)$), which is the uniqueness claim.

Conversely, suppose $\theta_{1}(\Upsilon)=\vartheta(\Upsilon)$. Define
$P_{k}\triangleq C+R\left(\cos\theta_{k}(\Upsilon),
\sin\theta_{k}(\Upsilon)\right)$ for $1\leq k\leq N$ and the endpoints by
their formulas. Then all $\theta_{k}(\Upsilon)$ lie in
$\left(0,\pi\right)$ because
$0<\theta_{N}\leq\theta_{k}(\Upsilon)\leq\theta_{1}(\Upsilon)
=\vartheta(\Upsilon)<\pi$, so $b_{k}=R\sin\theta_{k}(\Upsilon)>0$; the
angles decrease strictly in $k$, so $a_{1}<\cdots<a_{N}$; and
$a_{N}=\Upsilon-\sqrt{q}\,R<\Upsilon=a_{N+1}$. The norm condition holds
by construction. For the inner-product condition at $1\leq k\leq N-1$,
taking cosines in \eqref{eq:itemf-angle-recurrence} gives
$\cos\left(\theta_{k}-\theta_{k+1}\right)
=1-q-p\cos\theta_{k+1}
=1-q\,a_{k+1}/\Upsilon$, and
$(P_{k}-C)\cdot(P_{k+1}-C)
=R^{2}\cos\left(\theta_{k}-\theta_{k+1}\right)$ since both points lie on
the circle; at $k=N$ the condition holds because
$\sin\theta_{N}=\sqrt{1-q}$. At $k=0$: the point
$\left(a_{+},\sqrt{1-q}\,R\,a_{+}\right)$ lies on the circle with
positive ordinate, and $P_{1}$ is the point of the circle with the same
abscissa $a_{1}=\Upsilon+R\cos\vartheta(\Upsilon)=a_{+}$ and positive
ordinate, so the two coincide; hence
$\sin\theta_{1}(\Upsilon)=\sqrt{1-q}\,a_{1}$, which is
\eqref{eq:app-itemf-shooting} and, reversing the algebra that produced
it, the inner-product condition at $k=0$. Finally
$a_{0}=1/\Upsilon<a_{+}=a_{1}$.
\end{proof}

\begin{proof}[Proof of Lemma~\ref{lem:itemf-construction}]

\emph{Existence.}
We first show $\Xi_{N}(\Upsilon)<0$ for all large $\Upsilon$. Consider the
boundary parameter $p=q$: the map
$F_{q}(\theta)=\theta+\arccos\left(1-q(1+\cos\theta)\right)$ is still
well defined, since $1-q(1+\cos\theta)\in[1-2q,1]$. For
$\theta\in(0,\pi)$,
\begin{equation*}
\begin{aligned}
F_{q}(\theta)<\pi
&\Leftrightarrow \arccos\left(1-q(1+\cos\theta)\right)<\pi-\theta\\
&\Leftrightarrow 1-q(1+\cos\theta)>-\cos\theta\\
&\Leftrightarrow (1-q)(1+\cos\theta)>0,
\end{aligned}
\end{equation*}
which holds; so the orbit
$\overline{\theta}_{N}\triangleq\theta_{N}$,
$\overline{\theta}_{k}\triangleq F_{q}(\overline{\theta}_{k+1})$
satisfies $\overline{\theta}_{1}<\pi$ by induction. Each backward step
$(p,\theta)\mapsto F_{p}(\theta)$ is jointly continuous for
$p\in[0,q]$ because the $\arccos$ argument stays in $[1-2q,1]$; at the
endpoint $p=0$ it reads $F_{0}(\theta)=\theta+\arccos\left(1-q\right)$.
Hence
$\theta_{1}(\Upsilon)$, which depends on $\Upsilon$ only through
$p(\Upsilon)$, converges to $\overline{\theta}_{1}$ as
$p(\Upsilon)\to q$, that is, as $\Upsilon\to\infty$. Combining with
Lemma~\ref{lem:app-itemf-target-angle}, we have
$\lim_{\Upsilon\to\infty}\Xi_{N}(\Upsilon)=\overline{\theta}_{1}-\pi<0$, so $\Xi_{N}(b)<0$ for some $b>1$.

Next we produce a point where $\Xi_{N}>0$. As $\Upsilon\downarrow1$ we have
$p(\Upsilon)\to0$, so by the same joint continuity, now at the endpoint
$p=0$,
$\theta_{1}(\Upsilon)\to\theta_{N}+(N-1)\arccos\left(1-q\right)$, while
$\vartheta(\Upsilon)\to\arccos\left(\sqrt{q}\right)$ by
Lemma~\ref{lem:app-itemf-target-angle}. Therefore
\begin{equation*}
\lim_{\Upsilon\downarrow1}\Xi_{N}(\Upsilon)
=\theta_{N}+(N-1)\arccos\left(1-q\right)
-\arccos\left(\sqrt{q}\right)
\geq\pi-2\arccos\left(\sqrt{q}\right)
=\arccos\left(1-2q\right)>0,
\end{equation*}
where the middle inequality drops the nonnegative term
$(N-1)\arccos\left(1-q\right)$ and uses
$\theta_{N}=\arccos\left(-\sqrt{q}\right)=\pi-\arccos\left(\sqrt{q}\right)$,
and the last equality holds because
$\cos\left(\pi-2\arccos\sqrt{q}\right)=-\left(2q-1\right)=1-2q$ and
$\pi-2\arccos\sqrt{q}\in(0,\pi)$. By continuity, $\Xi_{N}(a)>0$ for some
$a>1$. Hence $\Xi_{N}(a)>0>\Xi_{N}(b)$ with $a<b$ (enlarging $b$ if necessary),
so the intermediate value theorem yields $\Upsilon_{N}\in(a,b)$ with
$\Xi_{N}\left(\Upsilon_{N}\right)=0$. By
Lemma~\ref{lem:app-itemf-shooting-equivalence}, an admissible
configuration exists at $\Upsilon_{N}$.

\emph{Uniqueness.}
Suppose $\Upsilon<\Upsilon'$ both admit admissible points, so that
$\theta_{1}(\Upsilon)=\vartheta(\Upsilon)<\pi$ and
$\theta_{1}(\Upsilon')=\vartheta(\Upsilon')$ by
Lemma~\ref{lem:app-itemf-shooting-equivalence}. For $N=1$ the
orbit is the constant $\arccos\left(-\sqrt{q}\right)$, and the strict
monotonicity of $\vartheta$ makes the two equalities incompatible. For
$N\geq2$, Lemma~\ref{lem:app-itemf-comparison} applies at $\Upsilon$
(its hypothesis $\theta_{1}(\Upsilon)\leq\pi$ holds) and gives, with
Lemma~\ref{lem:app-itemf-target-angle}, $\theta_{1}(\Upsilon')<\theta_{1}(\Upsilon)
=\vartheta(\Upsilon)<\vartheta(\Upsilon')=\theta_{1}(\Upsilon')$, a contradiction. Hence the admissible $\Upsilon$ is unique, and
Lemma~\ref{lem:app-itemf-shooting-equivalence} shows that the points are
uniquely determined by it.
\end{proof}

The same objects also give a sign characterization of the residual that
makes the numerical evaluation of $\Upsilon_{N}$ straightforward.

\begin{cor}[Bisection correctness]
\label{cor:app-upsilon-monotone}
Let $\Upsilon_{N}$ be as in Lemma~\ref{lem:itemf-construction}. Then the
residual $\Xi_{N}$ defined above satisfies
$\Xi_{N}(\Upsilon)>0$ for $1<\Upsilon<\Upsilon_{N}$ and
$\Xi_{N}(\Upsilon)<0$ for $\Upsilon>\Upsilon_{N}$. Consequently both
update rules of the bisection described above are correct, the
doubling of the right endpoint terminates and produces a bracket
containing $\Upsilon_{N}$, the invariant
$\Upsilon_{l}<\Upsilon_{N}<\Upsilon_{r}$ is maintained, and, if the
bisection is continued indefinitely, the midpoints converge to
$\Upsilon_{N}$. For any $\epsilon>0$, if it stops with bracket width below $\epsilon$, the
returned value is within $\epsilon$ of $\Upsilon_N$.
\end{cor}

\begin{proof}
By Lemma~\ref{lem:app-itemf-shooting-equivalence},
$\theta_{1}\left(\Upsilon_{N}\right)=\vartheta\left(\Upsilon_{N}\right)<\pi$,
so $\Xi_{N}\left(\Upsilon_{N}\right)=0$. For $\Upsilon>\Upsilon_{N}$,
Lemma~\ref{lem:app-itemf-comparison} applies at $\Upsilon_{N}$ (its
hypothesis $\theta_{1}\left(\Upsilon_{N}\right)\leq\pi$ holds) and
gives, with the strictly increasing target angle of
Lemma~\ref{lem:app-itemf-target-angle}, $\Xi_{N}(\Upsilon)=\theta_{1}(\Upsilon)-\vartheta(\Upsilon)<\theta_{1}\left(\Upsilon_{N}\right)-\vartheta\left(\Upsilon_{N}\right)=0.$ For $1<\Upsilon<\Upsilon_{N}$, either
$\theta_{1}(\Upsilon)\geq\pi>\vartheta(\Upsilon)$ and
$\Xi_{N}(\Upsilon)>0$ directly, or $\theta_{1}(\Upsilon)<\pi$ and
Lemma~\ref{lem:app-itemf-comparison}, now applied at $\Upsilon$, gives in
the same way $\Xi_{N}(\Upsilon)>\Xi_{N}\left(\Upsilon_{N}\right)=0$. (For
$N=1$ the orbit is the constant $\arccos\left(-\sqrt{q}\right)$ and both
inequalities hold through the target angle alone.)
For the bisection, the doubling terminates, since $\Xi_{N}<0$ once
$\Upsilon_{r}$ exceeds $\Upsilon_{N}$, producing a bracket
$\Upsilon_{l}=1<\Upsilon_{N}<\Upsilon_{r}$; by the sign characterization
each update rule fires exactly on its side of $\Upsilon_{N}$, and the
bisection bookkeeping is exactly parallel to the argument of
Appendix~\ref{sec:app-recurrence-existence}, so the invariant
$\Upsilon_{l}<\Upsilon_{N}<\Upsilon_{r}$ is preserved. If the bisection
is continued indefinitely, its bracket width tends to zero, so the
midpoints converge to $\Upsilon_{N}$. At finite termination with
$|\Upsilon_r-\Upsilon_l|<\epsilon$, the bracket invariant puts the
returned value within $\epsilon$ of $\Upsilon_N$.
\end{proof}

\begin{lem}[Symmetry of the construction]
\label{lem:itemf-involution-symmetry}
For $0\leq k\leq N+1$,
\begin{equation}\label{eq:itemf-coord-symmetry}
\begin{split}
 a_{k}a_{N+1-k}&=1,\\
 a_{k}b_{N+1-k}&=b_{k}.
\end{split}
\end{equation}
\end{lem}

\begin{proof}
The circle equation is equivalent to
$a^{2}+b^{2}-2\Upsilon_{N}a+1=0$, and the map, defined for $a>0$,
$T(a,b)=(1/a,b/a)$, which satisfies $T\circ T=\mathrm{id}$, maps this circle to itself, with
$T(P_{0})=P_{N+1}$. Moreover, if $P=(a,b)$ and $P'=(a',b')$ satisfy
\begin{equation*}
 (P-C)\cdot(P'-C)=R_{N}^{2}\left(1-\frac{qa'}{\Upsilon_{N}}\right),
\end{equation*}
then direct expansion gives
\begin{equation*}
 (T(P')-C)\cdot(T(P)-C)
 =\frac{(P-C)\cdot(P'-C)+R_{N}^{2}(aa'-1)}{aa'}
 =R_{N}^{2}\left(1-\frac{q}{\Upsilon_{N}a}\right).
\end{equation*}
This is the recurrence for the reversed pair, since the first coordinate
of $T(P)$ is $1/a$. Hence
$\widetilde P_k\triangleq T(P_{N+1-k})$ satisfies the same endpoint
conditions and recurrence as $P_k$. Its first coordinates are strictly
increasing and its second coordinates are positive, so uniqueness in
Lemma~\ref{lem:itemf-construction} gives
$T(P_k)=P_{N+1-k}$. Comparing coordinates proves
\eqref{eq:itemf-coord-symmetry}.
\end{proof}

The construction also gives the explicit relaxation used in
Theorem~\ref{thm:itemf-convergence}.

\begin{lem}
\label{lem:itemf-explicit-rate}
Let $\Upsilon_N$ be as in Lemma~\ref{lem:itemf-construction}. Then, for
every $N\geq1$, we have $1/\Upsilon_N^2
 <4 (1-\sqrt q)^{2N}$.
\end{lem}

\begin{proof}
For the construction angles $\theta_k$, define $t_k\triangleq\cot\left(\theta_k/2\right)$ for $1\leq k\leq N$. The half-angle formula gives
\begin{equation*}
 t_k=\frac{\sin\theta_k}{1-\cos\theta_k}
 =\frac{b_k}{\Upsilon_N+R_N-a_k}.
\end{equation*}

By the symmetry
of Lemma~\ref{lem:itemf-involution-symmetry}, $a_1a_N=1$ and
$b_1=a_1b_N$; substituting these together with
$a_N=\Upsilon_N-\sqrt q\,R_N$ and $b_N=\sqrt{1-q}\,R_N$ into the
half-angle formula, tedious but straightforward algebra gives $t_{1}=\left(\sqrt{1-q}\right)/\left((1-\sqrt{q})(\Upsilon_{N}+R_{N})\right)$
and $t_{N}=\left(\sqrt{1-q}\right)/\left(1+\sqrt{q}\right)$. For $1\leq k\leq N-1$, the inner-product condition,
$\cos\theta_{k+1}<0$, and $R_N<\Upsilon_N$ give
\begin{align*}
 \cos(\theta_k-\theta_{k+1})
 &=1-q\left(1+\frac{R_N}{\Upsilon_N}\cos\theta_{k+1}\right)\\
 &<1-q(1+\cos\theta_{k+1})
 =1-2q\cos^2\left(\frac{\theta_{k+1}}2\right),
\end{align*}
and hence, taking half-angle sines, expanding
$\sin((\theta_k-\theta_{k+1})/2)$, and dividing by
$\sin(\theta_k/2)\sin(\theta_{k+1}/2)$ (tedious but straightforward
algebra), $t_k<(1-\sqrt q)t_{k+1}$
for $1\leq k\leq N-1$, which gives $t_1\leq(1-\sqrt q)^{N-1}t_N$.

Combining the contraction with these boundary values,
\begin{equation*}
 \frac{1}{2\Upsilon_N}
 <\frac{1}{\Upsilon_N+R_N}
 =t_1t_N
 \leq t_N^2(1-\sqrt q)^{N-1}
 =\frac{(1-\sqrt q)^N}{1+\sqrt q}.
\end{equation*}
It follows that
$1/\Upsilon_N^2<\left(2/(1+\sqrt q)\right)^2(1-\sqrt q)^{2N}<4(1-\sqrt q)^{2N}$,
which completes the proof.
\end{proof}

\subsection{Norm identities for the ITEM-f Lyapunov sequence}
\label{sec:app-itemf-lyap-norms}
This appendix proves Lemma~\ref{lem:itemf-lyap-norms}
and shows how the update equation of \eqref{eq:item-f-momentum} enters
the analysis.
\begin{proof}[Proof of Lemma~\ref{lem:itemf-lyap-norms}]
We use the coordinate relations
\begin{equation}
\begin{aligned} & c_{k}a_{k}+s_{k}b_{k}=(1-q)a_{k+1},\quad1\leq k\leq N,\\
 & a_{k}=(c_{k}+q)a_{k+1}-s_{k}b_{k+1},\quad1\leq k\leq N-1.
\end{aligned}
\label{eq:itemf-lyap-coordinate-relations}
\end{equation}
For $1\leq k\leq N-1$, these are the coordinate relations for the rotation
from $P_k-C$ to $P_{k+1}-C$ and its inverse. At $k=N$, the first
identity follows from $c_N=1-q$, $s_N=\sqrt{q(1-q)}$,
$a_{N+1}=\Upsilon_N$, and the boundary facts $b_N=\sqrt{1-q}\,R_N$ and
$\Upsilon_N-a_N=\sqrt q\,R_N$ of the construction
(Appendix~\ref{sec:app-itemf-existence}).

\noindent\textbf{First identity.}
Let $g_0\triangleq\nabla f(x_0)$.  Since $x_{-1}^+=x_0$, the first
update~\eqref{eq:item-f-momentum} and the definition of $\phi_N$ with $a_{1}a_{N}=1$, $a_{0}=1/\Upsilon_N$ give
\begin{align*}
	x_{1}-x_{\star} & =x^{+}_{0}-x_{\star}-\frac{\phi_{N-1}}{\phi_{N}}\left(1+\frac{a_{0}}{(1-q)a_{1}}\right)\frac{g_{0}}{L}=x^{+}_{0}-x_{\star}-\frac{\phi_{N-1}}{\Upsilon_{N}\sqrt{1-q}}\frac{g_{0}}{L}.
\end{align*}
The definition of $s_1$ gives
\begin{equation*}
 \frac{(1-q)a_2}{s_1}
 \frac{\phi_{N-1}}{\Upsilon_N\sqrt{1-q}}
 =\sqrt{\frac{1-q}{q}}.
\end{equation*}
Substituting this update into the definition \eqref{eq:itemf-lyap-W-def} of $W_1$ and using the first
coordinate relation \eqref{eq:itemf-lyap-coordinate-relations}, we have
\begin{equation*}
 W_1
 =\begin{bmatrix}
  a_1(x_0^+-x_\star)\\[1mm]
  b_1(x_0^+-x_\star)
  -\sqrt{\dfrac{1-q}{q}}\dfrac{g_0}{L}
 \end{bmatrix}.
\end{equation*}
The boundary relations in the construction give
$b_1=\sqrt{1-q}\,R_Na_1$ (the shooting relation
\eqref{eq:app-itemf-shooting}) and $\Upsilon_Na_1-1=\sqrt q\,R_Na_1$
(equivalent to $a_{1}a_{N}=1$ from Lemma~\ref{lem:itemf-involution-symmetry}).  Together with the circle equation
for $P_1$, these imply
\begin{align*}
 & b_1=\sqrt{\frac{1-q}{q}}(\Upsilon_Na_1-1),\\
& a_1^2+b_1^2=2\Upsilon_Na_1-1.
\end{align*}
Expanding the squared norm and using $\mu=qL$, $\Upsilon_Na_0=1$, and
$(1-q)g_0=\widetilde g_0+\mu(x_0^+-x_\star)$, we obtain
\begin{equation*}
\begin{aligned}
 \norm{W_1}^{2}
 ={}&(2\Upsilon_N a_1-1)\norm{x_0^+-x_\star}^{2}
 +\frac{1-q}{qL^2}\norm{g_0}^{2}
 -\frac{2(\Upsilon_Na_1-1)}{\mu}
  \left\langle x_0^+-x_\star,(1-q)g_0\right\rangle\\
 ={}&\norm{x_0^+-x_\star}^{2}+\frac{1-q}{qL^2}\norm{g_0}^{2}
 -\frac{2\Upsilon_N}{\mu}(a_1-a_0)\left\langle x_0^+-x_\star,\widetilde g_0\right\rangle.
\end{aligned}
\end{equation*}
Finally, direct expansion of
$x_0^+-x_\star=x_0-x_\star-L^{-1}g_0$ and
$\widetilde g_0=g_0-qL(x_0-x_\star)$ gives
\begin{equation*}
 \norm{x_0^+-x_\star}^{2}
 +\frac{1-q}{qL^2}\norm{g_0}^{2}
 =(1-q)\norm{x_0-x_\star}^{2}
 +\frac{1}{qL^2}\norm{\widetilde g_0}^{2}.
\end{equation*}
Since $1/(qL^2)=\Upsilon_Na_0/(\mu L)$, substitution proves the first
identity.

\medskip
\noindent\textbf{Second identity.}
For $1\leq k\leq N-1$, the definition $s_k=\sqrt q\,a_{k+1}\phi_{N-k}/\Upsilon_N$ gives $\phi_{N-k-1}/\phi_{N-k}=(a_{k+1}s_{k+1})/(a_{k+2}s_{k})$.
Thus the update of~\eqref{eq:item-f-momentum} becomes
\begin{equation*}
 x_{k+1}=x_k^+
 +\frac{a_ks_{k+1}}{(1-q)a_{k+2}s_k}
 (x_k^+-x_{k-1}^+)
 +\frac{a_{k+1}s_{k+1}}{a_{k+2}s_k}(x_k^+-x_k).
\end{equation*}
Substituting this into the definition \eqref{eq:itemf-lyap-W-def} of $W_{k+1}$ and using the first
coordinate relation \eqref{eq:itemf-lyap-coordinate-relations} gives
\begin{equation*}
 W_{k+1}=
 \begin{bmatrix}
  a_{k+1}(x_k^+-x_\star)\\[1mm]
  b_{k+1}(x_k^+-x_\star)
  +\dfrac{a_k}{s_k}(x_k^+-x_{k-1}^+)
  +\dfrac{(1-q)a_{k+1}}{s_k}(x_k^+-x_k)
 \end{bmatrix}.
\end{equation*}

Subtracting $\mathcal O_kW_k$ of \eqref{eq:itemf-lyap-W-rotation}, the first
component of $W_{k+1}-\mathcal O_kW_k$ is:
\begin{equation*}
 a_{k+1}(x_k^+-x_\star)
 -(1-q)a_{k+1}(x_k-x_\star)
 =-\frac{a_{k+1}}{L}\widetilde g_k.
\end{equation*}
The second component is:
\begin{align*}
 &b_{k+1}(x_k^+-x_\star)
 +\frac{a_k}{s_k}(x_k^+-x_{k-1}^+)
 +\frac{(1-q)a_{k+1}}{s_k}(x_k^+-x_k)\\
 &\quad-\frac{(1-q)c_ka_{k+1}(x_k-x_\star)
 -a_k(x_{k-1}^+-x_\star)}{s_k}\\
 &=\left(b_{k+1}+\frac{a_k}{s_k}+\frac{(1-q)a_{k+1}}{s_k}\right)(x_k^+-x_\star)\\
 &\quad-\frac{(1-q)a_{k+1}(1+c_k)}{s_k}(x_k-x_\star)\\
 &=\frac{a_{k+1}(1+c_k)}{s_k}
 \left(x_k^+-x_\star-(1-q)(x_k-x_\star)\right)\\
 &=-\frac{\Upsilon_Ns_k}{\mu}\widetilde g_k,
\end{align*}
where the second equality uses the second coordinate relation \eqref{eq:itemf-lyap-coordinate-relations} and the last equality uses
$a_{k+1}(1+c_k)/s_k=a_{k+1}s_k/(1-c_k)=\Upsilon_N s_k/q$
and
$x_k^+-x_\star=(1-q)(x_k-x_\star)-L^{-1}\widetilde g_k$.
Since $1-c_k=qa_{k+1}/\Upsilon_N$, the two components give
\begin{align*}
	W_{k+1}-\mathcal{O}_{k}W_{k} & =-\frac{\Upsilon_{N}}{\mu}\begin{bmatrix}(1-c_{k})\widetilde{g}_{k}\\
		s_{k}\widetilde{g}_{k}
	\end{bmatrix}=-\frac{\Upsilon_{N}}{\mu}\mathcal{O}_{k}\begin{bmatrix}(c_{k}-1)\widetilde{g}_{k}\\
		s_{k}\widetilde{g}_{k}
	\end{bmatrix},
\end{align*}
where the last equality uses $c_k^2+s_k^2=1$. Rearrangement gives
\begin{equation*}
 W_{k+1}=\mathcal O_k\left(W_k-\frac{\Upsilon_N}{\mu}
 \begin{bmatrix}
  (c_k-1)\widetilde g_k\\s_k\widetilde g_k
 \end{bmatrix}\right).
\end{equation*}
Taking norms and using the orthogonality of $\mathcal O_k$ proves the
second identity.

\medskip
\noindent\textbf{Third identity.}
At $k=N$, the definition \eqref{eq:itemf-lyap-W-def} of $W_N$ and
$c_N=1-q$, $s_N=\sqrt{q(1-q)}$, and $a_{N+1}=\Upsilon_N$ give
\begin{equation*}
 W_N=
 \begin{bmatrix}
  a_N(x_{N-1}^+-x_\star)\\[1mm]
  \sqrt{\dfrac{1-q}{q}}
  \bigl(\Upsilon_N(x_N-x_\star)-a_N(x_{N-1}^+-x_\star)\bigr)
 \end{bmatrix}.
\end{equation*}
Expanding and regrouping its squared norm yields
\begin{align*}
 \norm{W_N}^2
 ={}&a_N^2\norm{x_{N-1}^+-x_\star}^2
 +\frac{1-q}{q}
  \norm{\Upsilon_N(x_N-x_\star)-a_N(x_{N-1}^+-x_\star)}^2\\
 ={}&(1-q)\Upsilon_N^2\norm{x_N-x_\star}^2
 +\frac1q\norm{(1-q)\Upsilon_N(x_N-x_\star)
 -a_N(x_{N-1}^+-x_\star)}^2,
\end{align*}
which is the third identity.
\end{proof}

\subsection{Computer-assisted algorithm design by PEP}
\label{sec:itemf-convergence-proof}\label{sec:itemf-pep-design}
\paragraph{PEP formulation of the problem.}

This subsection presents the Stage-1 derivation that
\textsf{bnb-pep-skill} wrote for the ITEM-f design problem, again
lightly edited for presentation. The derivation follows the PEP
methodology of Section~\ref{subsec:autoopt-stage-numerical}
\cite{drori2014performance,taylor2017smooth,dasgupta2022BnBPEP},
developed in detail for lemniscate acceleration in
Appendix~\ref{sec:lemniscate-pep-design}, whose notation
(Section~\ref{subsec:notation}) it reuses. Below we spell out the
case-specific ingredients, namely the function class, the function-value
performance measure with its normalized initial condition, and the
reduction to a transformed function, and refer to
Appendix~\ref{sec:lemniscate-pep-design} for the steps that carry over
unchanged.

The design problem is instantiated with the function class
$\mathcal{F}=\mathcal{F}_{\mu,L}$ and the method class $\mathcal{M}_{N}$ of
$N$-step FSFOMs. Applying a method in $\mathcal M_N$ to
$f\in\mathcal{F}_{\mu,L}$ from a starting point $x_{0}$, with the
stepsizes normalized by $L$ as in
Appendix~\ref{sec:lemniscate-pep-design}, produces the iterates
\begin{equation}\label{eq:itemf-fsfom}
x_{i}=x_{i-1}-\frac{1}{L}\sum_{j=0}^{i-1}h_{i,j}\nabla f(x_{j}),
\quad 1\leq i\leq N,
\end{equation}
with stepsizes $\{h_{i,j}\}_{0\leq j<i\leq N}$ fixed in advance, independently of $f$. We rate a method by its worst-case
function-value suboptimality after $N$ steps, relative to the initial suboptimality.
Since $\mathcal{F}_{\mu,L}$ is invariant under translations of the variable and shifts of
the function value, we set $x_{\star}=0$ and $f_{\star}=0$. Optimality conditions give $\nabla f(x_{\star})=0$, and we
normalize the initial gap by $f(x_{0})-f_{\star}\leq1$. The worst-case performance of
$M\in\mathcal{M}_{N}$ is
\begin{align}
\mathcal{R}(M)=&\left(\begin{array}{ll}
\textrm{maximize} & f(x_{N})-f_{\star}\\
\textrm{subject to} & f\in\mathcal{F}_{\mu,L},\\
& \{x_{i}\}_{1\leq i\leq N}\textrm{ generated from }x_{0}\textrm{ by }M,\\
& f(x_{0})-f_{\star}\leq1,\\
& x_{\star}=0,\ f_{\star}=0,\ \nabla f(x_{\star})=0
\end{array}\right),\tag{\ensuremath{\mathcal{O}^{\mathrm{in}}_{f}}}\label{eq:itemf-inner}
\end{align}
where the decision variables are the function $f$ and the iterates
$x_{0},\dots,x_{N}$. The optimal FSFOM $M^{\star}_{N}\in\mathcal{M}_{N}$ for
this setup solves the outer design problem
$\mathcal{R}^{\star}(\mathcal{M}_{N})=\min_{M\in\mathcal{M}_{N}}\mathcal{R}(M)$
of \eqref{eq:autoopt-outer-design}.

Following~\cite{taylor2021optimal}, the derivation passes to a transformed function. By scale invariance
\cite[$\mathsection$3.5]{taylor2017smooth} fix $L=1$, so
$\mu=q$. For $f\in\mathcal{F}_{q,1}$ with $x_{\star}=0$, $f_{\star}=0$, the shifted function
$\tilde f\triangleq f-(q/2)\norm{\cdot}^{2}$ belongs to $\mathcal{F}_{0,1-q}$, with gradients
$\tilde g_{i}=\nabla f(x_{i})-q\,x_{i}$ and $\tilde g_{\star}=0$; the original
suboptimality is recovered from the transformed values $\tilde f_{i}=\tilde f(x_{i})$
through
\begin{equation}\label{eq:itemf-orig-vs-transf}
f(x_{i})-f_{\star}=\tilde f_{i}+\tfrac q2\norm{x_{i}}^{2}.
\end{equation}
Re-expressing \eqref{eq:itemf-fsfom} through $\tilde g_{i}$, the iterates take the
form
\begin{equation}\label{eq:itemf-alpha-iterate}
x_{i}=\Big(1-q\sum_{j=0}^{i-1}\alpha_{i,j}\Big)x_{0}-\sum_{j=0}^{i-1}\alpha_{i,j}\,\tilde g_{j},
\quad 1\leq i\leq N,
\end{equation}
where the transformed stepsizes $\{\alpha_{i,j}\}_{0\leq j<i\leq N}$ are defined from the original stepsizes
$\{h_{i,j}\}_{0\leq j<i\leq N}$ through the triangular (hence invertible) system of
equations \cite[$\mathsection$3.2]{taylor2021optimal} $\alpha_{i,i-1}=h_{i,i-1}$ and $\alpha_{i,j}=h_{i,j}+\alpha_{i-1,j}-q{\displaystyle \sum^{i-1}_{k=j+1}h_{i,k}\alpha_{k,j}}$,
for $0\leq j\leq i-2$. The transformed function obeys
\eqref{eq:F_0L_formula}. We impose these smooth-convex interpolation
inequalities at the sampled iterates, with $1-q$ in place of $L$.

Collect the data into $P=[\,x_{0}\mid\tilde g_{0}\mid\cdots\mid\tilde g_{N}\,]\in\mathbb{R}^{d\times(N+2)}$, $G=P^{\top}P\in\mathbb{S}^{N+2}_{+}$, and
$F=[\,\tilde f_{0}\mid\cdots\mid\tilde f_{N}\,]$. With the selectors $\mathbf{x}_{0}=e_{1}$, $\mathbf{g}_{i}=e_{i+2}$ (so $\tilde g_{i}=P\mathbf{g}_{i}$), $\mathbf{f}_{i}=e_{i+1}$, $\mathbf{x}_{\star}=\mathbf{g}_{\star}=0$, and $\mathbf{f}_{\star}=0$,
define the iterate selectors by mirroring the coefficients of
\eqref{eq:itemf-alpha-iterate},
\begin{equation}\label{eq:itemf-x-selector}
\mathbf{x}_{i}\triangleq\Big(1-q\sum_{j=0}^{i-1}\alpha_{i,j}\Big)\mathbf{x}_{0}-\sum_{j=0}^{i-1}\alpha_{i,j}\,\mathbf{g}_{j},
\quad 1\leq i\leq N,
\end{equation}
so that $x_{i}=P\mathbf{x}_{i}$ by \eqref{eq:itemf-alpha-iterate}, and define
$A_{i,j}(\alpha)$, $B_{i,j}(\alpha)$, and $C_{i,j}$ for $i,j\in I^{\star}_{N}$
exactly as in \eqref{eq:ABCa-mat-vec}.

By \eqref{eq:itemf-x-selector}
$\mathbf{x}_{i}$ is affine in $\alpha$, so $A_{i,j}$ is affine and $B_{i,j}$ quadratic
in $\alpha$, while $C_{i,j}$ is constant. By \eqref{eq:itemf-orig-vs-transf} the
objective is $f(x_{N})-f_{\star}=F\mathbf{f}_{N}+(q/2)\mathbf{tr}(GB_{N,\star})$
and the initial gap is $f(x_{0})-f_{\star}=F\mathbf{f}_{0}+(q/2)\mathbf{tr}(GB_{0,\star})$. Under the large-scale assumption $d\geq N+2$ (Assumption~\ref{large-scale-assumption}) the rank constraint on $G$
is vacuous, and \eqref{eq:itemf-inner} relaxes to the convex SDP
\begin{align*}
\mathcal{R}(M)=&\left(\begin{array}{l}
\textrm{maximize}\quad F\mathbf{f}_{N}+\tfrac q2\mathbf{tr}GB_{N,\star}\\
\textrm{subject to}\\
F(\mathbf{f}_{j}-\mathbf{f}_{i})+\mathbf{tr}G\big[A_{i,j}(\alpha)+\tfrac{1}{2(1-q)}C_{i,j}\big]\leq0,\ \ i,j\in I^{\star}_{N}:i\neq j,\ \ {\color{annotgray}\rhd\,\textsf{dual var.}\ \lambda_{i,j}\geq0}\\
F\mathbf{f}_{0}+\tfrac q2\mathbf{tr}GB_{0,\star}\leq1,\ \ {\color{annotgray}\rhd\,\textsf{dual var.}\ \nu\geq0}\\
G\succeq0,\ \ {\color{annotgray}\rhd\,\textsf{dual var.}\ Z\succeq0}
\end{array}\right),
\end{align*}
with decision variables $F\in\mathbb{R}^{1\times(N+1)}$ and $G\in\mathbb{S}^{N+2}$, and
the indicated dual variables. Taking the dual gives
\begin{align}
\overline{\mathcal{R}}(M)=&\left(\begin{array}{l}
\textrm{minimize}\quad\nu\\
\textrm{subject to}\\
\mathbf{f}_{N}-\nu\,\mathbf{f}_{0}+\sum_{i,j\in I^{\star}_{N}:i\neq j}\lambda_{i,j}\,(\mathbf{f}_{i}-\mathbf{f}_{j})=0,\\
\tfrac q2\big(\nu B_{0,\star}-B_{N,\star}(\alpha)\big)+\sum_{i,j\in I^{\star}_{N}:i\neq j}\lambda_{i,j}\big[A_{i,j}(\alpha)+\tfrac{1}{2(1-q)}C_{i,j}\big]=Z,\\
Z\succeq0,\quad \nu\geq0,\quad \lambda_{i,j}\geq0\ \ (i,j\in I^{\star}_{N}:i\neq j)
\end{array}\right),\label{eq:itemf-dual}
\end{align}
with decision variables $\nu\in\mathbb{R}$, $\lambda=\{\lambda_{i,j}\}$, and
$Z\in\mathbb{S}^{N+2}$. By weak duality, we have $\mathcal{R}(M)\leq\overline{\mathcal{R}}(M)$.

The derivation then recasts the outer problem, in the
design variable $\alpha$, as a QCQP in exact parallel with
\eqref{eq:BnB-PEP-Preli}: it substitutes \eqref{eq:itemf-dual} for the
inner problem, adjoins $\{\alpha_{i,j}\}_{0\leq j<i\leq N}$ to the
decision variables $\nu$, $\lambda$, and $Z$, and replaces the
constraint $Z\succeq0$ with the Cholesky factorization $Z=VV^{\top}$,
where $V$ is lower triangular with nonnegative diagonal again
\cite[Corollary 7.2.9]{horn2012matrix}. Every constraint of the
resulting problem is at most quadratic in its decision variables; upon
user approval of the derivation, \textsf{bnb-pep-skill} solves the
problem numerically. The reformulation presumes strong duality between
the inner SDP and its dual \eqref{eq:itemf-dual}, the analogue of
Assumption~\ref{strong-duality-assumption}; weak duality alone already
makes its optimal value an upper bound on
$\mathcal{R}^{\star}(\mathcal{M}_{N})$.

In Stage 2 of the pipeline, the
\textsf{frontier-llm-consult} skill fitted symbolic expressions to the
numerical solutions of this QCQP, yielding analytic formulas for the
transformed stepsizes and an analytic feasible point of the inner dual
\eqref{eq:itemf-dual}, expressed in terms of the geometric construction
of Lemma~\ref{lem:itemf-construction} (at $L=1$, hence $\mu=q$); dual
feasibility alone already certifies the contraction bound.
The consultation then recast the method from the
transformed-stepsize representation into the momentum form
\eqref{eq:item-f-momentum} in which Section~\ref{sec:itemf} presents
it. From the
multipliers of this feasible point, continued consultation produced the
Lyapunov proof of Theorem~\ref{thm:itemf-convergence} presented in
Section~\ref{sec:itemf-lyapunov}, stated on the
momentum form, and Section~\ref{sec:itemf-lean}
describes the Lean coverage of the presented results.

\end{document}